\documentclass[a4paper,twoside]{article}  

\usepackage{amssymb,amsmath,amsthm,dsfont,amsfonts,color,latexsym}

\usepackage{hyperref}

\usepackage{mathrsfs} 

\definecolor{blue}{rgb}{0.00,0.00,1.00}
\definecolor{red}{rgb}{1.00,0.00,0.00}

\allowdisplaybreaks[4]
\renewcommand{\baselinestretch}{1.2}
\def\bq{\begin{equation}}
\def\eq{\end{equation}}
\def\ba{\begin{array}{ccc}}
\def\bal{\begin{array}{lll}}
\def\ea{\end{array}}
\def\bsp{\begin{split}}
\def\esp{\end{split}}

\def\({\left(}\def\){\right)}
\def\[{\left[}\def\]{\right]}

    \def \M   {\mathbb{M}}
    \def \O   {\mathbb{O}}

    \def \R   {\mathbb{R}}
    
    \def\C    {\mathbb{C}}
    
    \def\i    {\mathrm{i}}

    \def\S    {\mathbb{S}}
    \def\eps  {\epsilon}
    \def\intr {\int_{\R^3}}
    
    \def\intt {\int^t_0}

    \def \N    {\mathbb{N}}

    \def \pt   {\partial}
    
    \def \Dt   {\frac{\rm d}{{\rm d}t}}

    \def \dxa   {\partial^{\alpha}_x}

    \def\Tdv   {\nabla_v}

       \def\bq{\begin{equation}}
       \def\eq{\end{equation}}
       \def\be{\begin{equation}}
       \def\ee{\end{equation}}
       \def\bma#1\ema{{\allowdisplaybreaks\begin{align}#1\end{align}}}
       \def\bmas#1\emas{{\allowdisplaybreaks\begin{align*}#1\end{align*}}}
       \def\bln#1\eln{{\allowdisplaybreaks\begin{aligned}#1\end{aligned}}}
       \def\nnm{\notag}
       \def\bgr#1\egr{\allowdisplaybreaks\begin{gather}#1\end{gather}}
       \def\bgrs#1\egrs{\allowdisplaybreaks\begin{gather*}#1\end{gather*}}

       \theoremstyle{plain}
       \newtheorem{lem}{\bf Lemma}[section]
       \newtheorem{thm}[lem]{\textbf{Theorem}}

       \newtheorem{cor}{\textbf{Corollary}}[section]
       
\begin{document}


\title{Green's Function and Pointwise  Space-time Behaviors of the  Three-Dimensional Relativistic Boltzmann Equation}

\author{ Yanchao Li$^1$,\, Mingying Zhong$^2$,\, \\[2mm]
 \emph
    {\small\it  $^1$School of Physical Science and Technology,
    Guangxi University, P.R.China.}\\
    {\small\it E-mail:\ liyanchaozjl@$163$.com}\\
    {\small\it  $^2$School of  Mathematics and Information Sciences,
    Guangxi University, P.R.China.}\\
    {\small\it E-mail:\ zhongmingying@gxu.edu.cn}\\[5mm]
    }
\date{ }

\pagestyle{myheadings}
\markboth{Relativistic Boltzmann Equation}%
{Y.-C. Li, M.-Y. Zhong}

 \maketitle

 \thispagestyle{empty}

\begin{abstract}\noindent{
The pointwise space-time behavior of the Green's function of the three-dimensional relativistic Boltzmann equation is studied in this paper. It is shown that the Green's function has a decomposition of the macroscopic diffusive waves and Huygens waves with the speed $\sqrt{a^2+b^2}$ at low-frequency, the singular kinetic wave and the remainder term decaying exponentially in space and time. In addition, we establish the pointwise space-time estimate of the global solution to the nonlinear relativistic Boltzmann equation with non-smooth initial data based on the Green's function.
}

\medskip



\end{abstract}

\tableofcontents


\section{Introduction}
\label{sect1}
\setcounter{equation}{0}

We consider the relativistic Boltzmann equation:
\be
\partial_tF+\tilde{v}\cdot\nabla_xF=Q(F,F),\label{rbe1}
\ee
where $F=F(t,x,v)$ is the distribution function with $(t,x,v)\in\R^{+}\times\R^3_x\times\R^3_v$ and the relativistic velocity $\tilde{v}$ is defined by
$$
\tilde{v}=\frac{v}{v_0},\quad v_0=\sqrt{1+|v|^2}.
$$
The collision operator $Q(F,G)$ is given by
$$
Q(F,G)=\int_{\R^3}\int_{\mathbb{S}^2}v_M\sigma(g,\theta)(F(v')G(u')-F(v)G(u))dud\Omega,
$$
where   $v_M=v_M(u,v)$ is M{\o}ller velocity given by
\bmas
&v_M=\frac{g\sqrt{4+g^2}}{u_0v_0},\quad u_0=\sqrt{1+|u|^2}, \\
&g^2=s-4,\quad s=2(u_0v_0-u\cdot v+1).
\emas
The function
$\sigma(g,\theta)$ is called the differential cross section or the scattering kernel, and it measures the interactions between particles. 
In this paper, we consider ``hard ball'' condition, namely
$$\sigma(g,\theta)=1.$$
Notice that the conservation law of momentum and energy are given by
$$
 u+v=u'+v',\quad u_0+v_0=u'_0+v'_0.
$$

There have been a lot of works on the well-posedness and asymptotical behaviors of solution to  the relativistic Boltzmann equation. The background of relativistic Boltzmann equations is mentioned in \cite{Cercignani-1}.  The existence and uniqueness of the weak solution to the relativistic Boltzmann equation has been proved in \cite{Dudynski-1,Dudynski-2}. The global solution of relativistic Boltzmann equation near a relativistic Maxwellian is studied in \cite{Glassey-2,Glassey-3} for hard potentials, and in \cite{Duan-1,STRIN-1} for soft potentials. The global existence of solutions near vacuum was investigated in \cite{Glassey-1}. It is shown in \cite{Yang-1} that the global solutions to the relativistic Boltzmann and Landau equations tend to the equilibriums at $(1+t)^{-\frac{3}{4}}$ in $L^2$-norm by using compensating function and  energy method.  The spectrum of the relativistic  Boltzmann equation with  hard sphere and hard potential was constructed in \cite{ZHAO-1}.

There has been much important progress made on  the well-posedness and asymptotical behaviors of solution to  the classical Boltzmann equation. In particular, the spectrum of the classical Boltzmann equation with hard sphere and hard potentials was constructed in \cite{Ellis-1,Ukai-1,Ukai-2}. The global existence of renormalized weak solution subject to general large initial data was shown in \cite{DiPerna-1}. The global existence and optimal decay rate $(1+t)^{-\frac{3}{4}}$ of strong solution near Maxwellian for hard sphere, hard potential and soft potential was proved in \cite{Caflischi-1,Caflischi-2,GUO-1,LIU-1,Ukai-1,Ukai-2,Zhong-1}. The global existence of solutions near vacuum was investigated in \cite{Glassey-3,Illner-1}. The pointwise behavior of the Green function of the Boltzmann equation for hard sphere was first verified in \cite{LIU-2,LIU-3,LIU-5}. Later, the pointwise behavior of Green function of the Boltzmann equation for hard, Maxwellian and soft potentials was studied in \cite{Wang-2,Wang-3,Wang-4}.

In this paper, we study the pointwise space-time behaviors of the Green's function and the global solution to the relativistic Boltzmann equation \eqref{rbe1} based on the spectral analysis \cite{ZHAO-1}.  To begin with, let us consider the Cauchy problem for the  relativistic Boltzmann equation \eqref{rbe1} with the following initial data:
\bq
F(0,x,v)=F_0(x,v),\quad (x,v)\in\mathbb{R}^3_x\times\mathbb{R}^3_v.\label{rbe2}
\eq
Without loss of generality, a global relativistic Maxwellian takes the form
$$
M(v)=e^{-\sqrt{1+|v|^2}}.
$$

Define the perturbation $f(t,x,v)$ of $F(t,x,v)$ near $M(v)$ by
$$
F=M+\sqrt{M}f,
$$
then the relativistic Boltzmann equation \eqref{rbe1} and \eqref{rbe2} is reformulated into
\bma
&\partial_tf+\tilde{v}\cdot\nabla_xf-Lf=\Gamma(f,f),\label{rbe3}\\
&f(0,x,v)=f_0(x,v)=\frac{1}{\sqrt{M}}(F_0-M),\label{rbe4}
\ema
where the linearized collision operator $Lf$ and nonlinear term $\Gamma(f,f)$ are defined by
$$
\left\{\bln
&Lf=\frac{1}{\sqrt{M}}[Q(M,\sqrt{M}f)+Q(\sqrt{M}f,M)], \\
&\Gamma(f,f)=\frac{1}{\sqrt{M}}Q(\sqrt{M}f,\sqrt{M}f).
\eln\right.
$$
We have (cf. \cite{STRIN-1})
\bq
\left\{\bln
 &(Lf)(v)=(Kf)(v)-\nu(v) f(v),   \quad  (Kf)(v)=K_2f-K_1f, \\
  &\nu(v)=4\pi\int_{\R^3} \frac{g\sqrt{4+g^2}}{u_0v_0}M(u)du ,\\
 &K_if(v)=\int_{\R^3}k_i(v,u)f(u)du,\quad i=1,2,\\
 &k_1(v,u)=\frac{4\pi g\sqrt{4+g^2}}{u_0v_0}e^{-\frac{u_0+v_0}{2}}, \quad k_2(v,u)=\frac{2s^{3/2}}{gu_0v_0}U_1(v,u)e^{-U_2(v,u)},\\
 &U_1=\(1+\frac{u_0+v_0}{2}U_2^{-1}+\frac{u_0+v_0}{2}U_2^{-2}\)U_2^{-1},\quad U_2=\frac{\sqrt{s}|u-v|}{2g},\label{L-K}
\eln\right.
\eq
where $\nu(v)$, the collision frequency, is a real function, and $K$ is a self-adjoint compact operator on $L^2(\R^3_v)$ with a real symmetric integral kernel $k(v,u)$.

The null space of the operator $L$, denoted by $N_0$, is a subspace spanned by the orthogonal basis $\{\chi_j,~j=0,1,2,3,4\}$ with
\bq
\chi_0=\frac{\sqrt{M}}{\sqrt{p_0}},\quad \chi_j=\frac{v_j\sqrt{M}}{\sqrt{p_1}} \,\, (j=1,2,3),\quad \chi_4=\frac{(v_0-p_2)\sqrt{M}}{\sqrt{p_3}},
\eq
where the normalized constants $p_j~(j=0,1,2,3)$ are given by
$$
p_0=\int_{\R^3}M(v)dv, \ p_1=\int_{\R^3}v_1^2M(v)dv, \  p_2=\int_{\R^3}v_0M(v)dv, \  p_3=\int_{\R^3}v_0^2M(v)dv-p_2^2.
$$

Let $L^2(\R^3)$ be a Hilbert space of complex-value functions $f(v)$ on $\R^3$ with the inner product and the norm
$$
(f,g)=\int_{\R^3}f(v)\overline{g(v)}dv,\quad \|f\|=\(\int_{\R^3}|f(v)|^2dv\)^{1/2}.
$$
Introduce the macro-micro decomposition as follows
\bq\label{mami}
\left\{\bln
&f=P_0f+P_1f,\\
&P_0f=P_0^1f+P_0^2f+P_0^3f,\ \  P_1f=f-P_0f,\\
&P_0^1f=(f,\chi_0)\chi_0,\ \  P_0^3f=(f,\chi_4)\chi_4,\\
&P_0^2f=(f,\chi_1)\chi_1+(f,\chi_2)\chi_2+(f,\chi_3)\chi_3.
\eln\right.
\eq
The linearized operator $L$ is non-positive and moreover, $L$ is locally coercive in the sense that there is a constant $\mu>0$ such that (cf. \cite{Glassey-1,Glassey-2,STRIN-1})
\bq
(Lf,f)\leq-\mu\|P_1f\|^2,\quad f\in D(L),\label{LF-F}
\eq
where $D(L)$ is the domain of $L$ given by
$$
D(L)=\{f\in L^2(\R^3)|\nu(v)f\in L^2(\R^3)\}.
$$

Since we only consider the pointwise behavior with respect to the space-time variable $(t,x)$, it's convenient to regard the green's function $G(t,x)$ as an operator on $L^2(\R^3_v)$:
\bq\label{rbe5}
\left\{\bln
&\partial_tG=BG,\\
&G(0,x)=\delta(x)I_v,
\eln\right.
\eq
where $I_v$ is the identity map on $L^2(\R^3_v)$ and the operator $B$ is defined by
\bq
Bf=Lf-\tilde{v}\cdot\nabla_xf.\label{LFB1}
\eq
Thus, the solution for the initial value problem of the linearized relativistic Boltzmann equaion
\bq \label{LRB}
\left\{\bln
&\partial_tf=Bf,\\
&f(0,x,v)=f_0(x,v)
\eln\right.
\eq
can be represented by
\bq
f(t,x)=G(t)\ast f_0=\int_{\R^3}G(t,x-y)f_0(y)dy,
\eq
where $f_0(y)=f_0(y,v)$.

For any $(t,x)\in\R_{+}\times\R^3$ and $f\in L^2(\R^3_v)$, we define the $L^2$ norm of $G(t,x)$ by
\bq
\|G(t,x)\|=\sup_{\|f\|=1}\|G(t,x)f\|,
\eq
and define the $L^2$ norm of a operator $T$ in $L^2(\R^3_v)$ as
\bq
\|T\|=\sup_{\|f\|=1}\|Tf\|.
\eq
\textbf{Notations: }Before stating the main results in this paper, we list some notations.
For any $\alpha=(\alpha_1,\alpha_2,\alpha_3)\in\N^3$ and $\beta=(\beta_1,\beta_2,\beta_3)\in\N^3$, denote
$$
\partial_x^{\alpha}=\partial_{x_1}^{\alpha_1}\partial_{x_2}^{\alpha_2}\partial_{x_3}^{\alpha_3},\quad \partial_v^{\beta}=\partial_{v_1}^{\beta_1}\partial_{v_2}^{\beta_2}\partial_{v_3}^{\beta_3}.
$$
The Fourier transform of $f=f(x,v)$ is denoted by
$$
\hat{f}(\xi,v)=\mathcal{F}f(x,v)=\frac{1}{(2\pi)^{3/2}}\int_{\R^3}e^{-\i x\cdot\xi}f(x,v)dx,
$$
where and throughout this paper we denote $\i=\sqrt{-1}$.

For $\gamma\geq0$, define
$$
\|f(x)\|_{L^{\infty}_{v,\gamma}}=\sup_{v\in\R^3}|f(x,v)|(1+v_0)^{\gamma}.
$$


First, we have the pointwise space-time behaviors of the Green's function for the linearized relativistic Boltzmann equation \eqref{rbe5}.
\begin{thm}\label{rbeth1}
Let $G(t,x)$ be the Green's function for the relativistic Boltzmann equation defined by \eqref{rbe5}. Then, the Green's function $G(t,x)$ can be decomposed into
$$
G(t,x)=G_1(t,x)+G_2(t,x)+W_{4}(t,x),
$$
where $W_4(t,x)$ is the singular kinetic wave, $G_1(t,x)$ is the fluid part at low frequency and $G_2(t,x)$ is the remainder part. For any $\alpha\in\N^3$, there exist two positive constants $C$ and $D$ such that the fluid part $G_1(t,x)$ is smooth and satisfies
\bma
&\|\partial_x^{\alpha}P_0^lG_1(t,x)\|\leq C(1+t)^{-\frac{3+|\alpha|}{2}}\(e^{-\frac{|x|^2}{D(1+t)}} +(1+t)^{-\frac{1}{2}}e^{-\frac{(|x|-\mathbf{c}t)^2}{D(1+t)}}\)+Ce^{-\frac{|x|+t}{D}},\ \ l=1,3,\label{Thm1-1-1}\\
&\|\partial_x^{\alpha}P_0^2G_1(t,x)\|\leq C(1+t)^{-\frac{3+|\alpha|}{2}}\(e^{-\frac{|x|^2}{D(1+t)}}+(1+t)^{-\frac{1}{2}}e^{-\frac{(|x|-\mathbf{c}t)^2}{D(1+t)}}\)+Ce^{-\frac{|x|+t}{D}}\nnm\\
&\qquad\qquad\qquad\qquad\quad+C(1+t)^{-\frac{3+|\alpha|}{2}}B_{\frac{3}{2}}(t,|x|)1_{\{|x|\leq\mathbf{c}t\}},\label{Thm1-1-2}\\
&\|\partial_x^{\alpha}P_1G_1(t,x)\|,\,\|\partial_x^{\alpha}G_1(t,x)P_1\|\leq C(1+t)^{-\frac{4+|\alpha|}{2}}\(e^{-\frac{|x|^2}{D(1+t)}} +(1+t)^{-\frac{1}{2}}e^{-\frac{(|x|-\mathbf{c}t)^2}{D(1+t)}}\)\nnm\\
&\qquad\qquad\qquad\qquad\quad+Ce^{-\frac{|x|+t}{D}}+C(1+t)^{-\frac{4+|\alpha|}{2}}B_{\frac{3}{2}}(t,|x|)1_{\{|x|\leq\mathbf{c}t\}},\label{Thm1-1-3}\\
&\|\partial_x^{\alpha}P_1G_1(t,x)P_1\|\leq C(1+t)^{-\frac{5+|\alpha|}{2}}\(e^{-\frac{|x|^2}{D(1+t)}} +(1+t)^{-\frac{1}{2}}e^{-\frac{(|x|-\mathbf{c}t)^2}{D(1+t)}}\)+Ce^{-\frac{|x|+t}{D}}\nnm\\
&\qquad\qquad\qquad\qquad\quad+C(1+t)^{-\frac{5+|\alpha|}{2}}B_{\frac{3}{2}}(t,|x|)1_{\{|x|\leq\mathbf{c}t\}},\label{Thm1-1-4}
\ema
where $\mathbf{c}>0$ is the sound speed defined by
\be \mathbf{c}=\sqrt{a^2+b^2},\ \  a=(\tilde{v}_1\chi_1,\chi_4),\ \  b=(\tilde{v}_1\chi_0,\chi_1),\ee
and the space-time diffusive profile $B_k(t,|x|-\lambda t)$ is defined for any $k>0$ and $\lambda\ge 0$ by
$$
B_k(t,|x|-\lambda t)=\(1+\frac{(|x|-\lambda t)^2}{1+t}\)^{-k}, \ \ (t,x)\in\mathbb{R}_{+}\times\mathbb{R}^3.
$$
The remainder part $G_2(t,x)$ is bounded and satisfies
\bq
\|G_2(t,x)\|\leq Ce^{-\frac{|x|+t}{D}}.\label{Thm1-2}
\eq
The singular kinetic wave $W_{4}(t,x)$ is defined by
\bq\label{Thm1-3}
W_{4}(t,x)=\sum^{12}_{k=0}J_k(t,x)
\eq
with
$$
\left\{\bln
&J_0(t,x)=S^t\delta(x)I_v=e^{-\nu(v)t}\delta(x-\tilde{v}t)I_v,\\
&J_k(t,x)=\int^t_0S^{t-s}KJ_{k-1}ds,\quad k\geq1.
\eln\right.
$$
Here $I_v$ is the identity map in $L^2(\mathbb{R}^3_v)$ and the operator $S^t$ is defined by
\bq\label{Thm1-4}
S^tg(x,v)=e^{-\nu(v)t}g(x-\tilde{v}t,v).
\eq
\end{thm}

\begin{thm}\label{rbeth2}
There exists a small constant $\delta_0>0$ such that if the initial data $f_0$ satisfies
\bq\label{rbeth2-1}
\|f_0(x)\|_{L^{\infty}_{v,2}}\leq \delta_0e^{-\frac{|x|}{D_1}},
\eq
then there exists a unique global solution $f=f(t,x,v)$ to the relativistic Boltzmann equation \eqref{rbe3}--\eqref{rbe4}, which satisfies for two constants $C,D_2>0$  that
\bma
\|f(t,x)\|_{L^{\infty}_{v,2}}&\leq C\delta_0(1+t)^{-\frac{3}{2}}\(e^{-\frac{|x|^2}{D_2(1+t)}}+(1+t)^{-\frac{1}{2}} e^{-\frac{(|x|-\mathbf{c}t)^2}{D_2(1+t)}}\)+C\delta_0e^{-\frac{|x|+t}{D_2}}\nnm\\ &\quad+C\delta_0(1+t)^{-\frac{3}{2}}\(B_{\frac{3}{2}}(t,|x|)+(1+t)^{-\frac{1}{2}} B_{1}(t,|x|-\mathbf{c}t)\)1_{\{|x|\leq\mathbf{c}t\}}.\label{rbeth2-2}
\ema
\end{thm}

The results in Theorem \ref{rbeth1} on the pointwise behavior of the Green's function $G$ to the relativistic Boltzmann equation \eqref{rbe3} is proved based on the spectral analysis \cite{ZHAO-1} and the ideas inspired by \cite{LI-5,LIU-2,LIU-3}. First, we estimate the Green's function $G$ inside the Mach region $|x|\leq2\mathbf{c}t$ based on the spectral analysis. Indeed, we decompose the Green's function $G$ into the lower frequency part $G_L$ and the high frequency part $G_H$, and further split $G_L$ into the fluid parts $G_{L,0}$ and the non-fluid parts $G_{L,1}$,  namely
$$
\left\{\bln
&G=G_L+G_H,\\
&G_L=G_{L,0}+G_{L,1}.
\eln\right.
$$

By using Fourier analysis techniques, we can show that the low-frequency fluid part $G_{L,0}(t,x)$ is smooth and contains huygens waves and diffuse waves for $|x|\leq2\mathbf{c}t$ since the Fourier transform of the linear relativistic Boltzmann operator $B(\xi)$ has five eigenvalues $\{\lambda_j(|\xi|),\ j=-1,0,1,2,3\}$ at the low frequency region $|\xi|\leq r_0$ satisfying
$$
\left\{\bln
&\lambda_{\pm1}(|\xi|)=\pm \i\mathbf{c}|\xi|-A_{\pm1}|\xi|^2+O(|\xi|^3),\\
&\lambda_0(|\xi|)=-A_0|\xi|^2+O(|\xi|^3),\\
&\lambda_2(|\xi|)=\lambda_3(|\xi|)=-A_2|\xi|^2+O(|\xi|^3).
\eln\right.
$$

Next, we apply a Picard's iteration to estimate  the Green's function $G$ outside the Mach region $|x|>2\mathbf{c}t$. Since $\hat{G}(t,\xi)$ does not belong to $L^1_{\xi}(\R^3)$, $G(t,x)$ can be decomposed into the singular wave and the bounded remainder part. To exact the singular wave from $G(t,x)$, we defined the approximate sequence $\{\hat{J}_{k}(t,\xi),\, k\geq 0\}$ for the Green's function $\hat{G}(t,\xi)$, where $\hat{J}_{k}(t,\xi)=\hat{\mathbb{M}}_{k}^t(\xi)$  with $\hat{\mathbb{M}}_{k}^t(\xi)$ to be the Mixture operator. Since the difference between  linear relativistic Boltzmann operator  and classical Boltzmann operator,  we provide the proof of Mixture Lemma (Lemma \ref{rbegf9}) after analyzing linear collision operator of relativistic Boltzmann equation.  By the Mixture Lemma, $\hat{\mathbb{M}}^t_k(\xi)$ is analytic in $D_{\delta}=\{\xi\in\mathbb{C}^3|~|\mathrm{Im}\xi|\leq\delta\}$ and satisfies
$$
\|\hat{\mathbb{M}}^t_{3k}(\xi)\|\leq C_k(1+t)^{3k}(1+|\xi|)^{-k}e^{- (\nu_0-\delta)t}, \quad \xi\in D_{\delta}.
$$
This  implies that  the approximate solution $J_{12}(t,x)$ is bounded and satisfies
\be
\|J_{12}(t,x)\|\leq Ce^{-\frac{\nu_{0}(|x|+t)}{4}}. \label{J6j1}
\ee
We define the singular part $W_{k}(t,x)$ as
$$
W_{k}(t,x)=\sum^{3k}_{i=0}J_i(t,x).
$$

Note that \eqref{J6j1} implies that the remainder part $ G(t,x)-W_4(t,x)$ is bounded for all $(t,x)\in\R_{+}\times\R^3$. Thus, by choosing the weighted function as $w=e^{\eps (|x|-Yt)}$, we can show by the weighted energy method that the remainder part $ G(t,x)-W_4(t,x)$ satisfies (refer to Lemma \ref{rbegf12})
$$
\|G(t,x)-W_{4}(t,x)\|\leq Ce^{-\frac{|x|+t}{D}}, \quad |x|\ge 2 \mathbf{c}t.
$$
Applying the above decompositions and estimates, we can obtain the decomposition and pointwise space-time behavior for each part of the Green's function $G(t,x)$ as listed Theorem \ref{rbeth1}.

Finally, by using the estimates of the Green's functions in Theorem \ref{rbeth1} and the estimates of waves coupling in Lemmas \ref{rbepw2}--\ref{rbepw11}, one can establish the pointwise space-time estimate on the global solution to the nonlinear relativistic Boltzmann equation given in Theorem \ref{rbeth2}. Note that comparing to the Green's function, there is an addition wave $(1+t)^{-2}B_1(t,|x|-\mathbf{c}t)1_{\{|x|\le \mathbf{c}t\}}$ in  the solution to the nonlinear equation. This addition wave comes from the the following nonlinear waves coupling (See Lemma \ref{rbepw3}):
\bmas
&\quad \intt \intr (1+t-s)^{-\frac52}e^{-\frac{(|x-y|-\mathbf{c}(t-s))^2}{D(1+t-s)}}(1+s)^{-4}e^{-\frac{2(|y|-\mathbf{c}s)^2}{D_1(1+s)}}dyds\\
&\le  C(1+t)^{-\frac52} \(e^{-\frac{3|x|^2}{2D_1(1+t)}}+e^{-\frac{3(|x|-\mathbf{c} t)^2}{2D_1(1+t)}}\) +C(1+t)^{-2}B_1(t,|x|-\mathbf{c}t)1_{\{|x|\le \mathbf{c}t\}}.
\emas

The rest of this paper is organized as follows. In section \ref{sect2}, we present results regarding the spectrum analysis of the linear operators related to the linearized relativistic Boltzmann equation. In section \ref{sect3}, we establish the pointwise space-time estimates of the Green's function to the linearized relativistic Boltzmann equation. In section \ref{sect4}, we prove the pointwise space-time estimates of the global solutions to the original nonlinear relativistic Boltzmann equation.

\section{Spectral analysis}
\label{sect2}
\setcounter{equation}{0}

In this section, we recall the spectral structure of the linearized relativistic Boltzmann operator $B(\xi)$ and analyze the analyticity of the eigenvalues and eigenfunctions of  $B(\xi)$ in order to study the pointwise estimate of the fluid part of the Green's function.

First, we take the Fourier transform in \eqref{rbe5} with respect to $x$ to have
\bq
\left\{\bln
&\partial_t\hat{G}=B(\xi)\hat{G},\quad t>0,\\
&\hat{G}(\xi,0)=1(\xi)I_v,
\eln\right.
\eq
where the operator $B(\xi)$ is defined by
\bq
B(\xi)=L-\i\tilde{v}\cdot\xi.\label{HBXI}
\eq

We have the following results about the spectrum structure and semigroup of the operator $B(\xi)$.
\begin{lem}[\cite{ZHAO-1}]\label{rbesp1}

(1) For any $r_1>0$, there exists $\alpha=\alpha(r_1)>0$ such that when $|\xi|>r_1$,
\bq
\sigma(B(\xi))\subset\{\lambda\in \mathbb{C}\,|\,\mathrm{Re}\lambda<-\alpha\}.
\eq

(2) For any $\delta>0$ and all $\xi\in\R^3$, there exists $y_1=y_1(\delta)>0$ such that
\bq
\rho(B(\xi))\supset\{\lambda\in\mathbb{C}\,|\,\mathrm{Re}\lambda\geq-\nu_0+\delta,\,|\mathrm{Im}\lambda|\geq y_1\}\cup\{\lambda\in\mathbb{C}\,|\,\mathrm{Re}\lambda>0\}.
\eq
\end{lem}

\begin{lem}[\cite{ZHAO-1}]\label{rbesp2}

(1)~There exist a constant $r_0>0$ such that the spectrum $\sigma(B(\xi))\cap\{\mathrm{Re}\lambda\geq-\frac{\mu}{2}\}$ consists of five points $\{\lambda_j(|\xi|),\,j=-1,0,1,2,3\}$ for $|\xi|\leq r_0$. In particular, the eigenvalues $\lambda_j(|\xi|)$ are $C^{\infty}$ functions in $|\xi|$  and satisfy the following expansion for $|\xi|\leq r_0$:
\bq\label{SA1}
\left\{\bln
&\lambda_{\pm1}(|\xi|)=\pm \i\mathbf{c}|\xi|-A_{\pm1}|\xi|^2+O(|\xi|^3),\\
&\lambda_0(|\xi|)=-A_0|\xi|^2+O(|\xi|^3),\\
&\lambda_2(|\xi|)=\lambda_3(|\xi|)=-A_2|\xi|^2+O(|\xi|^3),
\eln\right.
\eq
where $\mathbf{c}>0$, and $A_j>0$, $j=-1,0,1,2$ are constants. 

(2)~The semigroup $S(t,\xi)=e^{tB(\xi)}$ can be decomposed into
\bq\label{sami1}
S(t,\xi)f=S_1(t,\xi)f+S_2(t,\xi)f,\quad \forall\, f\in L^2(\R^3_v),
\eq
where
\bq\label{sami2}
S_1(t,\xi)=\sum^3_{j=-1}e^{\lambda_j(|\xi|)t}\(f,\overline{\psi_j(\xi)}\)\psi_j(\xi)1_{\{|\xi|\leq r_0\}},
\eq
with $(\lambda_j(|\xi|),\psi_j(\xi))$ being the eigenvalue and eigenfunction of the operator $B(\xi)$ for $|\xi|\leq r_0$, and $S_2(t,\xi)=:S(t,\xi)f-S_1(t,\xi)f$ satisfies for two constants $C,\kappa_0>0$ independent of $\xi$ that
\bq\label{sami3}
\|S_2(t,\xi)f\|\leq Ce^{-\kappa_0t}\|f\|.
\eq
\end{lem}

Now we analyze the analyticity and expansion of the eigenvalues and eigenfunctions of $B(\xi)$ at low frequency. To this end, we first consider a 1-D eigenvalue problem:
\bq
(L-\i\tilde{v}_1\eta)e=\beta e,\quad \eta\in \R.\label{SA}
\eq

We have the analyticity and expansion of the eigenvalues $\beta_j(\eta)$ and the corresponding eigenfunctions $e_j(\eta)$ at low frequency as follows.
\begin{lem}\label{rbesp3}
Let $\sigma(\eta)$ denote the spectrum set of the operator $L-i\tilde{v}_1\eta$. We have


(1) There exists a constant $r_0>0$ such that the spectrum $\sigma(\eta)\cap\{\mathrm{Re}\lambda\geq-\frac{\mu}{2}\}$ consists of five points $\{\beta_j(\eta),\,j=-1,0,1,2,3\}$ for $|\eta|\leq r_0$. In particular, the eigenvalues $\beta_j(\eta)$ are analytic functions in $\eta$ and satisfy
\bq\label{SA5}
\left\{\bln
\beta_{\pm1}(\eta)&=\pm \i\eta A^1_1(\eta^2)-A^2_1(\eta^2),\\
\beta_{0}(\eta)&=-A_0^2(\eta^2),\\
\beta_{2}(\eta)&=\beta_{3}(\eta)=-A^2_2(\eta^2),
\eln\right.
\eq
where $A^1_1(\eta)$ and $A^2_j(\eta)~(j=0,1,2)$ are analytic functions in $\eta$ and satisfy
\bq\label{SA6}
\left\{\bln
&A^1_1(0)=\mathbf{c},\ \  A^2_1(0)=A_0^2(0)=A^2_2(0)=0,\\
&\mathrm{\frac{d}{d\eta}}A^2_1(0)=A_1,\ \ \mathrm{\frac{d}{d\eta}}A^2_0(0)=A_0,\ \  \mathrm{\frac{d}{d\eta}}A^2_2(0)=A_2,\\
&A_j=-(L^{-1}P_1\tilde{v}_1F_{j},\tilde{v}_1F_{j})>0, \ \  j=-1,0,1,2
\eln\right.
\eq
with $F_{j}$ $(j=-1,0,1,2,3)$  defined by \eqref{EUJ}.

(2) The corresponding eigenfunctions $e_j(\eta)~(j=-1,0,1,2,3)$ are analytic in $\eta$ and satisfy
\bq\label{SA6j1}
\left\{\bln
&P_0e_{0}(\eta)=a_0(\eta^2)\chi_0+ \i\eta b_0(\eta^2)\chi_1+c_0(\eta^2)\chi_4,\\
&P_0e_{\pm1}(\eta)=\(a_{1,0}(\eta^2)\pm \i\eta a_{1,1}(\eta^2)\)\chi_0+\(\pm b_{1,0}(\eta^2)+\i\eta b_{1,1}(\eta^2)\)\chi_1 \\
&\qquad\qquad\quad+\(c_{1,0}(\eta^2)\pm \i\eta c_{1,1}(\eta^2)\)\chi_4,\\
&P_0e_{k}(\eta)=c_2(\eta^2)\chi_k, \quad  k=2,3,\\
&P_1e_j(\eta)=\i(L-\beta_j-\i\eta P_1\tilde{v}_1 P_1)^{-1}P_1(\tilde{v}_1 P_0e_j(\eta)),
\eln\right.
\eq
where $a_0(\eta),b_0(\eta),c_0(\eta),c_2(\eta)$, $a_{1,l}(\eta), b_{1,l}(\eta), c_{1,l}(\eta)$ with $l=0,1$ are  analytic functions in $\eta$ and satisfy
\be \label{abc1}
\left\{\bal
a_0(0)=- \frac{a }{\sqrt{a^2+b^2}} ,\ \ b_0(0)=\frac{\sqrt{2}A_{0,-1}}{\sqrt{a^2+b^2}},\ \ c_0(0)=\frac{b }{\sqrt{a^2+b^2}} ,\ \  b_{1,0}(0)=-\sqrt{\frac{1}{2}}, \\
a_{1,0}(0)= \frac{b }{\sqrt{2a^2+2b^2} }, \ \ a_{1,1}(0)= \frac{ b A_{1,-1}+2\sqrt{2}aA_{1,0}}{2\sqrt{2}(a^2+b^2)} ,\ \  b_{1,1}(0)=\frac{A_{1,-1}}{2\sqrt{2a^2+2b^2}}, \\
c_{1,0}(0)=\frac{a }{\sqrt{2a^2+2b^2}},\ \  c_{1,1}(0)= \frac{aA_{1,-1}-2\sqrt{2}bA_{1,0}}{2\sqrt{2}(a^2+b^2)} ,\ \   c_2(0)=1,\\
A_{i,j}=-(L^{-1}P_1\tilde{v}_1F_{i},\tilde{v}_1F_{j})>0, \ \  i,j=-1,0,1.
\ea\right.
\ee
\end{lem}
\begin{proof}
We shall prove that eigenvalues $\beta_j~(j=-1,0,1,2,3)$ are analytic functions in $\eta$ and satisfy \eqref{SA5}--\eqref{SA6}. By the macro-micro decomposition, the eigenfunction $e$ of \eqref{SA} can be divided into
\bq
e=P_0e+P_1e=g_0+g_1.\label{safj1}
\eq
Let $\beta=\eta\gamma$. Hence, \eqref{SA} gives
\bma
&\gamma g_0=-\i P_0\tilde{v}_1(g_0+g_1),\label{safj2}\\
&\eta\gamma g_1=Lg_1-\i\eta P_1\tilde{v}_1(g_0+g_1).\label{safj3}
\ema
By \eqref{safj3}, the microscopic part $g_1$ can be represented in terms of the macroscopic part $g_0$ as
\bq
g_1=\i\eta(L-\eta\gamma -\i\eta P_1\tilde{v}_1P_1)^{-1}P_1\tilde{v}_1g_0.\label{safj4}
\eq
Substituting it into \eqref{safj2}, we obtain the eigenvalue problem for the macroscopic part $g_0$ as
\bq
\gamma g_0=-\i P_0\tilde{v}_1g_0+\eta P_0\tilde{v}_1(L-\eta\gamma -\i\eta P_1\tilde{v}_1P_1)^{-1}P_1\tilde{v}_1g_0.\label{safj5}
\eq

We have the matrix representation of $P_0\tilde{v}_1P_0$ as follows:
\bq
P_0\tilde{v}_1P_0=\left(
  \begin{array}{ccccc}
    0 & be_1 & 0 \\
    be_1^T & 0  & ae_1^T \\
    0 & ae_1 & 0 \\
  \end{array}
\right)\quad {\rm with}\quad e_1=(1,0,0).
\eq
It can be verified that the eigenvalues $u_j$ and  eigenvectors $F_j$ of $P_0\tilde{v}_1P_0$ are given by
\bq\label{EUJ}
\left\{\bal
u_{\pm1}=\mp\sqrt{a^2+b^2},\ \ u_j=0,\ \ j=0,2,3,\\
F_{\pm1}=\sqrt{\frac{b^2}{2a^2+2b^2}}\chi_0\mp\sqrt{\frac{1}{2}}\chi_1+\sqrt{\frac{a^2}{2a^2+2b^2}}\chi_4,\\
F_0=-\sqrt{\frac{a^2}{a^2+b^2}}\chi_0+\sqrt{\frac{b^2}{a^2+b^2}}\chi_4,\\
F_l=\chi_l,\quad l=2,3,\\
(F_j ,F_k )=\delta_{jk},\quad -1\leq j,k\leq3.
\ea\right.
\eq

We rewrite $g_0\in N_0$ in terms of the basis $F_j$ as
\bq
g_0=\sum^{4}_{j=0}C_jF_{j-1}\quad {\rm with}\quad C_{j}=(g_0,F_{j-1}),\,\, j=0,1,2,3,4.
\eq
Taking the inner product between \eqref{safj5} and $F_j~(j=-1,0,1,2,3)$ respectively, we have the following equations about $\gamma$ and $(C_0,C_1,C_2,C_3,C_4)$:
\bma
&\gamma C_j=-\i u_{j-1}C_j+\eta\sum^2_{k=0}C_kR_{kj}(\gamma,\eta),\quad j=0,1,2,\label{spew6}\\
&\gamma C_l=\eta C_lR_{33}(\gamma,\eta),\quad l=3,4,\label{spew7}
\ema
where
\bq
R_{kj}=R_{kj}(\gamma,\eta)=((L-\eta\gamma -i\eta P_1\tilde{v}_1P_1)^{-1}P_1(\tilde{v}_1F_{k-1}),\tilde{v}_1F_{j-1}).\label{sprkj1J1}
\eq

Denote
\bma
D_0(\gamma,\eta)&=\gamma-\eta R_{33}(\gamma,\eta),\label{spd01J1}\\
D_1(\gamma,\eta)&= \left|\begin{array}{ccc}
                  \gamma+\i u_{-1}-\eta R_{00} & -\eta R_{10} & -\eta R_{20} \\
                  -\eta R_{01} & \gamma-\eta R_{11} & -\eta R_{21} \\
                  -\eta R_{02} & -\eta R_{12} & \gamma+\i u_1-\eta R_{22}
                \end{array}\right|.\label{spd11J1}
\ema
The eigenvalues $\beta= \eta\gamma$ can be solved by $D_0(\gamma,\eta)=0$ and $D_1(\gamma,\eta)=0$. By a direct computation and the implicit function theorem, we have 
\begin{lem}[\cite{ZHAO-1}]\label{rbesp4}
(1)~The equation $D_0(\gamma,\eta)=0$ has a unique analytic solution $\gamma=\gamma(\eta)$ for $(\eta,\gamma)\in[-r_0,r_0]\times B_{r_1}(0)$ with $r_0,r_1>0$ being small constants that satisfies
$$
\gamma(0)=0,\quad \gamma'(0)=(L^{-1}P_1\tilde{v}_1F_{2},\tilde{v}_1F_{2}).
$$
Moreover, $\gamma(\eta)$ is a real, odd function.

(2)~There exist two small constants $r_0,r_1>0$ so that the equation $D_1(\gamma,\eta)=0$ admits three analytic solutions $\gamma_j(\eta)\ (j=-1,0,1)$ for $(\eta,\gamma_j)\in[-r_0,r_0]\times B_{r_1}(-\i u_j)$ that satisfy
$$
\gamma_j(0)=-\i u_j,\quad \gamma'_{j}(0)=(L^{-1}P_1\tilde{v}_1F_{j},\tilde{v}_1F_{j}).
$$
Moreover,  $\gamma_j(\eta)$ satisfies
$$
-\gamma_j(-\eta)=\overline{\gamma_j(\eta)}=\gamma_{-j}(\eta),\quad j=-1,0,1.
$$
\end{lem}

The eigenvalues $\beta_j(\eta)$ and the corresponding eigenfunctions $e_j(\eta)$ with $j=-1,0,1,2,3$ can be constructed as follows. For $j=2,3$, we take $\beta_j=\eta\gamma(\eta)$ with $\gamma(\eta)$ being the solution of the equation $D_0(\gamma,\eta)=0$, and choose $C_0=C_1=C_2=0$. And the corresponding eigenfunctions $e_j(\eta)$, $j=2,3$ are defined by
\bq\label{eginf1J1}
e_j(\eta)=C_3(\eta)F_j+i\eta(L-\beta_j-\i\eta P_1\tilde{v}_1P_1)^{-1}P_1(\tilde{v}_1C_3(\eta) F_j).
\eq

The coefficient $C_3(\eta)$ in \eqref{eginf1J1} is determined by the normalization condition $(e_j(\eta),\overline{e_j(\eta)})=1$:
$$
C_3(\eta)^2(1+\eta^2D(\eta))=1,
$$
where
$$
D(\eta)=((L-\beta_2-\i\eta P_1\tilde{v}_1P_1)^{-1}P_1(\tilde{v}_1F_2),(L-\overline{\beta_2}+\i\eta P_1\tilde{v}_1P_1)^{-1}P_1(\tilde{v}_1F_2)).
$$
This together with $D(\eta)=D(-\eta)$ gives $C_3(-\eta)=C_3(\eta)$ and $C_3(0)=1$.

For $j=-1,0,1$, we choose $\beta_j=\eta\gamma_j(\eta)$ with $\gamma_j(\eta)$ being the solution of $D_1(\gamma_j,\eta)=0$, and denote by $ \{C_0^j,C_1^j,C_2^j\}$ a solution of system \eqref{spew6}--\eqref{spew7} for $\gamma=\gamma_j(\eta)$. Then we can construct the corresponding eigenfunctions $e_j(\eta),~j=-1,0,1$ as
\bq\label{eginf2J1}
\left\{\bln
&e_j(\eta)=P_0e_j(\eta)+P_1e_j(\eta),\\
&P_0e_j(\eta)=C_0^j(\eta)F_{-1}+C_1^j(\eta)F_0+C_2^j(\eta)F_1,\\
&P_1e_j(\eta)=\i\eta(L-\beta_j-i\eta P_1\tilde{v}_1P_1)^{-1}P_1(\tilde{v}_1P_0e_j(\eta)).
\eln\right.
\eq
By \eqref{spew6}, the coefficients $C_0^j(\eta),C_1^j(\eta),C_2^j(\eta)$ with $j=-1,0,1$ in \eqref{eginf2J1} satisfy the following system:
\bq\label{eginf3J1}
\left\{\bln
&\(\beta_j+\i\eta\sqrt{a^2+b^2}-\eta^2R_{00}\)C_0^j(\eta)-\eta^2R_{10}C_1^j(\eta)-\eta^2R_{20}C_2^j(\eta)=0,\\
&-\eta^2R_{01}C_0^j(\eta)+\(\beta_j-\eta^2R_{11}\)C_1^j(\eta)-\eta^2R_{21}C_2^j(\eta)=0,\\
&-\eta^2R_{02}C_0^j(\eta)-\eta^2R_{12}C_1^j(\eta)+\(\beta_j-\i\eta\sqrt{a^2+b^2}-\eta^2R_{22}\)C_2^j(\eta)=0.
\eln\right.
\eq

Since $R_{kj}(\beta,\eta)$, $k,j=0,1,2$ are analytic in $(\beta,\eta)$ and $\beta_j(\eta)$, $j=-1,0,1$ are analytic functions of $\eta$ for $|\eta|\leq r_0$, it follows that $C_0^j(\eta),C_1^j(\eta),C_2^j(\eta)$ for $j=-1,0,1$ are analytic functions of $\eta$ for $|\eta|\leq r_0$. By Theorem 2.7 in \cite{ZHAO-1}, one has
\be \label{rbespbet01z}
\left\{\bln
&C^j_{j+1}(0)=1,\quad C^j_{k+1}(0)=1,\quad k\ne j,\\
&\frac{\rm d}{\rm d \eta}C^j_{j+1}(0)=0,\quad \frac{\rm d}{\rm d \eta}C^j_{k+1}(0)=\frac{(L^{-1}P_1\tilde{v}_1F_{j},\tilde{v}_1F_{k})}{\i(u_{k}-u_j)},\quad k\ne j.
\eln\right.
\ee
Since $\beta_0(\eta)=\beta_0(-\eta)$, we have by changing variable $v_1\rightarrow-v_1$ that
$$
\left\{\bln
&R_{00}(\beta_0,-\eta)=R_{22}(\beta_0,\eta),\ R_{10}(\beta_0,-\eta)=R_{12}(\beta_0,\eta),\ R_{01}(\beta_0,-\eta)=R_{21}(\beta_0,\eta),\\ &R_{20}(\beta_0,-\eta)=R_{02}(\beta_0,\eta),\ R_{11}(\beta_0,-\eta)=R_{11}(\beta_0,\eta),
\eln\right.
$$
which together with \eqref{eginf3J1} and \eqref{rbespbet01z} implies that $C_0^0(-\eta)=C_2^0(\eta)$ and $C_1^0(-\eta)=C_1^0(\eta)$.

Since $\beta_{-1}(\eta)=\beta_1(-\eta)$, we have by changing variable $v_1\rightarrow-v_1$ that
$$
\left\{\bln
&R_{00}(\beta_{-1},-\eta)=R_{22}(\beta_1,\eta),\ R_{10}(\beta_{-1},-\eta)=R_{12}(\beta_1,\eta),\ R_{01}(\beta_{-1},-\eta)=R_{21}(\beta_1,\eta),\\ &R_{20}(\beta_{-1},-\eta)=R_{02}(\beta_1,\eta),\ R_{11}(\beta_{-1},-\eta)=R_{11}(\beta_1,\eta),
\eln\right.
$$
which together with \eqref{eginf3J1} and \eqref{rbespbet01z} implies that $C_0^{-1}(-\eta)=C_2^{1}(\eta)$, $C_1^{-1}(-\eta)=C_1^{1}(\eta)$ and $C_2^{-1}(-\eta)=C_0^{1}(\eta)$.

Thus, we can rewrite $P_0e_j(\eta)$, $j=-1,0,1$ as
\bq \label{abc}
\left\{\bln
&P_0e_{-1}(\eta)=\(C^1_{2,0}(\eta^2)- \i\eta C^1_{2,1}(\eta^2)\)F_{-1}+\(C^1_{1,0}(\eta^2)- \i\eta C^1_{1,1}(\eta^2)\)F_0\\
&\qquad\qquad\quad+\(C^1_{0,0}(\eta^2)- \i\eta C^1_{0,1}(\eta^2)\)F_1,\\
&P_0e_{0}(\eta)=\(C^0_{0,0}(\eta^2)+\i\eta C^0_{0,1}(\eta^2)\)F_{-1}+C^0_{1,0}(\eta^2)F_0\\
&\qquad\qquad\quad+\(C^0_{0,0}(\eta^2)-\i\eta C^0_{0,1}(\eta^2)\)F_1,\\
&P_0e_{1}(\eta)=\(C^1_{0,0}(\eta^2)+ \i\eta C^1_{0,1}(\eta^2)\)F_{-1}+\(C^1_{1,0}(\eta^2)+ \i\eta C^1_{1,1}(\eta^2)\)F_0\\
&\qquad\qquad\quad+\(C^1_{2,0}(\eta^2)+ \i\eta C^1_{2,1}(\eta^2)\)F_1.
\eln\right.
\eq
By \eqref{EUJ} and \eqref{abc},  we can obtain \eqref{SA6j1}  with
$$
\left\{\bal
a_{1,l}(\eta^2)=\sqrt{\frac{b^2}{2a^2+2b^2}}\(C^1_{0,l}(\eta^2)+C^1_{2,l}(\eta^2)\)-\sqrt{\frac{a^2}{a^2+b^2}}C^1_{1,l}(\eta^2),\\
b_{1,l}(\eta^2)=\sqrt{\frac{1}{2}}\(C^1_{0,l}(\eta^2)-C^1_{2,l}(\eta^2)\),\\
c_{1,l}(\eta^2)=\sqrt{\frac{a^2}{2a^2+2b^2}}\(C^1_{0,l}(\eta^2)+C^1_{2,l}(\eta^2)\)+\sqrt{\frac{b^2}{a^2+b^2}}C^1_{1,l}(\eta^2),\\
a_0(\eta^2)=\sqrt{\frac{2b^2}{a^2+b^2}}C^0_{0,0}(\eta^2)-\sqrt{\frac{a^2}{a^2+b^2}}C^0_{1,0}(\eta^2),\ \ b_0(\eta^2)=-\sqrt{2}C^0_{0,1}(\eta^2),\\
c_0(\eta^2)=\sqrt{\frac{2a^2}{a^2+b^2}}C^0_{0,0}(\eta^2)+\sqrt{\frac{b^2}{a^2+b^2}}C^0_{1,0}(\eta^2).
\ea\right.
$$
The proof of the lemma is completed.
\end{proof}

Then, we consider the 3-D eigenvalue problem:
\bq\label{SPA1}
(L-\i\tilde{v}\cdot\xi)\psi=\lambda\psi,\quad \xi\in\R^3.
\eq

With the help of Lemma \ref{rbesp3}, we have the analyticity and expansion of the eigenvalues $\lambda_j(|\xi|)$ and the corresponding eigenfunctions $\psi_j(\xi)$ of $B(\xi)$ at low frequency.
\begin{lem}\label{rbesp5}
(1) The eigenvalues $\lambda_j(|\xi|)$, $j=-1,0,1,2,3$ given by \eqref{SA1} are analytic functions of $|\xi|$  and satisfy
\bq\label{rbesp501}
\left\{\bln
\lambda_{\pm1}(|\xi|)&=\pm i|\xi| A^1_1(|\xi|^2)-A^2_1(|\xi|^2),\\
\lambda_0(|\xi|) &=-A_0^2(|\xi|^2),\\
\lambda_2(|\xi|)&=\lambda_3(|\xi|)=-A^2_2(|\xi|^2),
\eln\right.
\eq
where $A^1_1(\eta)$ and $A^2_j(\eta)$, $j=0,1,2$ are analytic functions in $\eta$ given in Lemma \ref{rbesp3}.

(2) The corresponding eigenfunctions $\psi_j(\xi)$, $j=-1,0,1,2,3$ satisfy
\bq\label{SPA2}
\left\{\bln
&P_0\psi_{0}(\xi)=a_0(|\xi|^2)\chi_0+ \i b_0(|\xi|^2)(\xi\cdot\chi')+c_0(|\xi|^2)\chi_4,\\
&P_0\psi_{\pm 1}(\xi)=\(a_{1,0}(|\xi|^2)\pm i|\xi| a_{1,1}(|\xi|^2)\)\chi_0+\(\pm b_{1,0}(|\xi|^2)+\i|\xi| b_{1,1}(|\xi|^2)\)\frac{\xi\cdot\chi'}{|\xi|}\\
&\qquad\qquad\quad+\(c_{1,0}(|\xi|^2)\pm \i|\xi| c_{1,1}(|\xi|^2)\)\chi_4,\\
&P_0\psi_{k}(\xi)=c_2(|\xi|^2)(W^k\cdot\chi'),\quad  k=2,3,\\
&P_1\psi_j(\xi)=\i(L-\lambda_j(|\xi|)-\i P_1\tilde{v}\cdot\xi P_1)^{-1}P_1(\tilde{v}\cdot\xi P_0\psi_j(\xi)),
\eln\right.
\eq
where $\chi'=(\chi_1,\chi_2,\chi_3)$ and $a_0(\eta),b_0(\eta),c_0(\eta),c_2(\eta)$, $a_{1,l}(\eta),b_{1,l}(\eta), c_{1,l}(\eta)$ for $l=0,1$ are analytic functions in $\eta$, and $W^j~(j=2,3)$ are orthonormal vectors satisfying $W^j\cdot\xi=0$.
\end{lem}
\begin{proof}
Let $\mathbb{O}$ be a rotational transformation in $\R^3$ such that $\mathbb{O}:\frac{\xi}{|\xi|}\rightarrow(1,0,0)$. We have
\bq
\mathbb{O}^{-1}(L-\i\tilde{v}\cdot\xi)\mathbb{O}=L-\i\tilde{v}_1|\xi| .\label{Otrans1}
\eq
Thus, by Lemma \ref{rbesp3}, we have the following eigenvalues and eigenfunctions for \eqref{SPA1}:
\bma
&(L-\i\tilde{v}\cdot\xi)\psi_j(|\xi|)=\lambda_j(|\xi|)\psi_j(\xi),\nnm\\
&\lambda_j(|\xi|)=\beta_j(|\xi|),\quad \psi_j(\xi)=\mathbb{O}e_j(|\xi|),\quad  j=-1,0,1,2,3.\nnm
\ema
This proves the Lemma.
\end{proof}

\section{Green's function}
\label{sect3}
\setcounter{equation}{0}

In this section, we establish the pointwise space-time estimates of the Green's function to the linear relativistic Boltzmann equation \eqref{rbe5}. First, we consider the Green's function in finite Mach region. Based on the spectral analysis given in Section \ref{sect2}, we divide the Green's function into the fluid part and the non-fluid part and estimate the fluid part by complex analytical techniques. Then, we estimate the Green's function outside finite Mach region. We apply a Picard's iteration to construct the singular wave for the Green's function, and estimate the remainder part by weighted energy estimate. 

\subsection{Fluid part}

In this subsection, we establish the pointwise estimates of the fluid part of  Green's function  based on the spectral analysis given in Section \ref{sect2}. To this end, we decompose the operator $G(t,x)$ into a low-frequency part and a high-frequency part:
\be \label{GL-H}
\left\{\bln
&G(t,x)=G_L(t,x)+G_H(t,x),\\
&G_L(t,x)=\frac1{(2\pi)^{\frac{3}{2}}}\int_{\{|\xi|\leq \frac{r_0}{2}\}} e^{ \i x\cdot\xi +tB(\xi)}d\xi,\\
&G_H(t,x)=\frac1{(2\pi)^{\frac{3}{2}}}\int_{\{|\xi|>\frac{r_0}{2}\}} e^{ \i x\cdot\xi +tB(\xi)}d\xi.
\eln\right.
\ee
The operator $G_L(t,x)$ can be further divided into the fluid part and the non-fluid part:
\bq\label{GL-H2}
G_L(t,x)=G_{L,0}(t,x)+G_{L,1}(t,x),
\eq
where
\bma
&G_{L,0}(t,x)=\sum^{3}_{j=-1}\frac{1}{(2\pi)^{\frac{3}{2}}}\int_{\{|\xi|\le \frac{r_{0}}{2}\}}e^{\i x\cdot\xi +\lambda_j(|\xi|)t}\psi_{j}(\xi)\otimes\langle \psi_{j}(\xi)|d\xi,\label{GL-HGL0}\\
&G_{L,1}(t,x)=G_{L}(t,x)-G_{L,0}(t,x),
\ema
where $(\lambda_j,\psi_j)$ is defined by \eqref{rbesp501} and \eqref{SPA2}.
Here for any $f,g\in L^2(\mathbb{R}^3_v)$, the operator $f\otimes\langle g|$ on $L^2(\mathbb{R}^3_v)$ is defined by \cite{LIU-2,LIU-3}
$$
f\otimes\langle g|u=(u,\overline{g})f, \quad \forall\, u\in L^2(\mathbb{R}^3_v).
$$

From \cite{ZHAO-1}, we have the following estimates of each part of $G(t,x)$ defined by \eqref{GL-H}.
\begin{lem}\label{rbegf1}
 For any $g_{0}\in L^{2}(\mathbb{R}^{3}_{v})$, there exist positive constants $C$ and $\kappa_{0}$ such that
\bma
&\|\hat{G}_{L}(t,\xi)g_{0}\|\leq C\|g_{0}\|,\\
&\|\hat{G}_{L,1}(t,\xi)g_{0}\|\leq Ce^{-\kappa_{0}t}\|g_{0}\|,\\
&\|\hat{G}_{H}(t,\xi)g_{0}\|\leq Ce^{-\kappa_{0}t}\|g_{0}\|,
\ema
where $\hat{G}_{L}(t,\xi)$, $\hat{G}_{L,1}(t,\xi)$, and $\hat{G}_{H}(t,\xi)$ are the Fourier transforms of $G_L(t,x)$, $G_{L,1}(t,x)$ and $G_{H}(t,x)$.
\end{lem}

Denote
\bma
&\hat{G}^1_{L,0}(t,\xi)=\sum_{j=\pm1}e^{\lambda_j(|\xi|)t}\psi_j\otimes\langle \psi_j|-\sum_{j=\pm1}e^{\lambda_j(|\xi|)t}\psi^{\ast}_j\otimes\langle \psi^{\ast}_j|,\label{GL01}\\
&\hat{G}^2_{L,0}(t,\xi)=e^{\lambda_0(|\xi|)t}\psi_0\otimes\langle \psi_0|,\label{GL02}\\
&\hat{G}^3_{L,0}(t,\xi)=e^{\lambda_2(|\xi|)t}\bigg(\sum_{j=2,3}\psi_j\otimes\langle \psi_j|+\sum_{j=\pm1}\psi^{\ast}_j\otimes\langle \psi^{\ast}_j|\bigg),\label{GL03}\\
&\hat{G}^4_{L,0}(t,\xi)=\sum_{j=\pm1}\(e^{\lambda_j(|\xi|)t}-e^{\lambda_2(|\xi|)t}\)\psi^{\ast}_j\otimes\langle \psi^{\ast}_j|,\label{GL04}
\ema
where
$$
\psi^{\ast}_j=P_0^2\psi_j+\i(L-\beta_j(|\xi|)-\i P_1\tilde{v}\cdot\xi P_1)^{-1}P_1\tilde{v}\cdot\xi P_0^2\psi_j.
$$
By \eqref{GL-HGL0} and \eqref{GL01}--\eqref{GL04}, we have
\bq
G_{L,0}(t,x)=\sum^{4}_{j=1}G^j_{L,0}(t,x),
\eq
where
\bq
G^j_{L,0}(t,x)=\frac{1}{(2\pi)^{\frac{3}{2}}}\int_{\{|\xi|\leq\frac{r_0}{2}\}}e^{\i x\cdot\xi}\hat{G}^j_{L,0}(t,\xi)d\xi.\label{GL06}
\eq

Then, we estimate the pointwise behavior of $G^j_{L,0}(t,x)$ $(j=1,2,3,4)$ as follows. By Lemma \ref{rbesp5} and noting that (cf. \cite{LIU-3,LIU-5})
$$
(L-\lambda_{\pm1}(|\xi|)-\i P_1\tilde{v}\cdot\xi P_1)^{-1}=\mu_1(\xi)\pm|\xi|\mu_2(\xi),
$$  where $\mu_l(\xi)$, $l=1,2$ are analytic operators in $\xi$ for $|\xi|\leq r_0$, we have
\bma
&\psi_{\pm1}(\xi)=g_0\pm |\xi|g_1+\psi^*_{\pm 1},\label{gle1}\\
&\psi^*_{\pm1}(\xi)=\pm \frac1{|\xi|}(h\cdot\xi)+(\tilde{h}\cdot\xi),\label{gle2}
\ema
where $g_0,g_1$, $h=(h_1,h_2,h_3)$ and $\tilde{h}=(\tilde{h}_1,\tilde{h}_2,\tilde{h}_3)$ are analytic functions in $\xi$ for $|\xi|\leq r_0$ given by
\be \label{ghl}
\left\{\bln
g_0&=u_0+\mu_1(\xi)P_1(\tilde{v}\cdot\xi)u_0+|\xi|^2\mu_2(\xi)P_1(\tilde{v}\cdot\xi)u_1,\\
g_1&=u_1+\mu_1(\xi)P_1(\tilde{v}\cdot\xi)u_1+ \mu_2(\xi)P_1(\tilde{v}\cdot\xi)u_0,\\
h_j&=b_{1,0}(\chi_j+\mu_1(\xi)P_1(\tilde{v}\cdot\xi)\chi_j)+\i|\xi|^2b_{1,1}\mu_2(\xi)P_1(\tilde{v}\cdot\xi)\chi_j,\\
 \tilde{h}_j&=\i b_{1,1}(\chi_j+\mu_1(\xi)P_1(\tilde{v}\cdot\xi)\chi_j)+ b_{1,0}\mu_2(\xi)P_1(\tilde{v}\cdot\xi)\chi_j,
\eln\right.
\ee
with $u_0=a_{1,0}\chi_0+c_{1,0}\chi_4$ and $u_1=\i a_{1,1}\chi_0+\i c_{1,1}\chi_4$.

By  Lemma \ref{rbesp5}, \eqref{gle1}, \eqref{gle2} and using the similar arguments as those of Lemma 6.1 and Lemmas 7.1--7.3 in \cite{LIU-3}, we can show the analytical property of $\hat{G}^{j}_{L,0}(t,\xi)~(j=1,2,3,4)$.
Hence, we only state the result as follows and omit the detail of the proof for brevity.

\begin{lem}\label{rbegf2}
The fluid parts $\hat{G}^j_{L,0}(t,\xi)~(j=1,2,3,4)$ defined by \eqref{GL01}--\eqref{GL04} can be written in the following forms:
\bma
\hat{G}^1_{L,0}(t,\xi)&=e^{-A^2_1(|\xi|^2)t}\cos(|\xi| A^1_1(|\xi|^2)t)\mathcal{B}_1(\xi)+e^{-A^2_1(|\xi|^2)t}\frac{\sin(|\xi| A^1_1(|\xi|^2)t)}{|\xi|}\mathcal{B}_2(\xi),\label{Hgf1}\\
\hat{G}^{2}_{L,0}(t,\xi)&=e^{-A^2_0(|\xi|^2)t}\mathcal{B}_3(\xi), \quad \hat{G}^{3}_{L,0}(t,\xi)=e^{-A^2_2(|\xi|^2)t}\mathcal{B}_4(\xi),\label{Hgf2-3}\\
\hat{G}^{4}_{L,0}(t,\xi)&= \(e^{-A^2_1(|\xi|^2)t}\cos(|\xi|A^1_1(|\xi|^2)t)-e^{-A^2_2(|\xi|^2)t}\)\sum^3_{i,j=1} \frac{\xi_i\xi_j}{|\xi|^2}\mathcal{B}^{ij}_5(\xi)\nnm\\
&\qquad +e^{-A_1^2(|\xi|^2)t}\sum_{i,j=1}^3\xi_i\xi_j\frac{\sin(|\xi|A^1_1(|\xi|^2)t)}{|\xi|}\mathcal{B}^{ij}_6(\xi),\label{Hgf4}
\ema
where $\mathcal{B}_l(\xi)$ $(l=1,2,3,4)$ and $\mathcal{B}_k^{ij}(\xi)$ $(k=5,6$, $i,j=1,2,3)$ are analytic functions in $ \xi$ defined by 
\bmas
\mathcal{B}_1(\xi)&= 2(g_0\otimes\langle g_0|+|\xi|^2g_1\otimes\langle g_1|)+2(g_1\otimes\langle (h\cdot\xi)|+(h\cdot\xi)\otimes\langle g_1|)\\
&\quad+2(g_0\otimes\langle (\tilde{h}\cdot\xi)|+(\tilde{h}\cdot\xi)\otimes\langle g_0|),\\
\mathcal{B}_2(\xi)&= 2\i|\xi|^2(g_0\otimes\langle g_1|+g_1\otimes\langle g_0|)+2\i(g_0\otimes\langle (h\cdot\xi)|+(h\cdot\xi)\otimes\langle g_0|)\\
&\quad+2\i|\xi|^2(g_1\otimes\langle (\tilde{h}\cdot\xi)|+(\tilde{h}\cdot\xi)\otimes\langle g_1|),\\
\mathcal{B}_3(\xi)&= \psi_0(\xi)\otimes \langle \psi_0(\xi)|,\\
\mathcal{B}_4(\xi)&= \frac1{|\xi|^2} [2(h\cdot\xi)\otimes \langle (h\cdot\xi)|-(l\cdot\xi)\otimes \langle (l\cdot\xi)|]+2(\tilde{h}\cdot\xi)\otimes \langle (\tilde{h}\cdot\xi)|+l\otimes \langle l|,\\
\mathcal{B}_5^{ij}(\xi)&=4(h_i\otimes\langle h_j|+|\xi|^2\tilde{h}_i\otimes\langle \tilde{h}_j|),\\
\mathcal{B}_6^{ij}(\xi)&=4i(h_i\otimes\langle \tilde{h}_j|+ \tilde{h}_i\otimes\langle h_j|),
\emas
where $g_0,g_1$, $h=(h_1,h_2,h_3)$ and $\tilde{h}=(\tilde{h}_1,\tilde{h}_2,\tilde{h}_3)$ are defined by \eqref{ghl}, and $l=(l_1,l_2,l_3)$ with
$$l_j =c_2(|\xi|^2)\chi_j+\i c_2(|\xi|^2)(L-\lambda_{2}(|\xi|)-\i P_1\tilde{v}\cdot\xi P_1)^{-1}P_1(\tilde{v}\cdot\xi)\chi_j, \ \ j=1,2,3.$$
\end{lem}

To establish the pointwise space-time behaviors of $G^{j}_{L,0}(t,x)~(j=1,2,3,4)$,  we give the following prepared lemmas \ref{rbegf3}--\ref{rbegf4j1}. 
\begin{lem}$\mathrm{(Kirchhoff)}$.\label{rbegf3}
Let $w(t,x)$ be a function given by its 3-D Fourier transformation:
$$
\hat{w}=\frac{\sin(c|\xi|t)}{c|\xi|}, \quad \hat{w}_t=\cos(c|\xi|t).
$$
Then, it holds for any functions $g(x)$ and $h(x)$ that
\bmas
&w\ast g(x)=\frac{t}{4\pi}\iint_{|y|=1}g(x+cty)dS_y,\\
&w_t\ast h(x)=\frac{1}{4\pi}\iint_{|y|=1}h(x+cty)dS_y+\frac{ct}{4\pi}\iint_{|y|=1}\nabla h(x+cty)\cdot ydS_y.
\emas
\end{lem}
\begin{lem}\label{rbegf4}
For any positive integer $l$, we have
\bma
\left|w\ast(1+t)^{-\frac{l}{2}} e^{-\frac{|x|^2}{C(1+t)}} \right|&\leq C(1+t)^{-\frac{l}{2}} e^{-\frac{(|x|-ct)^2}{3C(1+t)}} ,\label{rbegf4-1}\\
\left|w_t\ast(1+t)^{-\frac{l}{2}} e^{-\frac{|x|^2}{C(1+t)}} \right|&\leq C(1+t)^{-\frac{l+1}{2}} e^{-\frac{(|x|-ct)^2}{4C(1+t)}} \label{rbegf4-2}.
\ema
\end{lem}
\begin{proof}
First, we prove \eqref{rbegf4-1} as follows. By Lemma \ref{rbegf3}, we have
\be
I_1 =\frac{t}{4\pi}\iint_{|y|=1}e^{-\frac{|x+cty|^2}{C(1+t)}}dS_y.\label{rbegf4-1-1}
\ee
Thus,
\bma
I_1&=\frac{t}{4\pi}\(\iint_{\{|y|=1,\cos\theta>0\}}+\iint_{\{|y|=1,\cos\theta<0\}}\)e^{-\frac{|x|^2+c^2t^2|y|^2+2ct|x||y|\cos\theta}{C(1+t)}}dS_y\nnm\\
&=:I_1^1+I_1^2.\label{rbegf4-1-2}
\ema
It is very easy to verify that $I_1^1<I_1^2$. We only estimate $I_1^2$ as follows. For $\cos\theta<0$,
\bmas
&\quad|x|^2+c^2t^2|y|^2+2ct|x||y|\cos\theta\\
&=(|x|-ct|y|)^2+2c^2t^2|y|^2(1+\cos\theta)+2ct|y|(|x|-ct|y|)(1+\cos\theta)\\
&\geq \frac{(|x|-ct)^2}{3}+\frac{c^2t^2(1+\cos\theta)}{2},
\emas
where we have used
$$
2ct|y|(|x|-ct|y|)(1+\cos\theta)\leq\frac{2(|x|-ct|y|)^2(1+\cos\theta)}{3}+\frac{3c^2t^2|y|^2(1+\cos\theta)}{2}.
$$
Thus,
\be
I_1^{2}\leq\frac{t}{2}e^{-\frac{(|x|-ct)^2}{3C(1+t)}}\int_{\frac{\pi}{2}}^{\pi}e^{-\frac{c^2t^2(1+\cos\theta)}{2C(1+t)}}\sin\theta d\theta
\leq Ce^{-\frac{(|x|-ct)^2}{3C(1+t)}},\label{rbegf4-1-5}
\ee
where we have used
$$
\left\{\bln
&\int_{\frac{\pi}{2}}^{\pi}e^{-\frac{c^2t^2(1+\cos\theta)}{2C(1+t)}}\sin\theta d\theta\leq C\leq C(1+t)^{-1},\ \ \ \ \ \ \ \ \ \ \ \ \ \ \ \ t\leq1,\\
&\int_{\frac{\pi}{2}}^{\pi}e^{-\frac{c^2t^2(1+\cos\theta)}{2C(1+t)}}\sin\theta d\theta\leq C(1+t)t^{-2}\leq C(1+t)^{-1},\ \ \ t\geq1.
\eln\right.
$$
By combining \eqref{rbegf4-1-1}--\eqref{rbegf4-1-5}, we obtain \eqref{rbegf4-1}.

Finally, we prove \eqref{rbegf4-2}. By Lemma \ref{rbegf3}, we have
\bma
I_2&=\frac{1}{4\pi}\iint_{|y|=1}e^{-\frac{|x+cty|^2}{C(1+t)}}dS_y+\frac{ct}{4\pi}\iint_{|y|=1}\nabla e^{-\frac{|x+cty|^2}{C(1+t)}}\cdot ydS_y\nnm\\
&=:I_2^1+I_2^2.\label{rbegf4-2-1}
\ema
By \eqref{rbegf4-1-1}--\eqref{rbegf4-1-5}, we have
\be
I_2^1\leq C(1+t)^{-1}e^{-\frac{(|x|-ct)^2}{3C(1+t)}}.\label{rbegf4-2-2}
\ee

For $I_2^2$, we obtain
\bmas
|I_2^2|&=\left| \frac{ct}{2\pi\sqrt{1+t}}\iint_{|y|=1} e^{-\frac{|x+cty|^2}{C(1+t)}}\frac{(x+cty)\cdot y}{C\sqrt{1+t}}dS_y\right|\nnm\\
&\leq C \frac{t}{\sqrt{1+t}}\iint_{|y|=1}e^{-\frac{3|x+cty|^2}{4C(1+t)}}dS_y.
\emas
By a similar argument as \eqref{rbegf4-1-1}, we have
\be
|I_2^2|\leq C(1+t)^{-\frac{1}{2}}e^{-\frac{(|x|-ct)^2}{4C(1+t)}}.\label{rbegf4-2-4}
\ee
By combining \eqref{rbegf4-2-1}, \eqref{rbegf4-2-2} and \eqref{rbegf4-2-4}, we obtain \eqref{rbegf4-2}. 
\end{proof}

\begin{lem}\label{rbegf4j1}
For any given $c,D_1>0$, there exists a constant $C>0$ such that
\bma
&\quad\left|\frac{\pt^2}{\pt x_i\pt x_j}\intt \frac{s}{4\pi}\iint_{|y|=1} t^{-\frac32}e^{-\frac{|x+csy|^2}{D_1t}}dS_y ds\right|\nnm\\
&\leq Ct^{-\frac32}\bigg( (1+t)^{-\frac{1}{2}}e^{-\frac{(|x|-ct)^2}{2D_1t}}+\(1+\frac{|x|^2}{t}\)^{-\frac32}1_{\{|x|\le ct\}}\bigg).\label{rbegf4j1-1}
\ema
\end{lem}
\begin{proof}
We prove \eqref{rbegf4j1-1} holds for $t\leq1$ and $t\ge 1$. Let
$$
G_{ij}(t,x)=\frac{\pt^2}{\pt x_i\pt x_j}\intt \frac{s}{4\pi}\iint_{|y|=1} t^{-\frac32}e^{-\frac{|x+csy|^2}{D_1t}}dS_y ds.
$$

For $t\leq1$, we split $x$ into $|x|\leq ct$ and $|x|\geq ct$. For $|x|\le ct\leq c\sqrt{t}$, we can obtain
\bma
\left|G_{ij}(t,x)\right|
&\leq C\intt st^{-\frac{5}{2}}\iint_{|y|=1} e^{-\frac{3|x+csy|^2}{4D_1t}}dS_y ds\le C\intt t^{-\frac{5}{2}}e^{-\frac{(|x|-cs)^2}{4D_1t}}ds\nnm\\
&\leq C\intt t^{-\frac{5}{2}}ds\leq Ct^{-\frac{3}{2}}\leq Ct^{-\frac32}\(1+\frac{|x|^2}{t}\)^{-\frac32}, \label{rbegf4j1-1-2}
\ema
where we have used \eqref{rbegf4-1-5}.

For $|x|\geq ct$ and $|y|=1$, we have $|x+csy|\geq||x|-cs|\geq||x|-ct|$. Thus, for $t\le 1$,
\bma
\left|G_{ij}(t,x)\right|&\leq C\intt st^{-\frac{5}{2}}\iint_{|y|=1} e^{-\frac{3|x+csy|^2}{4D_1t}}dS_y ds\nnm\\
&\leq Ce^{-\frac{(|x|-ct)^2}{2D_1t}}\intt st^{-\frac{5}{2}}\iint_{|y|=1} e^{-\frac{|x+csy|^2}{4D_1t}}dS_y ds\nnm\\
&\leq Ce^{-\frac{(|x|-ct)^2}{2D_1t}}\intt t^{-\frac{5}{2}}e^{-\frac{(|x|-cs)^2}{12D_1t}}ds\nnm\\
&\leq Ct^{-\frac32}e^{-\frac{(|x|-ct)^2}{2D_1t}}\leq Ct^{-\frac32}(1+t)^{-\frac12}e^{-\frac{(|x|-ct)^2}{2D_1t}}.\label{rbegf4j1-1-3}
\ema

For $t\ge 1$, we split $x$ into $|x|\leq c\sqrt{t}$, $c\sqrt{t}\leq|x|\leq ct$ and $|x|\geq ct$. For $|x|\leq c\sqrt{t}$, we have by \eqref{rbegf4j1-1-2} that
\be
\left|G_{ij}(t,x)\right|\leq Ct^{-\frac32}\(1+\frac{|x|^2}{t}\)^{-\frac32}.\label{rbegf4j1-1-4}
\ee
For $c\sqrt{t}\leq|x|\leq ct$, we have
\bma
&\quad\intt \frac{s}{4\pi}\iint_{|y|=1} t^{-\frac32}e^{-\frac{|x+csy|^2}{D_1t}}dS_y ds\nnm\\
&=\intt \frac{s}{4\pi}t^{-\frac32}\int^{2\pi}_0\int^{\pi}_0e^{-\frac{|x|^2+c^2s^2|y|^2+2cs|x||y|\cos\theta}{D_1t}}\sin\theta d\theta d \varphi ds\nnm\\
&=\frac{D_1 }{4c|x|t^{\frac12}}\int_0^t \(e^{-\frac{(|x|-cs)^2}{D_1t}}-e^{-\frac{(|x|+cs)^2}{D_1t}}\)ds\nnm\\
&=\frac{D_1}{4c^2|x|t^{\frac{1}{2}}}\(\int_{|x|-ct}^{|x|}e^{-\frac{z^2}{D_1t}}dz-\int^{|x|+ct}_{|x|}e^{-\frac{z^2}{D_1t}}dz\).\label{rbegf4j1-1-5}
\ema
Thus,
\bma
G_{ij}(t,x)
&=C\frac{\pt^2}{\pt x_i\pt x_j}\frac{1}{|x|t^{\frac{1}{2}}}\(\int_{0}^{|x|}+\int^{0}_{|x|-ct}-\int^{|x|+ct}_{0}+\int_{0}^{|x|}\)e^{-\frac{z^2}{D_1t}}dz\nnm\\
&=C\frac{\pt^2}{\pt x_i\pt x_j}\frac{1}{|x|t^{\frac{1}{2}}}\(2\int_{0}^{|x|}e^{-\frac{z^2}{D_1t}}dz-\int^{|x|+ct}_{ct-|x|}e^{-\frac{z^2}{D_1t}}dz\) \nnm\\
&=:J_1+J_2.\label{rbegf4j1-1-6}
\ema

Then, we estimate $J_i$ with $i=1,2$ as follows. For $J_1$, we have by \eqref{rbegf4j1-1-6} that
$$
J_1=C\(-\frac{\delta_{ij}}{|x|^3\sqrt{t}}+\frac{3x_ix_j}{|x|^5\sqrt{t}}\)\int_0^{|x|}e^{-\frac{z^2}{D_1t}}dz+C\(\frac{\delta_{ij}}{|x|^2\sqrt{t}}-\frac{3x_ix_j}{|x|^4\sqrt{t}}\)e^{-\frac{|x|^2}{D_1t}},
$$
which gives
\be
|J_1|\leq\frac{C}{|x|^3}\leq Ct^{-\frac{3}{2}}\(1+\frac{|x|^2}{t}\)^{-\frac{3}{2}}.\label{rbegf4j1-1-7}
\ee
For $J_2$, we have by \eqref{rbegf4j1-1-6} that
\bmas
J_2
&=C\(-\frac{\delta_{ij}}{|x|^3\sqrt{t}}+\frac{3x_ix_j}{|x|^5\sqrt{t}}\)\int^{|x|+ct}_{ct-|x|}e^{-\frac{z^2}{D_1t}}dz\\
&\quad-C\(\frac{\delta_{ij}}{|x|^2\sqrt{t}}-\frac{2x_ix_j}{|x|^4\sqrt{t}}\)\(e^{-\frac{(|x|-ct)^2}{D_1t}}+e^{-\frac{(|x|+ct)^2}{D_1t}}\)\\
&\quad-C\(\frac{2x_ix_j}{D_1t|x|^3}\frac{(|x|+ct)}{\sqrt{t}}e^{-\frac{(|x|+ct)^2}{D_1t}}+\frac{2x_ix_j}{D_1t|x|^3}\frac{(|x|-ct)}{\sqrt{t}}e^{-\frac{(|x|+ct)^2}{D_1t}}\).
\emas
Thus,
\bma
|J_2|&\leq \frac{C}{|x|^3\sqrt{t}}\int^{|x|+ct}_{ct-|x|}e^{-\frac{z^2}{D_1t}}dz+C\(\frac{1}{|x|^2\sqrt{t}}+\frac{1}{|x|t}\)e^{-\frac{(|x|-ct)^2}{2D_1t}}\nnm\\
&\leq Ct^{-\frac{3}{2}}\(1+\frac{|x|^2}{t}\)^{-\frac{3}{2}}+Ct^{-\frac{3}{2}}(1+t)^{-\frac{1}{2}}e^{-\frac{(|x|-ct)^2}{2D_1t}},\label{rbegf4j1-1-8}
\ema
where we have used
 $$
\left\{\bln
&\(\frac{1}{|x|^2\sqrt{t}}+\frac{1}{|x|t}\)e^{-\frac{3(|x|-ct)^2}{4D_1t}}\leq Ct^{-\frac{3}{2}}e^{-\frac{c^2t^2}{16D_1t}}e^{-\frac{(|x|-ct)^2}{2D_1t}} ,\ \ \ \ \ |x|\leq\frac{ct}{2},\\
&\(\frac{1}{|x|^2\sqrt{t}}+\frac{1}{|x|t}\)e^{-\frac{3(|x|-ct)^2}{4D_1t}}\leq C\(t^{-\frac{5}{2}}+t^{-2}\)e^{-\frac{(|x|-ct)^2}{2D_1t}} , \ \ |x|\geq\frac{ct}{2}.
\eln\right.
$$

For $|x|>ct$ and $|y|=1$, we have $|x+csy|\geq||x|-cs|\geq||x|-ct|$. Thus, by \eqref{rbegf4j1-1-3}, we can obtain
\bma
|G_{ij}(t,x)|
&\leq Ce^{-\frac{(|x|-ct)^2}{2D_1t}}\intt t^{-\frac52}e^{-\frac{(|x|-cs)^2}{12D_1t}}ds\nnm\\
&\leq Ct^{-2} e^{-\frac{(|x|-ct)^2}{2D_1t}}.\label{rbegf4j1-1-9}
\ema
By combining \eqref{rbegf4j1-1-2}--\eqref{rbegf4j1-1-9}, we have \eqref{rbegf4j1-1}. The proof of the Lemma is completed.
\end{proof}

With the help of Lemmas \ref{rbegf2}--\ref{rbegf4j1}, we are able to establish the pointwise space-time behaviors of $G^{j}_{L,0}(t,x)$~$(j=1,2,3,4)$ as follows.
\begin{lem}\label{rbegf6}
For any given $N>\mathbf{c}$ and $\alpha\in\N^3$, there exist $C,D>0$ such that for $|x|\leq Nt$, the following estimates hold.

(1) For $G^{1}_{L,0}(t,x)$, we have
\bq
\left\{\bln
&\|\partial_x^{\alpha}P_0^lG^{1}_{L,0}(t,x)\|\leq C(1+t)^{-\frac{4+|\alpha|}{2}}e^{-\frac{(|x|-\mathbf{c}t)^2}{D(1+t)}}+Ce^{-\frac{t}{D}},\ \  l=1,2,3,\\
&\|\partial_x^{\alpha}P_1G^{1}_{L,0}(t,x)\|,\,\|\partial_x^{\alpha}G^{1}_{L,0}(t,x)P_1\|\leq C(1+t)^{-\frac{5+|\alpha|}{2}}e^{-\frac{(|x|-\mathbf{c}t)^2}{D(1+t)}}+Ce^{-\frac{t}{D}},\\
&\|\partial_x^{\alpha}P_1G^{1}_{L,0}(t,x)P_1\|\leq C(1+t)^{-\frac{6+|\alpha|}{2}}e^{-\frac{(|x|-\mathbf{c}t)^2}{D(1+t)}}+Ce^{-\frac{t}{D}}.\label{g1l0}
\eln\right.
\eq

(2) For $G^{k}_{L,0}(t,x)$ with  $k=2,3$, we have
\bq
\left\{\bln
&\|\partial_x^{\alpha}P_0^lG^{k}_{L,0}(t,x)\|\leq C(1+t)^{-\frac{3+|\alpha|}{2}}e^{-\frac{|x|^2}{D(1+t)}}+Ce^{-\frac{t}{D}},\ \  l=1,2,3,\\
&\|\partial_x^{\alpha}P_1G^{k}_{L,0}(t,x)\|,\,\|\partial_x^{\alpha}G^{k}_{L,0}(t,x)P_1\|\leq C(1+t)^{-\frac{4+|\alpha|}{2}}e^{-\frac{|x|^2}{D(1+t)}}+Ce^{-\frac{t}{D}},\\
&\|\partial_x^{\alpha}P_1G^{k}_{L,0}(t,x)P_1\|\leq C(1+t)^{-\frac{5+|\alpha|}{2}}e^{-\frac{|x|^2}{D(1+t)}}+Ce^{-\frac{t}{D}}.\label{g2-3l0}
\eln\right.
\eq

(3) For $G^{4}_{L,0}(t,x)$, we have  $P^1_0G^{4}_{L,0}=P^3_0G^{4}_{L,0}=0$, and
\bq\label{g5l0}
\left\{\bln
&\|\partial_x^{\alpha}P^2_0G^{4}_{L,0}(t,x)\|\leq C(1+t)^{-\frac{3+|\alpha|}{2}}\((1+t)^{-\frac{1}{2}}e^{-\frac{(|x|-\mathbf{c}t)^2}{D(1+t)}}+e^{-\frac{|x|^2}{C(1+t)}}\)+Ce^{-\frac{t}{D}}\\
&~~~~~~~~~~~~~~~~~~~~~~~~~+C(1+t)^{-\frac{3+|\alpha|}{2}}B_{\frac{3}{2}}(t,|x|)1_{\{|x|\leq\mathbf{c}t\}},\\
&\|\partial_x^{\alpha}P_1G^{4}_{L,0}(t,x)\|,\,\|\partial_x^{\alpha}G^{4}_{L,0}(t,x)P_1\|\leq C(1+t)^{-\frac{4+|\alpha|}{2}}\((1+t)^{-\frac{1}{2}}e^{-\frac{(|x|-\mathbf{c}t)^2}{D(1+t)}}+e^{-\frac{|x|^2}{D(1+t)}}\)\\
&~~~~~~~~~~~~~~~~~~~~~~~~~+Ce^{-\frac{t}{D}}+C(1+t)^{-\frac{4+|\alpha|}{2}}B_{\frac{3}{2}}(t,|x|)1_{\{|x|\leq\mathbf{c}t\}},\\
&\|\partial_x^{\alpha}P_1G^{4}_{L,0}(t,x)P_1\|\leq C(1+t)^{-\frac{5+|\alpha|}{2}}\((1+t)^{-\frac{1}{2}}e^{-\frac{(|x|-\mathbf{c}t)^2}{D(1+t)}}+e^{-\frac{|x|^2}{D(1+t)}}\)+Ce^{-\frac{t}{D}}\\
&~~~~~~~~~~~~~~~~~~~~~~~~~~~+C(1+t)^{-\frac{5+|\alpha|}{2}}B_{\frac{3}{2}}(t,|x|)1_{\{|x|\leq\mathbf{c}t\}}.
\eln\right.
\eq
\end{lem}
\begin{proof}
By \eqref{SA6}, we have $A^1_1(|\xi|^2)=\mathbf{c}+|\xi|^2\mathcal{A}(|\xi|^2)$, where $\mathcal{A}(|\xi|^2)$ is analytic in $\xi$ for $|\xi|\leq r_0$. Thus,
\bma
\cos(|\xi|A^1_1(|\xi|^2)t)&=\cos(\mathbf{c}|\xi|t)\cos(|\xi|^3\mathcal{A}(|\xi|^2)t)- \sin(\mathbf{c}|\xi|t)\sin(|\xi|^3\mathcal{A}(|\xi|^2)t),\label{hgfcos1}\\
\sin(|\xi|A^1_1(|\xi|^2)t)&=\sin(\mathbf{c}|\xi|t)\cos(|\xi|^3\mathcal{A}(|\xi|^2)t)+ \cos(\mathbf{c}|\xi|t)\sin(|\xi|^3\mathcal{A}(|\xi|^2)t),\label{hgfsin1}
\ema
where $\cos(|\xi|^3\mathcal{A}(|\xi|^2)t)$ and $\frac{\sin(|\xi|^3\mathcal{A}(|\xi|^2)t)}{|\xi|}$ are analytic in $\xi$ for $|\xi|\leq r_0$. These together with  \eqref{Hgf1} imply that
\be
\partial_x^{\alpha}G^{1}_{L,0}(t,x)=w_t\ast \partial_x^{\alpha}I_1(t,x)+w\ast \partial_x^{\alpha}I_2(t,x),\label{g1l0w1}
\ee
where
\bmas
I_1&=\frac{1}{(2\pi)^{\frac{3}{2}}}\int_{\{|\xi|\leq\frac{r_0}{2}\}}e^{\i x\cdot\xi}e^{-A^2_1(|\xi|^2)t} \(\cos(|\xi|^3\mathcal{A}(|\xi|^2)t)\mathcal{B}_1(\xi)+\frac{\sin(|\xi|^3\mathcal{A}(|\xi|^2)t)}{|\xi|}\mathcal{B}_2(\xi)\)d\xi, \\
I_2&=\frac{\mathbf{c}}{(2\pi)^{\frac{3}{2}}}\int_{\{|\xi|\leq\frac{r_0}{2}\}}e^{\i x\cdot\xi}e^{-A^2_1(|\xi|^2)t} \(\cos(|\xi|^3\mathcal{A}(|\xi|^2)t)\mathcal{B}_2(\xi)-\sin(|\xi|^3\mathcal{A}(|\xi|^2)t)|\xi|\mathcal{B}_1(\xi)\)d\xi.
\emas

First, we estimate the pointwise space-time behavior of $\partial_x^{\alpha}I_1(t,x)$ for $|x|\leq Nt$. Denote
$$
\mathbb{B}=\[-\frac{r_0}{4},\frac{r_0}{4}\]\times\[-\frac{r_0}{4},\frac{r_0}{4}\]\times\[-\frac{r_0}{4},\frac{r_0}{4}\],
$$
and
$$
\left\{\bln
&\Gamma_{1}=\Big\{\xi_1\in\C\,|\, \mathrm{Re}(\xi_1)=-\frac{r_{0}}{4},\, 0\leq\mathrm{Im}(\xi_1)\leq\frac{|x|}{C_1t}\Big\},\\
&\Gamma_{2}=\Big\{\xi_1\in\C\,|\, \mathrm{Im}(\xi_1)=\frac{|x|}{C_1t},\, -\frac{r_{0}}{4}\leq\mathrm{Re}(\xi_1) \leq\frac{r_{0}}{4}\Big\},\\
&\Gamma_{3}=\Big\{\xi_1\in\C\,|\, \mathrm{Re}(\xi_1)=\frac{r_{0}}{4},\, 0\leq\mathrm{Im}(\xi_1)\leq\frac{|x|}{C_1t}\Big\}.
\eln\right.
$$
By changing variable $\xi\rightarrow \O\xi$ such that $\O^T x\rightarrow(|x|,0,0)$, we have
\bma
\partial_x^{\alpha}I_1(t,x)&=\frac{1}{(2\pi)^{\frac{3}{2}}}\(\int_{\mathbb{B}} +\int_{\{|\xi|\leq\frac{r_0}{2}\}\cap\mathbb{B}^{c}}\)
\i^{|\alpha|}(\O\xi)^{\alpha}e^{\i|x|\xi_1}e^{-A^2_1(|\xi|^2)t} \nnm\\
&\qquad\qquad \times\(\cos(|\xi|^3\mathcal{A}(|\xi|^2)t)\mathcal{B}_1(\O\xi)+\frac{\sin(|\xi|^3\mathcal{A}(|\xi|^2)t)}{|\xi|}\mathcal{B}_2(\O\xi)\)d\xi\nnm\\
&=:I_1^1+I_1^2.\label{gfpi1}
\ema
By taking $C_1=\max \{\frac{4N}{r_0},4A_1 \}$, we have $\Gamma_n\in \{\xi_1\in\mathbb{C}\,|\,|\xi_1|\leq\frac{r_0}{2}\}$ for $n=1,2,3$. By Cauchy theorem, we obtain
\bma
I^1_1&=\frac{1}{(2\pi)^{\frac{3}{2}}}\sum^3_{n=1}\int_{-\frac{r_0}{4}}^{\frac{r_0}{4}} \int_{-\frac{r_0}{4}}^{\frac{r_0}{4}}\int_{\Gamma_n}\i^{|\alpha|}(\O\xi)^{\alpha}e^{\i|x|\xi_1-A^2_1(|\xi|^2)t} \nnm\\
&\qquad\qquad \times\(\cos(|\xi|^3\mathcal{A}(|\xi|^2)t)\mathcal{B}_1(\O\xi)+\frac{\sin(|\xi|^3\mathcal{A}(|\xi|^2)t)}{|\xi|}\mathcal{B}_2(\O\xi)\)d\xi_1d\xi_2d\xi_3\nnm\\
&=:J^{\alpha}_1+J^{\alpha}_2+J^{\alpha}_3.\label{hpw1}
\ema
Let $u=\mathrm{Re}(\xi_1)$ and $y=\mathrm{Im}(\xi_1)$. Since
$$
|\cos(|\xi|^3\mathcal{A}(|\xi|^2)t)|,\ \Big|\frac{\sin(|\xi|^3\mathcal{A}(|\xi|^2)t)}{|\xi|}\Big|\leq Ce^{O(|\xi|^3)t},\quad |\xi|\in \Gamma_n,
$$
 it follows that
\bma
\big\|P^l_0J^{\alpha}_2(t,x)\big\|
&\leq Ce^{-(1-\frac{A_1}{C_1})\frac{|x|^2}{C_1t}}\int_{-\frac{r_0}{4}}^{\frac{r_0}{4}}\int_{-\frac{r_0}{4}}^{\frac{r_0}{4}} \int^{\frac{r_0}{4}}_{-\frac{r_0}{4}}|\xi_2^{\alpha_2}\xi_3^{\alpha_3}| e^{-A_1(\xi_2^2+\xi_3^2)t+O(1)(|\xi_2|+|\xi_3|)^3t}\nnm\\
&\qquad\qquad \times e^{-A_1u^2t+O(1)(|u|+\frac{|x|}{C_1t})^3t}\(|u|^{\alpha_1}+\Big|\frac{|x|}{C_1t}\Big|^{\alpha_1}\)dud\xi_2d\xi_3\nnm\\
&\leq C(1+t)^{-\frac{3+|\alpha|}{2}}e^{-\frac{|x|^2}{Dt}},\quad l=1,2,3,\label{hpw2}
 \ema
where $C,D>0$ are two constants.

Noting that $P_1\mathcal{B}_i(\xi)=O(|\xi|)$ and $P_1\mathcal{B}_i(\xi)P_1=O(|\xi|^2)$, $i=1,2$, we can obtain by a similar argument as  \eqref{hpw2} that
\bq
\left\{\bln
&\|P_1J^{\alpha}_2(t,x)\|,\, \|J^{\alpha}_2P_1(t,x)\|\leq C(1+t)^{-\frac{4+|\alpha|}{2}}e^{-\frac{|x|^2}{D(1+t)}},\\
&\|P_1J^{\alpha}_2(t,x)P_1\|\leq C(1+t)^{-\frac{5+|\alpha|}{2}}e^{-\frac{|x|^2}{D(1+t)}}.\label{hpw3}
\eln\right.
\eq

For $J^{\alpha}_1$ and  $J^{\alpha}_3$, we have
\bma
\big\|J^{\alpha}_n(t,x)\big\|
&\leq Ce^{-\frac{A_1r^2_0t}{4}}\int_{-\frac{r_0}{4}}^{\frac{r_0}{4}} \int_{-\frac{r_0}{4}}^{\frac{r_0}{4}}\int_{0}^{\frac{|x|}{C_1t}}|\xi_2^{\alpha_2}\xi_3^{\alpha_3}|e^{-A_1(\xi_2^2+\xi_3^2)t+O(1)(|\xi_2|+|\xi_3|)^3t}\nnm\\
&\qquad\qquad \times  e^{-(|x|-A_1yt)y-O(1)(r_0^3+|y|^3)t}\(|r_0|^{\alpha_1}+|y|^{\alpha_1}\)dyd\xi_2d\xi_3\nnm\\
&\leq Ce^{-\frac{t}{D}},\quad n=1,3.\label{hpw4}
\ema

Since $e^{-A^2_1(|\xi|^2)t}\le e^{-t/D}$ for $\frac{r_0}{4}\leq|\xi|\leq\frac{r_0}{2}$, it follows that
\be
\|I^2_1(t,x)\|\leq Ce^{-\frac{t}{D}},\quad x\in\R^3.\label{hpw6}
\ee
By combining \eqref{gfpi1}--\eqref{hpw6}, we obtain
\bq
\left\{\bln
&\|\partial_x^{\alpha}P^l_0I_1(t,x)\|\leq C(1+t)^{-\frac{3+|\alpha|}{2}}e^{-\frac{|x|^2}{D(1+t)}}+Ce^{-\frac{t}{D}},\ \ l=1,2,3,\\
&\|\partial_x^{\alpha}P_1I_1(t,x)\|,\,\|\partial_x^{\alpha}I_1(t,x)P_1\|\leq C(1+t)^{-\frac{4+|\alpha|}{2}}e^{-\frac{|x|^2}{D(1+t)}}+Ce^{-\frac{t}{D}},\\
&\|\partial_x^{\alpha}P_1I_1(t,x)P_1\|\leq C(1+t)^{-\frac{5+|\alpha|}{2}}e^{-\frac{|x|^2}{D(1+t)}}+Ce^{-\frac{t}{D}}.\label{hpw7}
\eln\right.
\eq

For the term $I_2(t,x)$, noting that $\mathcal{B}_2(\xi)=O(|\xi|)$, we can obtain by the similar arguments as \eqref{hpw1}--\eqref{hpw6} that
\bq
\left\{\bln
&\|\partial_x^{\alpha}P^l_0I_2(t,x)\|\leq C(1+t)^{-\frac{4+|\alpha|}{2}}e^{-\frac{|x|^2}{D(1+t)}}+Ce^{-\frac{t}{D}},\ \ l=1,2,3,\\
&\|\partial_x^{\alpha}P_1I_2(t,x)\|,\,\|\partial_x^{\alpha}I_2(t,x)P_1\|\leq C(1+t)^{-\frac{5+|\alpha|}{2}}e^{-\frac{|x|^2}{D(1+t)}}+Ce^{-\frac{t}{D}},\\
&\|\partial_x^{\alpha}P_1I_2(t,x)P_1\|\leq C(1+t)^{-\frac{6+|\alpha|}{2}}e^{-\frac{|x|^2}{D(1+t)}}+Ce^{-\frac{t}{D}}.\label{hpw8}
\eln\right.
\eq
These together with \eqref{g1l0w1}, \eqref{hpw7}, and Lemma \ref{rbegf4}, lead to \eqref{g1l0}.
By applying the proof of \eqref{hpw2}--\eqref{hpw6} to \eqref{Hgf2-3}, we can obtain the estimates of $G^2_{L,0}(t,x)$ and $G^3_{L,0}(t,x)$ as listed in \eqref{g2-3l0}.

Finally, we estimate $G^{4}_{L,0}(t,x)$. Define
$$ \hat{R}_0(t,\xi)=\(e^{-A_1|\xi|^2t}\cos(\mathbf{c}|\xi|t)-e^{-A_2|\xi|^2t}\)\sum_{i,j=1}^3\frac{\xi_i\xi_j}{|\xi|^2}\mathcal{B}^{ij}_5(\xi) .$$
We decompose $\hat{R}_0$ into
\bma
\hat{R}_0&=\sum_{i,j=1}^3\[e^{-A_1|\xi|^2t}(\cos(\mathbf{c}|\xi|t)-1) + \(e^{-A_1|\xi|^2t}-e^{-A_2|\xi|^2t}\)\]\frac{\xi_i\xi_j}{|\xi|^2}\mathcal{B}^{ij}_5\nnm\\
&=\sum_{i,j=1}^3\intt \xi_i\xi_j\frac{\sin(\mathbf{c}|\xi|s)}{|\xi|}e^{-A_1|\xi|^2t}\mathcal{B}^{ij}_5ds+\sum_{i,j=1}^3\int_{A_1t}^{A_2t}e^{-|\xi|^2s} \xi_i\xi_j\mathcal{B}^{ij}_5 ds\nnm\\
  &=: \hat{R}^0_{0}+\hat{R}^1_{0}. \label{R0a}
\ema
By the similar arguments as \eqref{hpw2}--\eqref{hpw6} and noting that $P_0^l\mathcal{B}^{ij}_5(\xi)=0$ for $l=1,3$, we can obtain
\bma
\|\partial_x^{\alpha}P_0^2R^1_0(t,x)\|&\leq C\int_{A_1t}^{A_2t}(1+s)^{-\frac{5+|\alpha|}{2}}e^{-\frac{|x|^2}{D(1+s)}}+e^{-\frac{s}{D}}ds\nnm\\
&\leq C(1+t)^{-\frac{3+|\alpha|}{2}}e^{-\frac{|x|^2}{D(1+t)}}+Ce^{-\frac{t}{D}}. \label{R0-1}
\ema
Since
$$\hat{R}^0_{0}=\sum_{i,j=1}^3\intt \xi_i\xi_j\frac{\sin(\mathbf{c}|\xi|s)}{|\xi|}e^{-\frac{A_1}{2}|\xi|^2t}e^{-\frac{A_1}{2}|\xi|^2t}\mathcal{B}^{ij}_5(\xi)ds,$$
it follows that
$$
R^0_{0}(t,x)  =\sum_{i,j=1}^3 \intr G_{ij}(t,x-y) F_{ij}(t,y) dy,
$$
where
\bmas
G_{ij}(t,x)&= \frac{\pt^2}{\pt x_i\pt x_j}\intt \frac{s}{4\pi}\iint_{|y|=1} (2\pi A_1t)^{-\frac32}e^{-\frac{|x+\mathbf{c}sy|^2}{2A_1t}}dS_y ds,
\\
F_{ij}(t,x)&=\frac{1}{(2\pi)^{\frac{3}{2}}}\int_{\{|\xi|\le \frac{r_0}{2}\}} e^{\i x\cdot\xi} e^{-\frac{A_1}{2}|\xi|^2t}\mathcal{B}^{ij}_5(\xi)d\xi.
\emas
By Lemma \ref{rbegf4j1}, we have
$$
|G_{ij}(t,x)|
\le Ct^{-\frac32}\bigg((1+t)^{-\frac12}e^{-\frac{(|x|-\mathbf{c}t)^2}{4A_1t}}+\(1+\frac{|x|^2}{t}\)^{-\frac32}1_{\{|x|\le \mathbf{c}t\}}\bigg).
$$
Then, we apply a similar argument as Lemma \ref{rbepw2} to obtain
\bma
\|\dxa P_0^2 R^0_{0}(t,x)\| &\le C\intr  t^{-\frac32}\bigg((1+t)^{-\frac12}e^{-\frac{(|x-y|-\mathbf{c}t)^2}{4A_1t}}+\(1+\frac{|x-y|^2}{t}\)^{-\frac32}1_{\{|x-y|\le \mathbf{c}t\}}
\bigg) \nnm\\
&\qquad \times \((1+t)^{-\frac{3+|\alpha|}2}e^{-\frac{|y|^2}{Dt}}+e^{ -\frac{t}{D} }\) dy \nnm\\
&\le C(1+t)^{-\frac{3+|\alpha|}2}\bigg((1+t)^{-\frac12}e^{-\frac{(|x|-\mathbf{c}t)^2}{2D(1+t)}} +B_{\frac32}(t,|x|)1_{\{|x|\le \mathbf{c}t\}}\bigg) +Ce^{-\frac{t}{2D}}. \label{R0-0}
\ema
Thus, it follows from \eqref{R0a}, \eqref{R0-1} and \eqref{R0-0} that
 \bma
\|\dxa P_0^2 R_{0}(t,x)\|
&\leq C(1+t)^{-\frac{3+|\alpha|}{2}}\((1+t)^{-\frac{1}{2}}e^{-\frac{(|x|-\mathbf{c}t)^2}{D(1+t)}}+e^{-\frac{|x|^2}{D(1+t)}}\)+Ce^{-\frac{t}{D}}
\nnm\\
&\quad+C(1+t)^{-\frac{3+|\alpha|}{2}}B_{\frac{3}{2}}(t,|x|)1_{\{|x|\leq\mathbf{c}t\}}. \label{R0}
\ema
Similarly,
\bq \label{R01}
\left\{\bln
&\|\partial_x^{\alpha}P_1R_0(t,x)\|,~\|\partial_x^{\alpha}R_0(t,x)P_1\|\leq C(1+t)^{-\frac{4+|\alpha|}{2}}\((1+t)^{-\frac{1}{2}}e^{-\frac{(|x|-\mathbf{c}t)^2}{D(1+t)}}+e^{-\frac{|x|^2}{D(1+t)}}\)\\
&~~~~~~~~~~~~~~~~~~~~~~~~~+Ce^{-\frac{t}{D}}+C(1+t)^{-\frac{4+|\alpha|}{2}}B_{\frac{3}{2}}(t,|x|)1_{\{|x|\leq\mathbf{c}t\}},\\
&\|\partial_x^{\alpha}P_1R_0(t,x)P_1\|\leq C(1+t)^{-\frac{5+|\alpha|}{2}}\((1+t)^{-\frac{1}{2}}e^{-\frac{(|x|-\mathbf{c}t)^2}{D(1+t)}}+e^{-\frac{|x|^2}{D(1+t)}}\)+Ce^{-\frac{t}{D}}\\
&~~~~~~~~~~~~~~~~~~~~~~~~~~~+C(1+t)^{-\frac{5+|\alpha|}{2}}B_{\frac{3}{2}}(t,|x|)1_{\{|x|\leq\mathbf{c}t\}}.
\eln\right.
\eq

By \eqref{SA6}, we have
$$A^2_1(|\xi|^2)=A_1|\xi|^2+|\xi|^4\mathcal{A}_1(|\xi|^2),\quad A^2_2(|\xi|^2)=A_2|\xi|^2+|\xi|^4\mathcal{A}_2(|\xi|^2),$$
where $\mathcal{A}_j(|\xi|^2),\ j=1,2$ are analytic in $\xi$ for $|\xi|\leq r_0$. Thus,
\bma
&\quad(\hat{G}^{4}_{L,0}-\hat{R}_0)(t,\xi)\nnm\\
&=\sum_{i,j=1}^3\xi_i\xi_j \cos(\mathbf{c}|\xi|t)e^{-A_1^2(|\xi|^2)t}\[\frac{\cos(|\xi|^3\mathcal{A}(|\xi|^2)t)-1}{|\xi|^2} \mathcal{B}^{ij}_5(\xi)+\frac{\sin(|\xi|^3\mathcal{A}(|\xi|^2)t)}{|\xi|}\mathcal{B}^{ij}_6(\xi)\] \nnm\\
&\quad+\sum_{i,j=1}^3\xi_i\xi_j \frac{\sin(\mathbf{c}|\xi|t)}{|\xi|}e^{-A^2_1(|\xi|^2)t}\[\cos(|\xi|^3\mathcal{A}(|\xi|^2)t)\mathcal{B}^{ij}_6(\xi) -\frac{\sin(|\xi|^3\mathcal{A}(|\xi|^2)t)}{|\xi|}\mathcal{B}^{ij}_5(\xi)\]\nnm\\
&\quad+ \sum_{i,j=1}^3\xi_i\xi_j \cos(\mathbf{c}|\xi|t)e^{-A_1|\xi|^2t}\frac{e^{-|\xi|^4\mathcal{A}_1(|\xi|^2)t}-1}{|\xi|^2}\mathcal{B}^{ij}_5(\xi) \nnm\\
&\quad+ \sum_{i,j=1}^3 \xi_i\xi_j e^{-A_2|\xi|^2t}\frac{e^{-|\xi|^4\mathcal{A}_2(|\xi|^2)t}-1}{|\xi|^2} \mathcal{B}^{ij}_5(\xi),     \label{rbegfgl05}
\ema
where $\frac{e^{-|\xi|^4\mathcal{A}_j(|\xi|^2)t}-1}{|\xi|^2}$, $j=1,2$, and $\frac{\cos(|\xi|^3\mathcal{A}(|\xi|^2)t)-1}{|\xi|^2}$ are analytic functions in $\xi$ and satisfy
$$
\bigg|\frac{e^{-|\xi|^4\mathcal{A}_j(|\xi|^2)t}-1}{|\xi|^2}\bigg| ,\,
\bigg|\frac{\cos(|\xi|^3\mathcal{A}(|\xi|^2)t)-1}{|\xi|^2}\bigg| \le C|\xi|^2 te^{O(|\xi|^4)t}, \quad |\xi|\in \Gamma_n.
$$
Then, it follows from \eqref{rbegfgl05} that
$$
\partial^{\alpha}_x(G^{4}_{L,0}-R_0)(t,x)= w\ast \dxa I_3+w_t\ast \dxa I_4+\dxa I_5 ,
$$
where
\bmas
I_3&=\frac{\mathbf{c}}{(2\pi)^{\frac{3}{2}}} \sum^3_{i,j=1}\int_{\{|\xi|\leq\frac{r_0}{2}\}}
 e^{\i x\cdot\xi}\xi_i\xi_je^{-A^2_1(|\xi|^2)t}\\
&\qquad \times \(\cos(|\xi|^3\mathcal{A}(|\xi|^2)t)\mathcal{B}^{ij}_6(\xi) -\frac{\sin(|\xi|^3\mathcal{A}(|\xi|^2)t)}{|\xi|}\mathcal{B}^{ij}_5(\xi)\)d\xi,
\\
I_4&=\frac{1}{(2\pi)^{\frac{3}{2}}} \sum^3_{i,j=1}\int_{\{|\xi|\leq\frac{r_0}{2}\}}
 e^{\i x\cdot\xi}\xi_i\xi_j\bigg[e^{-A_1|\xi|^2t}\frac{e^{-|\xi|^4\mathcal{A}_1(|\xi|^2)t}-1}{|\xi|^2}\mathcal{B}^{ij}_5(\xi)\\
&\qquad+ e^{-A_1^2(|\xi|^2)t}\bigg(\frac{\cos(|\xi|^3\mathcal{A}(|\xi|^2)t)-1}{|\xi|^2} \mathcal{B}^{ij}_5(\xi)+
\frac{\sin(|\xi|^3\mathcal{A}(|\xi|^2)t)}{|\xi|}\mathcal{B}^{ij}_6(\xi)\bigg)\bigg]d\xi,
\\
I_5&=\frac{1}{(2\pi)^{\frac{3}{2}}} \sum^3_{i,j=1}\int_{\{|\xi|\leq\frac{r_0}{2}\}}
 e^{\i x\cdot\xi}\xi_i\xi_je^{-A_2|\xi|^2t}\frac{e^{-|\xi|^4\mathcal{A}_2(|\xi|^2)t}-1}{|\xi|^2} \mathcal{B}^{ij}_5(\xi)d\xi.
\emas
Repeating the proof of \eqref{hpw2}--\eqref{hpw6}, we obtain
$$
\|\dxa I_j(t,x)\| \le C\((1+t)^{-\frac{5+|\alpha|}2}e^{-\frac{|x|^2}{D(1+t)}}+e^{-\frac{t}{D} }\), \quad j=3,4,5,
$$
and hence
\be
\|\dxa (G^{4}_{L,0}-R_0)(t,x)\| \le C\((1+t)^{-\frac{6+|\alpha|}2}e^{-\frac{(|x|-\mathbf{c}t)^2}{D(1+t)}}+(1+t)^{-\frac{5+|\alpha|}2}e^{-\frac{|x|^2}{D(1+t)}}+e^{-\frac{t}{D} }\). \label{R02}
\ee
By combining \eqref{R0}, \eqref{R01} and \eqref{R02}, we can obtain  \eqref{g5l0}.
\end{proof}

\subsection{Kinetic Part}
In this subsection, we extract the singular kinetic waves from the Green's function $G$ and establish the pointwise estimate of the remainder part. Since $\hat{G}(t,\xi)$ does not belongs to $L^1(\mathbb{R}^3_{\xi})$, $G(t,x)$ can be decompose into the singular part and the remainder
bounded part. Indeed, we construct the approximate sequences of $\hat{G}$ with faster decay rate in frequency space, which is equivalent to the higher regularity in physical space, and estimate the bounded remainder term by the weighted energy method. To begin with, we define the $k$-th degree Mixture operator $\hat{\mathbb{M}}^t_k(\xi)$ by (cf. \cite{LIU-2,LIU-3})
\bq
\hat{\mathbb{M}}^t_k(\xi)=\int^t_0\int^{s_1}_0\cdot\cdot\cdot\int^{s_{k-1}}_0\hat{S}^{t-s_1}K\hat{S}^{s_{1}-s_2}\cdot\cdot\cdot K\hat{S}^{s_{k-1}-s_k}K\hat{S}^{s_k}ds_k\cdot\cdot\cdot ds_1,
\eq
where $\xi\in\mathbb{C}^3$, and $\hat{S}^t$ is a operator on $L^2(\mathbb{R}^3_v)$ defined by
$$
\hat{S}^t=e^{-(\nu(v)+\i\tilde{v}\cdot\xi)t}.
$$


\begin{lem}\label{rbegf7}
(1) There exist three positive constants $C$, $\nu_0$ and $\nu_1$ such that for $j=1,2,3$,
\bq
\nu_0\leq \nu(v)\leq \nu_1,\quad |\partial_{v_j}\nu(v)|\leq\frac{C}{v_0}.\label{nu0}
\eq

(2) There exists a constant $C>0$ such that for $j=1,2,3$,
\bq \label{nu1}
\left\{\bln
&k_1(v,u)\leq Ce^{-\frac{u_0+v_0}{2}},\quad k_2(v,u)\leq C\frac{1}{\max\{u_0,v_0\}|u-v|}e^{-\frac{|u-v|}{4}},\\
&|\partial_{v_j}k_1(v,u)|,~|\partial_{u_j}k_1(v,u)|\leq Ce^{-\frac{u_0+v_0}{3}},\\
&|\partial_{v_j}k_2(v,u)|,~|\partial_{u_j}k_2(v,u)|\leq C\frac{1}{\max\{u_0,v_0\}|u-v|}\(1+\frac{1}{|u-v|}\)e^{-\frac{|u-v|}{8}}.
\eln\right.
\eq
\end{lem}
\begin{proof}
We first prove \eqref{nu0}. Note that (cf. \cite{Glassey-2,Glassey-3})
\bgr
v_M=\frac{g\sqrt{4+g^2}}{u_0v_0}\leq4,\label{num1}\\
\frac{[|u-v|^2+|u\times v|^2]^{\frac{1}{2}}}{(u_0v_0)^{\frac{1}{2}}}\leq g\leq |u-v|.\label{rbegfkpg-3}
\egr
Thus, it follows from \eqref{L-K} and \eqref{num1} that
\bma
\nu(v)&=4\pi\int_{\R^3}\frac{g\sqrt{4+g^2}}{u_0v_0}e^{-\sqrt{1+|u|^2}}du\nnm\\
&\leq 16\pi\int_{\R^3}e^{-\sqrt{1+|u|^2}}du=:\nu_1.\label{nu2}
\ema

Since $g\sqrt{4+g^2}\geq g^2\geq\frac{|u-v|^2}{u_0v_0},$  we obtain
$$
\nu(v)\geq4\pi\int_{\R^3}\frac{|u-v|^2}{(u_0v_0)^2}e^{-\sqrt{1+|u|^2}}du.
$$
Thus, it holds that for $|v|>1$,
\bma
\nu(v)
&\geq4\pi\int_{\R^3}\frac{||v|-|u||^2}{(1+|v|^2)(1+|u|^2)}e^{-\sqrt{1+|u|^2}}du\nnm\\
&\geq\frac{16\pi}{5}\int_{\{|u|<\frac{1}{2}\}}\frac{||v|-\frac{1}{2}|^2}{(1+|v|^2)}e^{-\sqrt{1+|u|^2}}du>0,\label{nu4}
\ema
and for $|v|<1$,
\be
\nu(v)\geq2\pi\int_{\{|u|>2\}}\frac{||u|-1|^2}{(1+|u|^2)}e^{-\sqrt{1+|u|^2}}du>0.\label{nu5}
\ee

We estimate $\partial_{v_j}\nu(v)$, $j=1,2,3$ as follows.
Since
\be
\partial_{v_j}\(\frac{g\sqrt{g^2+4}}{u_0v_0}\)
 =\frac{1}{v_0}\(\frac{\sqrt{g^2+4}}{g}+\frac{g}{\sqrt{g^2+4}}\)\(\frac{v_j}{v_0}-\frac{u_j}{u_0}\)-\frac{1}{v_0}\frac{g\sqrt{g^2+4}}{u_0v_0}\frac{v_j}{v_0},\label{rbegfkppnu1-1}
\ee
it follows that
\bma
|\partial_{v_j}\nu(v)|
&\leq\frac{C}{v_0}\int_{\R^3}\(\(\frac{1}{g}+1\)\frac{|u-v|}{v_0}+1\)e^{-\sqrt{1+|u|^2}}du\nnm\\
&\leq\frac{C}{v_0}\int_{\R^3}\(\frac{(u_0v_0)^{\frac{1}{2}}}{v_0}+\frac{u_0+v_0}{v_0}+1\)e^{-\sqrt{1+|u|^2}}du\nnm\\
&\leq\frac{C}{v_0}\int_{\R^3}\((u_0)^{\frac{1}{2}}+u_0+1\)e^{-\sqrt{1+|u|^2}}du\leq\frac{C}{v_0}, \label{nu6}
\ema
where we had used \eqref{num1}, \eqref{rbegfkpg-3} and the following inequality:
\be
\left|\frac{v_j}{v_0}-\frac{u_j}{u_0}\right|
=\left|\frac{v_j-u_j}{v_0}+\frac{u_j(u_0-v_0)}{v_0u_0}\right|\le  \frac{2|u-v|}{v_0},\quad j=1,2,3. \label{rbegfk2ujvj}
\ee
By combining \eqref{nu6}, \eqref{nu2} and \eqref{nu4}--\eqref{nu5}, we prove \eqref{nu0}.

Next, we prove the estimates of $k_1$ and $\partial_{v_j}k_1$ in \eqref{nu1}. By \eqref{L-K} and \eqref{num1}, we have
$$
k_1(v,u)=\Big|\frac{4\pi g\sqrt{4+g^2}}{v_0u_0}e^{-\frac{u_0+v_0}{2}}\Big|\leq Ce^{-\frac{u_0+v_0}{2}}.
$$
For $\partial_{v_j}k_1$, $j=1,2,3,$ we have
\bmas
&\quad \partial_{v_j}k_1(v,u)\\
&=\frac{4\pi}{v_0}\(\bigg(\frac{\sqrt{g^2+4}}{g}+\frac{g}{\sqrt{g^2+4}}\bigg)\(\frac{v_j}{v_0}-\frac{u_j}{u_0}\)-\frac{g\sqrt{g^2+4}}{u_0v_0}\(\frac{v_j}{2}+\frac{v_j}{v_0}\)\)e^{-\frac{u_0+v_0}{2}},
\emas
which together with \eqref{rbegfkpg-3} and \eqref{rbegfk2ujvj} implies that
\bmas
|\partial_{v_j}k_1(v,u)|&\leq C\(\(\frac{1}{g}+1\)\frac{|u-v|}{v_0}+v_0+1\)e^{-\frac{u_0+v_0}{2}}\nnm\\
&\leq C\(\frac{(u_0v_0)^{\frac{1}{2}}}{v_0}+\frac{u_0+v_0}{v_0}+v_0+1\)e^{-\frac{u_0+v_0}{2}}\nnm\\
&\leq Ce^{-\frac{u_0+v_0}{3}}.
\emas

Finally, we prove  the estimates of $k_2$ and $\partial_{v_j}k_2$ in \eqref{nu1}. By \eqref{L-K} and \eqref{rbegfkpg-3},
we have
\bma
k_2(v,u)&=\frac{4}{u_0v_0}\(\frac{s}{|u-v|}+\frac{g\sqrt{s}(u_0+v_0)}{|u-v|^2}+\frac{2g^2(u_0+v_0)}{|u-v|^3}\)e^{-\frac{\sqrt{s}|u-v|}{2g}}\nnm\\
&\leq C\frac{1}{u_0v_0|u-v|}\(|u-v|^2+|u-v|(u_0+v_0)+u_0+v_0+1\)e^{-\frac{ |u-v|}{2}}.\label{k2-2}
\ema
Noting that
$$
u_0^2\leq 2|u-v|^2+2v_0^2,\quad v_0^2\leq 2|u-v|^2+2u_0^2,
$$
we obtain
\be
k_2(v,u)\leq C\frac{1}{\max\{u_0,v_0\}|u-v|}e^{-\frac{|u-v|}{4}}.\label{k2-3}
\ee

By  \eqref{rbegfk2ujvj}, it holds that
\bma
\partial_{v_j}\(\frac{\sqrt{s}|u-v|}{2g}\)
&=\(\frac{g^2|u-v|-s|u-v|}{2g^3\sqrt{s}}\)u_0\(\frac{v_j}{v_0}-\frac{u_j}{u_0}\)+\frac{\sqrt{s}(u_j-v_j)}{2g|u-v|}\nnm\\
&=\frac{2s^{\frac{3}{2}}|u-v|^3}{g^3}\frac{u_0}{s^2|u-v|^2}\(\frac{u_j}{u_0}-\frac{v_j}{v_0}\)+\frac{\sqrt{s}(u_j-v_j)}{2g|u-v|}\nnm\\
&\leq C\frac{s^{\frac{3}{2}}|u-v|^3}{g^3}\(1+\frac{1}{|u-v|}\)+\frac{\sqrt{s}|u-v|}{2g}\frac{1}{|u-v|},\label{k2-5}
\ema
where we have used
$$
\frac{u_0}{|u-v|v_0}\leq
C\left\{\bal
 \frac{1}{|u-v|}, & u_0\leq v_0,\\
 1+\frac{1}{|u-v|},& u_0\geq v_0.
\ea\right.
$$
By \eqref{k2-3} and \eqref{k2-5}, we can obtain
\bmas
|\partial_{v_j}k_2(v,u)|&\leq C\(\frac{s^{\frac{3}{2}}|u-v|^3}{g^3}\(1+\frac{1}{|u-v|}\)+\frac{\sqrt{s}|u-v|}{2g}\frac{1}{|u-v|}\)k_2(v,u)\nnm\\
&\leq C\(1+\frac{1}{|u-v|}\)\frac{|u-v|^2+|u-v|+1}{\max\{u_0,v_0\}|u-v|}e^{-\frac{\sqrt{s}|u-v|}{4g}}\nnm\\
&\leq C\frac{1}{\max\{u_0,v_0\}|u-v|}\(1+\frac{1}{|u-v|}\)e^{-\frac{|u-v|}{8}}.
\emas
The proof of the lemma is completed.
\end{proof}

By Lemma \ref{rbegf7}, we have the following estimate for $K$. The proof is straightforward, and we omit the detail.
\begin{lem}\label{rbegf8j1}
For any $\gamma\in\R$, it holds for $k=k_1+k_2$ that
\be
\int_{\R^3}|k(v,u)|(1+u_0)^{\gamma}du\leq C(1+v_0)^{\gamma-1}.\label{rbegf8j1-2}
\ee
\end{lem}

Denote
$$
D_{\delta}=\{\xi\in\mathbb{C}^3\,|\,|\mathrm{Im}\xi|\leq\delta\},\quad \delta\ge 0.
$$
By Lemmas \ref{rbegf7} and \ref{rbegf8j1}, we have the Mixture Lemma as follows.
\begin{lem}\label{rbegf9}
For any $k\geq1$ and $\delta\in[0,\nu_0)$, $\hat{\mathbb{M}}^t_k(\xi)$ is analytic for $\xi\in D_{\delta}$ and satisfies
\bma
\|\hat{\mathbb{M}}^t_{2}(\xi)g_0\| &\le C(1+|\xi|)^{-1}(1+t)^2e^{-(\nu_0-\delta)t}(\|v_0g_0\| +\|v_0\Tdv g_0\| ), \label{mix1}
\\
\|\hat{\mathbb{M}}^t_{3k}(\xi)g_0\|&\leq C_k(1+|\xi|)^{-k}(1+t)^{3k}e^{-(\nu_0-\delta)t}\|g_0\|  ,\label{mix2}
\ema
where $v_0g_0\in L^2(\R^3_v)$, $C$ and $C_k$ are positive constants, and $\nu_0$ is defined by \eqref{nu0}.
\end{lem}
\begin{proof}
To begin with, we prove \eqref{mix1}. Define
$$
\mu_1^t(\xi)=K\int_0^t\hat{S}^{t-s}K\hat{S}^{s}ds,\quad \xi\in\mathbb{C}^3.
$$
We claim that for any $\xi \in D_{\delta}$ and $v_0g_0\in L^2(\R^3_v)$, it holds that
\be |\xi|\|\mu^t_1(\xi)g_0\| \le Ce^{-(\nu_0-\delta)t}(\|v_0g_0\|+\|v_0\Tdv g_0\|). \label{mu1}\ee
Indeed, we split $\mu_1^t(\xi)g_0$ into two parts:
\be
K\(\int_0^{\frac{t}{2}}+\int^t_{\frac{t}{2}}\)\hat{S}^{t-s}K\hat{S}^{s}g_0ds=:I_1+I_2, \label{rbegfmix1-1}
\ee
and estimate $I_1$ and $I_2$ as follows. By writing
$$
\i\xi_1I_1(t,\xi)=\i\xi_1\int_{\R^3}\int^{\frac{t}{2}}_0\int_{\R^3}e^{-[\nu(u)+\i(\tilde{u}\cdot\xi)](t-s)-[\nu(w)+\i(\tilde{w}\cdot\xi)]s} k(v,u)k(u,w)g_0(w)dwduds,
$$
and using
$$
\i\xi_1e^{-i(\tilde{u}\cdot\xi)(t-s)}=-\frac{1}{t-s}\partial_{\tilde{u}_1}e^{-\i(\tilde{u}\cdot\xi)(t-s)}=-\sum^3_{i=1} \frac{1}{t-s}\partial_{u_i}e^{-\i(\tilde{u}\cdot\xi)(t-s)}\frac{\partial u_i}{\partial\tilde{u}_1} ,
$$
we obtain
\bmas
\i\xi_1I_1(t,\xi)&=\int_{\R^3}\int^{\frac{t}{2}}_0\int_{\R^3}e^{-[\nu(u)+\i(\tilde{u}\cdot\xi)](t-s)-[\nu(w)+\i(\tilde{w}\cdot\xi)]s}\sum^3_{i=1}\bigg(\frac{1}{t-s}\frac{\partial}{\partial u_i}\(\frac{\partial u_i}{\partial\tilde{u}_1}\)k(v,u)k(u,w)\nnm\\
&\qquad\qquad\quad -\frac{\partial\nu(u)}{\partial u_i}\frac{\partial u_i}{\partial\tilde{u}_1}k(v,u)k(u,w)+\frac{1}{t-s}\frac{\partial u_i}{\partial\tilde{u}_1}\frac{\partial(k(v,u)k(u,w))}{\partial u_i}\bigg)g_0(w)dwduds.
\emas
It holds that
\bma
\bigg|\frac{\partial u_i}{\partial \tilde{u}_1}\bigg|&=\bigg|\frac{\partial}{\partial\tilde{u}_1}\bigg(\frac{\tilde{u}_i}{\sqrt{1-|\tilde{u}|^2}}\bigg)\bigg|
=\bigg|\frac{\delta_{1i}+u_1u_i}{(1+|u|^2)(1-|\tilde{u}|^2)^{\frac{3}{2}}}\bigg|\nnm\\
&\leq\frac{1}{(1-|\tilde{u}|^2)^{\frac{3}{2}}}=(u_0)^3,\quad i=1,2,3. \label{uv}
\ema
By Lemma \ref{rbegf8j1}, we have
\bma
\|v_0^nK(u_0^kg)\|^2&\leq \int_{\R^3}\(v_0^n \int_{\R^3}  k(v,u)u^k_0 g(u) du \)^2dv \nnm\\
&\leq \int_{\R^3} v_0^{2n} \int_{\R^3}  k(v,u)u^{-2n+1}_0 du \intr  k(v,u)u_0 u^{2n+2k-2}_0 g^2(u) du dv \nnm\\
&\leq C\sup_u\(\intr  k(v,u)u_0 dv\)\int_{\R^3} u^{2n+2k-2}_0 g^2(u) du \nnm\\
& \le C\|v_0^{n+k-1}g\|^2, \quad \forall\, k,n\in \R. \label{kkk}
\ema Note that \eqref{kkk} also holds for $\pt_v k(v,u)$ or $\pt_u k(v,u)$ replacing $k(v,u)$.
Thus, by using \eqref{uv}, \eqref{kkk} and the fact that
\be
\|e^{-[\nu(v)+\i(\tilde{v}\cdot\xi)]t}g_0\| \le e^{-(\nu_0-\delta)t}\|g_0\| ,\quad \xi \in D_{\delta}, \label{st}
\ee
we obtain
\bma
|\xi_1|\|I_1(t,\xi)\|&\le Ce^{-(\nu_0-\delta)t} \|Ku_0\|\|u_0Kg_0\|+te^{-(\nu_0-\delta)t} \|Ku_0\|\|u_0Kg_0\|\nnm\\
&\quad+Ce^{-(\nu_0-\delta)t} (\|\pt_{u_1} Ku_0\|\|u^2_0Kg_0\|+\| Ku_0\|\|u^2_0\pt_{u_1} Kg_0\|)\nnm\\
&\leq C(1+t)e^{-(\nu_0-\delta)t}\|v_0g_0\|. \label{rbegfmix1-2}
\ema

By the virtue of the formula
$$
\i\xi_1I_2(t,\xi)=\i\xi_1\int_{\R^3}\int_{\frac{t}{2}}^t\int_{\R^3}e^{-[\nu(u)+\i(\tilde{u}\cdot\xi)](t-s)-[\nu(w)+\i(\tilde{w}\cdot\xi)]s} k(v,u)k(u,w)g_0(w)dwduds,
$$
and the relation
$$
\i\xi_1e^{-i(\tilde{w}\cdot\xi)s}=-\frac{1}{s}\partial_{\tilde{w}_1}e^{-\i(\tilde{w}\cdot\xi)s}=-\sum^3_{i=1}\frac{1}{s}\partial_{w_i}e^{-\i(\tilde{w}\cdot\xi)s}\frac{\partial w_i}{\partial\tilde{w}_1},
$$
we have
\bmas
\i\xi_1I_2(t,\xi)&=\int_{\R^3}\int_{\frac{t}{2}}^t\int_{\R^3}e^{-[\nu(u)+\i(\tilde{u}\cdot\xi)](t-s)-[\nu(w)+\i(\tilde{w}\cdot\xi)]s}k(v,u)\sum^3_{i=1}\bigg(\frac{1}{s}\frac{\partial}{\partial w_i}\(\frac{\partial w_i}{\partial\tilde{w}_1}\)k(u,w)g_0(w)\nnm\\
&\qquad\qquad\qquad -\frac{\partial w_i}{\partial\tilde{w}_1}\frac{\partial\nu(w)}{\partial w_i}k(u,w)g_0(w)+\frac{1}{s}\frac{\partial w_i}{\partial\tilde{w}_1}\frac{\partial(k(u,w)g_0(w))}{\partial w_i}\bigg)dwduds.
\emas
 Thus,
\bma
|\xi_1|\|I_2(t,\xi)\|&\le Ce^{-(\nu_0-\delta)t} \|Ku_0\|\|u_0^{-1}Kw_0^2g_0\|+te^{-(\nu_0-\delta)t}\|Ku_0\|\|u_0^{-1}Kw_0^2g_0\|\nnm\\
&\quad+Ce^{-(\nu_0-\delta)t} (\|Ku_0\|\|u_0^{-1}\pt_{w_1}Kw_0^3g_0\|+\|Ku_0\|\|u_0^{-1}Kw_0^3\pt_{w_1}g_0\|)\nnm\\
&\leq C(1+t)e^{-(\nu_0-\delta)t}(\|v_0g_0\|+\|v_0\nabla_vg_0\|).\label{rbegfmix1-3}
\ema
Hence, we combine \eqref{rbegfmix1-1}, \eqref{rbegfmix1-2} and \eqref{rbegfmix1-3} to obtain for $i=1$ that
\be
\bigg|\xi_iK\int_0^t\hat{S}^{t-s}K\hat{S}^sg_0ds\bigg|\leq C(1+t)e^{-(\nu_0-\delta)t}(\|v_0g_0\|+\|v_0\nabla_vg_0\|),\quad \xi\in D_{\delta}.\label{rbegfmix1-4}
\ee
Similarly, we can show that \eqref{rbegfmix1-4} is also valid for $i=2,3,$ and \eqref{mu1} is proved.

By \eqref{mu1}, we have
\bma
|\xi|\|\hat{\mathbb{M}}_2^t(\xi)g_0\|&=\bigg\|\int_0^t\hat{S}^{t-s}|\xi|\mu_1^{s}(\xi)g_0ds\bigg\| \nnm\\
&\leq C\int^t_0e^{-(\nu_0-\delta)(t-s)}(1+s)e^{-(\nu_0-\delta)s}ds(\|v_0g_0\|+\|v_0\nabla_vg_0\|)\nnm\\
&\leq C(1+t)^2e^{-(\nu_0-\delta)t}(\|v_0g_0\|+\|v_0\nabla_vg_0\|), \label{rbegfmix1-6}
\ema
which proves \eqref{mix1}.

Finally, we prove \eqref{mix2}. By \eqref{rbegfmix1-6},
we have
\bma
|\xi|\|\hat{\mathbb{M}}_3^t(\xi)g_0\|&=\bigg\|\int_0^t  |\xi|\M^{t-s}_2(\xi)K\hat{S}^{s} g_0ds \bigg\| \nnm\\
&\leq C\int_0^t (1+t-s)^2e^{-(\nu_0-\delta)(t-s)} (\|v_0K\hat{S}^{s}g_0\|+\|v_0\nabla_vK\hat{S}^{s}g_0\|)ds \nnm\\
&\leq C(1+t)^3e^{-(\nu_0-\delta)t}\|g_0\|.
\ema
Thus \eqref{mix2} holds for $k=1$.

Suppose that \eqref{mix2} holds for $k=n$. We have
\bmas
 \||\xi_1|^{n+1}\hat{\mathbb{M}}_{3n+3}^t(\xi)g_0\|&=\bigg\|\int_0^t |\xi|\M^{t-s}_2(\xi) K|\xi_1|^{n}\hat{\mathbb{M}}_{3n}^s(\xi)g_0 ds\bigg\|\\
&\leq C\int_0^t (1+t-s)^2e^{-(\nu_0-\delta)(t-s)}|\xi_1|^{n} \|\hat{\mathbb{M}}_{3n}^sg_0\|ds ,
\emas
which, with an induction hypothesis, will imply that \eqref{mix2} holds for $k=n+1$, so that it holds for any $k\geq1$. The proof of the lemma is completed.
\end{proof}

We define the approximate sequence $\hat{J}_{k}$ for the Green's function $\hat{G}$ as follows:
\bq\label{rbegh2}
\left\{\bln
&\partial_{t}\hat{J}_{0}+\i\tilde{v}\cdot\xi \hat{J}_{0}+\nu(v)\hat{J}_{0}=0,\\
&\hat{J}_{0}(0,\xi)=I_{v},
\eln\right.
\eq
and
\bq\label{rbegh3}
\left\{\bln
&\partial_{t}\hat{J}_{k}+\i\tilde{v}\cdot\xi \hat{J}_{k}+\nu(v)\hat{J}_{k}=K\hat{J}_{k-1},\\
&\hat{J}_{k}(0,\xi)=0, \quad k\ge 1.
\eln\right.
\eq

With the help of Lemma \ref{rbegf9}, we can show that the pointwise estimates of the approximate solution $J_{k}(t,x)$ as follows.
\begin{lem}\label{rbegf10}
For each $k\geq0$, $\hat{J}_{k}(t,\xi)$ is analytic for $\xi\in D_{\frac{\nu_0}{2}}$ and satisfies
\bq
\|\hat{J}_{3k}(t,\xi)\|\leq C_k(1+|\xi|)^{-k}e^{-\frac{\nu_{0}t}{4}},\label{gfhjk1}
\eq
where $C_k>0$ is a constant dependent of $k$. In particular, it holds that for $k\geq4$,
\bq
\|J_{3k}(t,x)\|\leq C_ke^{-\frac{\nu_{0}(|x|+t)}{4}}.\label{gfhjk2}
\eq
\end{lem}
\begin{proof}Since
$$
\hat{J}_{k}(t,\xi)=\mathbb{\hat{M}}^{t}_{k}(\xi)
$$
with $\mathbb{\hat{M}}^{t}_{0}=\hat{S}^{t}$,  it follows from Lemma \ref{rbegf9}  that $\hat{J}_{k}(t,\xi)$ is analytic for $\xi\in D_{\frac{\nu_0}{2}}$ and satisfies
 \eqref{gfhjk1}.

 Since $\hat{J}_{k}(t,\xi)$ is analytic for $\xi\in D_{\frac{\nu_0}{2}}$, we can obtain by Cauchy Theorem that  for $0\le 2|b|\le \nu_0$ and $k\geq4$,
\bmas
\|J_{3k}(t,x)\|&=C\bigg\|\int_{\R^3}e^{\i x\cdot\xi}\hat{J}_{3k}(t,\xi)d\xi\bigg\|=Ce^{-x\cdot b}\bigg\|\int_{\R^3}e^{\i x\cdot u}\hat{J}_{3k}(t,u+\i b)du\bigg\|\nnm\\
&\leq C_ke^{-x\cdot b}\int_{\R^3}\(1+|u+\i b|\)^{-k}e^{-\frac{\nu_{0}t}{4}}du\leq C_ke^{-\frac{\nu_{0}(|x|+t)}{4}}.
\emas
 The proof of the lemma is completed.
\end{proof}

We define
\be
W_{k}(t,x)=\sum^{3k}_{i=0}J_{i}(t,x), \quad R_{k}(t,x)=(G-W_{k})(t,x),\label{gfhjk7}
\ee
where $J_{k}(t,x)$ is given by \eqref{rbegh2} and \eqref{rbegh3}.

Thus, it follows from \eqref{rbegh2}--\eqref{rbegh3} and \eqref{gfhjk7} that $W_{k}(t,x)$ and $R_{k}(t,x)$ satisfies
\bq\label{gfhjk8}
\left\{\bln
&\partial_{t}W_{k}+\tilde{v}\cdot\nabla_x W_{k}-LW_{k}=-KJ_{3k},\\
&W_{k}(0,x)=\delta(x)I_{v},
\eln\right.
\eq
and
\bq\label{gfhjk9}
\left\{\bln
&\partial_{t}R_{k}+\tilde{v}\cdot\nabla_x R_{k}-LR_{k}=KJ_{3k},\\
&R_{k}(0,x)=0.
\eln\right.
\eq

With the help of Lemma \ref{rbegf1} and Lemma \ref{rbegf10}, we can show the pointwise estimate of the term $G_H(t,x)-W_{k}(t,x)$ as follows.
\begin{lem}\label{rbegf11}
For any given $k\geq4$, there exists a constant $C>0$ such that
\bq
\|G_H(t,x)-W_{k}(t,x)\|\leq Ce^{-\frac{\nu_2t}{6}},
\eq
where $\nu_2=\min\{\kappa_0,\frac{\nu_0}{4}\}$, and $W_k(t,x)$ is the singular kinetic wave defined by \eqref{gfhjk7}.
\end{lem}

\begin{proof}
From Lemma \ref{rbegf1},  \eqref{gfhjk1} and \eqref{gfhjk9}, it holds that
\bma
\|\hat{R}_{k}(t,\xi)\|&=\bigg\|\int^t_0\hat{G}(t-s)K\hat{J}_{3k}(s)ds\bigg\|\leq C\int^t_0  \|K\hat{J}_{3k}(s)\| ds\nnm\\
& \leq C(1+|\xi|)^{-k}.\label{gfhjk10}
\ema
By \eqref{GL-H} and \eqref{gfhjk7}, we have
$$
\hat{G}_L(t,\xi)-\hat{R}_k(t,\xi)=\hat{W}_k(t,\xi)-\hat{G}_H(t,\xi).
$$
This together with Lemma \ref{rbegf1} and \eqref{gfhjk10} implies that
\be\label{gfhjk11}
\|\hat{W}_k(t,\xi)-\hat{G}_H(t,\xi)\|\leq \|\hat{G}_L(t,\xi)\|+\|\hat{R}_k(t,\xi)\|\leq  C(1+|\xi|)^{-k}.
\ee

By Lemma \ref{rbegf1} and Lemma \ref{rbegf10}, we have
\be\label{gfhjk12}
\|\hat{W}_k(t,\xi)\|+\|\hat{G}_H(t,\xi)\|\leq Ce^{-\nu_2t},
\ee
where $\nu_2=\min\{\kappa_0,\frac{\nu_0}{4}\}$.

Thus, it follows from \eqref{gfhjk11} and \eqref{gfhjk12} that
$$
\|\hat{G}_H(t,\xi)-\hat{W}_k(t,\xi)\|\leq Ce^{-\frac{\nu_2t}{6}} (1+|\xi|)^{-\frac{5k}{6}} ,
$$
which gives
\bmas
\|G_H(t,x)-W_k(t,x)\|&=C\bigg\|\int_{\R^3}e^{\i x\cdot\xi}[\hat{G}_H(t,\xi)-\hat{W}_k(t,\xi)]d\xi\bigg\|\nnm\\
&\leq C\int_{\R^3}e^{-\frac{\nu_2t}{6}}(1+|\xi|)^{-\frac{5k}{6}}d\xi\leq Ce^{-\frac{\nu_2t}{6}}
\emas
for each $k\geq4$. The proof is completed.
\end{proof}

Next, we show the pointwise space-time estimates of the remainder term $R_k(t,x)$ outside  Mach number region. First, we give a prepared lemma.
\begin{lem}\label{rbegf8}
Assume that $\hat{V}(\xi)$ is analytic for $\xi\in D_{\delta}$ with $\delta>0$ and $|\xi|^{\alpha}|\hat{V}(\xi)|\rightarrow0$ when $|\xi|\rightarrow\infty$.
For any $ \Omega\in \S^2$ and $0\le b\le \delta$, there exists a constant $C>0$ such that
\be
\int_{\R^3}e^{2b\Omega\cdot x}|\partial_x^{\alpha}V(x)|^2dx\leq C\int_{\R^3}|u+\i b\Omega|^{2\alpha}|\hat{V}(u+\i b\Omega)|^2du. \label{vvv}
\ee
\end{lem}

\begin{proof}
Since $\hat{V}(\xi)$ is analytic for $\xi\in D_{\delta}$ with $\delta>0$ and $|\xi|^{\alpha}|\hat{V}(\xi)|\rightarrow0$ when $|\xi|\rightarrow\infty$, it follows from Cauchy Theorem that for  $0\le b\le \delta$,
\bmas
e^{b\Omega\cdot x}\partial_x^{\alpha}V(x)&=C\int_{\R^3}e^{\i x\cdot\xi+b\Omega\cdot x}(\xi)^{\alpha}\hat{V}(\xi)d\xi\\
&=C\int_{\R^3}e^{\i x\cdot u}(u+\i b\Omega)^{\alpha}\hat{V}(u+\i b\Omega)du.
\emas
By applying Parseval's equality to above, we obtain \eqref{vvv}.
\end{proof}


\begin{lem}\label{rbegf12}
Given any integer $k\geq4$, there exist constants $C,D>0$ such that for $|x|>2\mathbf{c}t$,
\bq
\|R_{k}(t,x)\|\leq Ce^{-\frac{|x|+t}{D}},\label{gfout1}
\eq
where $R_k(t,x)$ is the remainder part defined by \eqref{gfhjk7}.
\end{lem}
\begin{proof}
We use the weighted energy method. Set
\bq\label{rbeoutm1}
w=e^{\eps (|x|-Yt)},
\eq
where $0<\eps <1$ is a sufficient small constant,  and $Y>1$ is determined later. It holds that
$$
\partial_{t}w=-\eps Yw,\quad \nabla_x w=\eps\frac{x}{|x|} w.
$$

Taking the inner product between \eqref{gfhjk9} and $R_{k}w$, and integrate it over $x$, we have
\bma
&\frac12\Dt\int_{\R^3}\|R_{k}\|^{2}wdx+\frac{\eps }{2}\int_{\R^3}\((Y-\frac{x}{|x|}\cdot \tilde{v})R_k,R_k\)wdx\nnm\\
&-\int_{\R^3}(LR_{k},R_{k})wdx=\int_{\R^3}(KJ_{3k},R_{k})w dx. \label{rbeoutm2}
\ema
By \eqref{EUJ}, we have
$$
|(\tilde{v}_iP_0f,P_0f)|\leq \mathbf{c}(P_0f,P_0f),\quad i=1,2,3,
$$
which leads to
\bma
&\quad\int_{\R^3}|\(\tilde{v}_{i}R_{k},R_{k}\)|wdx \nnm\\
&\le \int_{\R^3}|(\tilde{v}_{i}P_0R_{k},P_0R_{k}) |wdx+2\int_{\R^3}|(\tilde{v}_{i}P_0R_{k},P_1R_{k})| wdx+\int_{\R^3}|(\tilde{v}_{i}P_1R_{k},P_1R_{k})|wdx \nnm\\
&\leq  \frac{5}{4}\mathbf{c}\int_{\R^3}\|P_0R_k\|^2wdx+C\int_{\R^3}(P_1R_k,P_1R_k)wdx,\quad i=1,2,3.\label{rbeoutm3}
\ema

From \eqref{rbeoutm2} and \eqref{rbeoutm3}, it holds that for $0<\epsilon\ll1$ and $Y=\frac{3}{2}\mathbf{c}$,
\be
\Dt\int_{\R^3}\|R_{k}\|^{2}wdx+\eps\int_{\R^3}\|R_k\|^2wdx\leq \frac{C}{\epsilon}\int_{\R^3}\|KJ_{3k}\|^2w dx.\label{rbeoutm4}
\ee
By Lemma \ref{rbegf8} and Lemma \ref{rbegf10}, we have for $k\geq2+|\alpha|$ and $0\leq \epsilon\leq\frac{\nu_0}{2}$ that
\bmas
\int_{\R^3}\|\dxa J_{3k}\|^{2}e^{2\epsilon|x|}dx&\leq C\int_{\R^3}\Big|u+ \i\epsilon\frac{x}{|x|}\Big|^{2|\alpha|} \Big\|\hat{J}_{3k}\Big(t,u+ \i\epsilon\frac{x}{|x|}\Big)\Big\|^{2}du\\
&\leq C\int_{\R^3}e^{-\frac{\nu_0t}{4}}\bigg(1+\Big|u+ \i\epsilon\frac{x}{|x|}\Big|\bigg)^{2|\alpha|-2k}du\nnm\\
&\leq Ce^{-\frac{\nu_0t}{4}}.
\emas
Applying Gronwall's inequality to \eqref{rbeoutm4}, we have
\be
\int_{\R^3}\|R_{k}\|^{2}wdx\le C\intt\int_{\R^3} e^{-  \eps (t-s) }\|J_{3k}\|^{2}w dx ds\le C . \label{rbeoutm5}
\ee
Similarly, it holds for $k\ge 4$ that
\bq\label{rbeoutm6}
\int_{\R^3}\|\nabla^j_x R_{k}\|^{2}wdx\le C\intt\int_{\R^3} e^{-  \eps (t-s) }\|\nabla_x^j J_{3k}\|^{2}w dx ds\leq C,\ j=1,2.
\eq
Thus, by \eqref{rbeoutm5}, \eqref{rbeoutm6} and Sobolev's embedding theorem, we have
\bq\label{rbeoutm7}
e^{\frac{\eps (|x|-Yt)}{2}}\|R_k(t,x)\|\leq C.
\eq
For $|x|>2\mathbf{c}t$,
\be
|x|-\frac{3}{2}\mathbf{c}t=\frac{|x|}{6}+\frac{5|x|}{6}-\frac{9\mathbf{c}t}{6}
>\frac{|x|}{6}+\frac{\mathbf{c}t}{6}. \label{rbeoutm8}
\ee
By \eqref{rbeoutm7} and \eqref{rbeoutm8}, we prove \eqref{gfout1}.
\end{proof}


With the help of Lemma \ref{rbegf1}, Lemma \ref{rbegf6}, Lemma \ref{rbegf11} and Lemma \ref{rbegf12}, we can prove Theorem \ref{rbeth1} as follows.

\begin{proof}[\textbf{Proof of Theorem \ref{rbeth1}}]
By \eqref{GL-H}--\eqref{GL-H2}, we can decompose $G(t,x)$ into
\bma
G(t,x)&=\(G(t,x)-W_4(t,x)\)1_{\{|x|\leq 2\mathbf{c}t\}}+\(G(t,x)-W_4(t,x)\)1_{\{|x|> 2\mathbf{c}t\}}+W_4(t,x)\nnm\\
&=G_{L,0}(t,x)1_{\{|x|\leq 2\mathbf{c}t\}}+\(G_{L,1}(t,x)+G_{H}(t,x)-W_4(t,x)\)1_{\{|x|\leq 2\mathbf{c}t\}}\nnm\\
&\quad+R_4(t,x) 1_{\{|x|> 2\mathbf{c}t\}}+W_4(t,x).\label{gfth1-1}
\ema
Thus,
\bq
G(t,x)=G_1(t,x)+G_2(t,x)+W_4(t,x),\label{gfth1-2}
\eq
where
$$\left\{\bln
G_1(t,x)&=G_{L,0}(t,x)1_{\{|x|\leq 2\mathbf{c}t\}}, \\
G_2(t,x)&=\(G_{L,1}(t,x)+G_{H}(t,x)-W_4(t,x)\)1_{\{|x|\leq 2\mathbf{c}t\}}+R_4(t,x) 1_{\{|x|> 2\mathbf{c}t\}}.
\eln\right.
$$
 \eqref{Thm1-1-1}--\eqref{Thm1-1-4}  directly follows from Lemma \ref{rbegf6}.
By Lemmas \ref{rbegf1} and \ref{rbegf11}, we have for $|x|\leq 2\mathbf{c}t$ that
\bq\label{gfth1-5}
\| G_{L,1}(t,x)\|+\|G_{H}(t,x)-W_{4}(t,x)\|\leq Ce^{-\frac{\nu_2t}{6}}\le Ce^{-\frac{|x|+t}{D}}.
\eq
By combining \eqref{gfout1} and \eqref{gfth1-5}, we can obtain  \eqref{Thm1-2}. This completes the proof.
\end{proof}

\section{The Nonlinear system}\setcounter{equation}{0}
\label{sect4}

In this section, we prove Theorem \ref{rbeth2} on the pointwise behaviors of the global solution to the nonlinear relativistic Boltzmann equation with the help of the estimates of the Green's function given in Section \ref{sect3}.

First, we  give some basic estimates  of convolution of the initial data and different waves in order to analyze the pointwise behaviors of the solution.

\begin{lem}\label{rbepw2}
For any $\alpha\geq0$, $D,D_1>0$ and $\lambda\geq0$, there exists $C,D_2>0$ such that for $D_2\ge 12\max\{D,D_1\}\times\max\{1,\lambda\}$
\bma
&\quad\int_{\R^3}(1+t)^{-\alpha}e^{-\frac{|x-y|^2}{D(1+t)}}e^{-\frac{|y|}{D_1}}dy\leq C(1+t)^{-\alpha}e^{-\frac{3|x|^2}{2D_2(1+t)}}+Ce^{-\frac{3 (|x|+t)}{2D_2}},\label{rbepw201j1}\\
&\quad\int_{\R^3}(1+t)^{-\alpha}e^{-\frac{(|x-y|-\lambda t)^2}{D(1+t)}}e^{-\frac{|y|}{D_1}}dy\leq C(1+t)^{-\alpha}e^{-\frac{3(|x|-\lambda t)^2}{2D_2(1+t)}}+Ce^{-\frac{3(|x|+t)}{2D_2}},\label{rbepw201}\\
&\quad\int_{\{|x-y|\leq\lambda t\}}(1+t)^{-\alpha}B_{\frac{3}{2}}(t,|x-y|)e^{-\frac{|y|}{D_1}}dy\nnm\\
&\leq C(1+t)^{-\alpha-\frac{3}{2}}e^{-\frac{3(|x|-\lambda t)^2}{2D_2(1+t)}}+Ce^{-\frac{3(|x|+t)}{2D_2}}+C(1+t)^{-\alpha}B_{\frac{3}{2}}(t,|x|)1_{\{|x|\leq\lambda t\}}.\label{rbepw203}
\ema
\end{lem}

\begin{proof}
First, we prove \eqref{rbepw201j1}. We split $y$ into $|y|\leq\frac{|x|}{2}$ and $|y|\geq\frac{|x|}{2}$. If $|y|\leq\frac{|x|}{2}$, we have $|x-y|\geq\frac{|x|}{2}$. Thus,
\bma
&\quad\int_{\R^3}(1+t)^{-\alpha}e^{-\frac{|x-y|^2}{D(1+t)}}e^{-\frac{|y|}{D_1}}dy\nnm\\
&=(1+t)^{-\alpha}\(\int_{\{|y|\leq\frac{|x|}{2}\}}+\int_{\{|y|\geq\frac{|x|}{2}\}}\)e^{-\frac{|x-y|^2}{D(1+t)}}e^{-\frac{|y|}{D_1}}dy\nnm\\
&\leq C(1+t)^{-\alpha}e^{-\frac{|x|^2}{4D(1+t)}}\int_{\R^3}e^{-\frac{|y|}{D_1}}dy+C(1+t)^{-\alpha}e^{-\frac{|x|}{4D_1}}\int_{\R^3}e^{-\frac{|y|}{2D_1}}dy\nnm\\
&\leq C(1+t)^{-\alpha}e^{-\frac{|x|^2}{4D(1+t)}}+C(1+t)^{-\alpha}e^{-\frac{|x|}{4D_1}}.
\ema
 Since it holds for $\lambda_1=\max\{1,\lambda\}$ that
\bq\label{rbepw204j1}
\left\{\bal
e^{-\frac{|x|}{4D_1}}\leq e^{-\frac{|x|^2}{8D_1\lambda_1(1+t)}}, & |x|\leq\lambda_1 t,\\
e^{-\frac{|x|}{4D_1}}\leq e^{-\frac{|x|+\lambda_1 t}{8D_1}}, & |x|>\lambda_1 t,
\ea\right.
\eq
we can obtain \eqref{rbepw201j1}.

Next, we prove \eqref{rbepw201} as follows. We split $y$ into $|y|\leq\frac{||x|-\lambda t|}{2}$ and $|y|\geq\frac{||x|-\lambda t|}{2}$.
For $|x|\leq\lambda t$ and $|y|\leq\frac{||x|-\lambda t|}{2}$, we have $\lambda t-|x-y|\geq\lambda t-|x|-|y|\geq\frac{\lambda t-|x|}{2}$.
For $|x|>\lambda t$ and $|y|\leq\frac{||x|-\lambda  t|}{2}$, we have $|x-y|-\lambda t\geq|x|-|y|-\lambda t\geq\frac{|x|-\lambda t}{2}$.
Thus,
\bma
&\quad\int_{\R^3}(1+t)^{-\alpha}e^{-\frac{(|x-y|-\lambda t)^2}{D(1+t)}}e^{-\frac{|y|}{D_1}}dy\nnm\\
&=\(\int_{\{|y|\leq\frac{||x|-\lambda t|}{2}\}}+\int_{\{|y|\geq\frac{||x|-\lambda t|}{2}\}}\)(1+t)^{-\alpha}e^{-\frac{(|x-y|-\lambda t)^2}{D(1+t)}}e^{-\frac{|y|}{D_1}}dy \nnm\\
&\leq C(1+t)^{-\alpha}e^{-\frac{(|x|-\lambda t)^2}{4D(1+t)}}\int_{\R^3}e^{-\frac{|y|}{D_1}}dy+C(1+t)^{-\alpha}e^{-\frac{||x|-\lambda t|}{4D_1}}\int_{\R^3}e^{-\frac{|y|}{2D_1}}dy\nnm\\
&\leq C(1+t)^{-\alpha}e^{-\frac{(|x|-\lambda t)^2}{4D(1+t)}}+C(1+t)^{-\alpha}e^{-\frac{||x|-\lambda t|}{4D_1}}.
\ema
Since it holds for $\lambda_1=\max\{1,\lambda\}$ that
\bq\label{rbepw204}
\left\{\bal
e^{-\frac{||x|-\lambda t|}{4D_1}}\leq e^{\frac{(|x|-\lambda t)^2}{8D_1\lambda_1 (1+t)}},& ||x|-\lambda t|\leq2\lambda_1 t,\\
e^{-\frac{||x|-\lambda t|}{4D_1}}\leq e^{-\frac{|x|+\lambda_1 t}{8D_1}},& ||x|-\lambda t|\geq2\lambda_1 t,
\ea\right.
\eq
 we can obtain \eqref{rbepw201}.

Finally, we prove \eqref{rbepw203}. For $ |x|\leq\lambda t$, we have
\bma
&\quad \int_{\{|x-y|\leq\lambda t\}}(1+t)^{-\alpha}B_{\frac{3}{2}}(t,|x-y|)e^{-\frac{|y|}{D_1}}dy\nnm\\
&=(1+t)^{-\alpha}\(\int_{\{|y|\leq\frac{|x|}{2}\}}+\int_{\{|y|\geq\frac{|x|}{2}\}}\)B_{\frac{3}{2}}(t,|x-y|)e^{-\frac{|y|}{D_1}}dy\nnm\\
&\leq C(1+t)^{-\alpha}B_{\frac{3}{2}}(t,|x|)+C(1+t)^{-\alpha}e^{-\frac{|x|}{4D_1}}\nnm\\
&\leq C(1+t)^{-\alpha}B_{\frac{3}{2}}(t,|x|).
\ema
For $|x|>\lambda t$ and $|x-y|\leq\lambda t$, we have $|y|\geq|x|-\lambda t$. We split $y$ into $|y|\leq\frac{\lambda t}{2}$ and $|y|\geq\frac{\lambda t}{2}$. If $|y|\leq\frac{\lambda t}{2}$ and $|x|>\lambda t$, then we have $|x-y|\ge \frac{\lambda t}{2}$. Thus,
\bma
&\quad\int_{\{|x-y|\leq\lambda t\}}(1+t)^{-\alpha}B_{\frac{3}{2}}(t,|x-y|)e^{-\frac{|y|}{D_1}}dy\nnm\\
&\leq C(1+t)^{-\alpha}\(1+\frac{\lambda^2t^2}{4(1+t)}\)^{-\frac{3}{2}}e^{-\frac{||x|-\lambda t|}{2D_1}}\int_{\{|y|\leq\frac{\lambda t}{2}\}}e^{-\frac{|y|}{2D_1}}dy\nnm\\
&\quad+C(1+t)^{-\alpha}e^{-\frac{\lambda t}{8D_1}}e^{-\frac{||x|-\lambda t|}{2D_1}}\int_{\{|y|\geq\frac{\lambda t}{2}\}}e^{-\frac{|y|}{4D_1}}dy\nnm\\
&\leq C(1+t)^{-\alpha-\frac{3}{2}}e^{-\frac{(|x|-\lambda t)^2}{8D_1\lambda(1+t)}}+Ce^{-\frac{|x|+t}{8D_1}},
\ema
where we have used \eqref{rbepw204}. The proof of the lemma is completed.
\end{proof}

We now consider the following integrals for estimating the nonlinear interactions:
\bmas
&\quad I^{\alpha,\beta}(t,x;t_1,t_2;\mu_1,\mu_2,D,D_1)\\
&=\int^{t_2}_{t_1}\int_{\R^3}(1+t-s)^{-\alpha}e^{-\frac{(|x-y|-\mu_1(t-s))^2}{D(1+t-s)}}(1+s)^{-\beta}e^{-\frac{2(|y|-\mu_2 s)^2}{D_1(1+s)}}dyds,\\
&\quad J^{\alpha,\beta}(t,x;t_1,t_2;\lambda,\mu,D_1)\\
&=\int^{t_2}_{t_1}\int_{\{|x-y|\leq\lambda(t-s)\}}(1+t-s)^{-\alpha}B_{\frac{3}{2}}(t-s,|x-y|)(1+s)^{-\beta}e^{-\frac{2(|y|-\mu s)^2}{D_1(1+s)}}dyds,\\
&\quad K^{\alpha,\beta}(t,x;t_1,t_2;\lambda,\mu)\\
&=\int^{t_2}_{t_1}\int_{\{|x-y|\leq\lambda(t-s)\}}(1+t-s)^{-\alpha}B_{\frac{3}{2}}(t-s,|x-y|) (1+s)^{-\beta}B_{k}(s,|y|-\mu s)1_{\{|y|\leq\lambda s\}}dyds,\\
&\quad L^{\alpha,\beta}(t,x;t_1,t_2;\lambda,\mu_1,\mu,D)\\
&=\int^{t_2}_{t_1}\int_{\R^3}(1+t-s)^{-\alpha}e^{-\frac{(|x-y|-\mu_1(t-s))^2}{D(1+t-s)}} (1+s)^{-\beta}B_{k}(s,|y|-\mu s)1_{\{|y|\leq\lambda s\}}dyds,
\emas
and
\bmas
&M^{\alpha}(t,x;t_1,t_2;\mu,D,D_1)=\int^{t_2}_{t_1}\int_{\R^3}e^{-\frac{|x-y|+t-s}{D}}(1+s)^{-\alpha}e^{-\frac{2(|y|-\mu s)^2}{D_1(1+s)}}dyds,\\
&N^{\alpha}(t,x;t_1,t_2;\lambda,\mu,D)=\int^{t_2}_{t_1}\int_{\{|y|\leq\lambda s\}}e^{-\frac{|x-y|+t-s}{D}}(1+s)^{-\alpha}B_{k}(s,|y|-\mu s)dyds,
\emas
where $k=0$ for $\mu=0$, and $k=2$ for $\mu>0.$

Set
$$
\Gamma_{\alpha}(t)=\int_0^t(1+s)^{-\alpha}=C
\left\{\bln
&(1+t)^{1-\alpha},\quad\alpha<1,\\
&\ln(1+t), \quad~~\alpha=1,\\
&1,\qquad\qquad\quad \alpha>1.
\eln\right.
$$

Lemmas \ref{rbepw3}--\ref{rbepw7} are the estimates of coupling waves, and the proofs  are given in Appendix.
\begin{lem}\label{rbepw3}
For given $\alpha,\beta>0$ and $D>0$, there exist two constants $C, D_1>0$ such that for $D_1\ge 16D$,
\bma
&\quad I^{\alpha,\beta}(t,x;0,t;0,0,D,D_1)\nnm\\
&\leq C\((1+t)^{-\alpha}\Gamma_{\beta-\frac{3}{2}}(t)+(1+t)^{-\beta}\Gamma_{\alpha-\frac{3}{2}}(t)\)e^{-\frac{3|x|^2}{2D_1(1+t)}},\label{rbepw301}\\
&\quad I^{\alpha,\beta}(t,x;0,t;0,\lambda,D,D_1),\,I^{\beta,\alpha}(t,x;0,t;0,\lambda,D,D_1)\nnm\\
&\leq C\((1+t)^{-\alpha}\Gamma_{\beta-\frac{5}{2}}(t)+(1+t)^{-\beta}\Gamma_{\alpha-\frac{5}{2}}(t)\)\(e^{-\frac{3|x|^2}{2D_1(1+t)}}+e^{-\frac{3(|x|-\lambda t)^2}{2D_1(1+t)}}\)\nnm\\
&\quad+C(1+t)^{-\alpha+2}\(1+|x|\)^{-\beta}1_{\{\sqrt{1+t}\leq|x|\leq\lambda t-\sqrt{1+t}\}}\nnm\\
&\quad+C(1+t)^{-\beta+2}\(1+\lambda t-|x|\)^{-\alpha}1_{\{\sqrt{1+t}\leq|x|\leq\lambda t-\sqrt{1+t}\}},\label{rbepw302}\\
&\quad I^{\alpha,\beta}(t,x;0,t;\lambda,\lambda,D,D_1)\nnm\\
&\leq C\((1+t)^{-\alpha}\Gamma_{\beta-\frac{5}{2}}(t)+(1+t)^{-\beta}\Gamma_{\alpha-\frac{5}{2}}(t)\)\(e^{-\frac{3|x|^2}{2D_1(1+t)}}+e^{-\frac{3(|x|-\lambda t)^2}{2D_1(1+t)}}\)\nnm\\
&\quad  +C (1+t)^{-\alpha+1}(1+\lambda t-|x|)^{-\beta+\frac52}1_{\{\sqrt{1+t}\leq|x|\leq\lambda t-\sqrt{1+t}\}} \nnm\\
&\quad  +C(1+t)^{-\beta+1}(1+\lambda t-|x|)^{-\alpha+\frac52}1_{\{\sqrt{1+t}\leq|x|\leq\lambda t-\sqrt{1+t}\}} .\label{rbepw303}
\ema
\end{lem}


\begin{lem}\label{rbepw4}
For given $\alpha,\beta,\lambda>0$ and $D>0$, there exist two constants $C, D_1>0$ such that for $D_1\ge 2D$,
\bma
&\quad L^{\alpha,\beta}(t,x;0,t;\lambda,0,0,D)\nnm\\
&\leq C\((1+t)^{-\alpha}\Gamma_{\beta-\frac{3}{2}}(t)+(1+t)^{-\beta}\Gamma_{\alpha-\frac{3}{2}}(t)\) \(e^{-\frac{3(|x|-\lambda t)^2}{2D_1(1+t)}}+B_{3}(t,|x|)1_{\{|x|\leq\lambda t\}}\),\label{rbepw401j1}\\
&\quad J^{\alpha,\beta}(t,x;0,t;\lambda,0,D_1)\nnm\\
&\leq C\((1+t)^{-\alpha}\Gamma_{\beta-\frac{3}{2}}(t)+(1+t)^{-\beta}\ln(1+t)\Gamma_{\alpha-\frac{3}{2}}(t)\) \(e^{-\frac{3(|x|-\lambda t)^2}{2D_1(1+t)}}+B_{\frac{3}{2}}(t,|x|)1_{\{|x|\leq\lambda t\}}\) ,\label{rbepw401}\\
&\quad J^{\alpha,\beta}(t,x;0,t;\lambda,\lambda,D_1)\nnm\\
&\leq C\((1+t)^{-\alpha}\Gamma_{\beta-\frac{5}{2}}(t)+(1+t)^{-\beta}\Gamma_{\alpha-\frac{5}{2}}(t)\)\(e^{-\frac{3|x|^2}{2D_1(1+t)}}+e^{-\frac{3(|x|-\lambda t)^2}{2D_1(1+t)}}\)\nnm\\
&\quad+C\((1+t)^{-\alpha}\Gamma_{\beta-\frac{5}{2}}(t)B_{\frac{3}{2}}(t,|x|)+(1+t)^{-\beta+1}\Gamma_{\alpha-\frac{3}{2}}(t)B_{\frac{3}{2}}(t,|x|-\lambda t)\)1_{\{|x|\leq\lambda t\}}\nnm\\
&\quad+C(1+t)^{-\alpha+2}\ln(1+t)\(1+|x|\)^{-\beta}1_{\{\sqrt{1+t}\leq|x|\leq\lambda t-\sqrt{1+t}\}}\nnm\\
&\quad+C(1+t)^{-\beta+2}\ln(1+t)\(1+\lambda t-|x|\)^{-\alpha}1_{\{\sqrt{1+t}\leq|x|\leq\lambda t-\sqrt{1+t}\}}.\label{rbepw402}
\ema
\end{lem}


\begin{lem}\label{rbepw5}
For given $\alpha,\beta>0$, there exists a constant $C>0$ such that
\bma
&\quad K^{\alpha,\beta}(t,x;0,t;\lambda,0)\nnm\\
&\leq C\((1+t)^{-\alpha}\Gamma_{\beta-\frac{3}{2}}(t)+(1+t)^{-\beta}\Gamma_{\alpha-\frac{3}{2}}(t)\)\ln(1+t)B_{\frac{3}{2}}(t,|x|)1_{\{|x|\leq\lambda t\}},\label{rbepw501}\\
&\quad K^{\alpha,\beta}(t,x;0,t;\lambda,\lambda)\nnm\\
&\leq C\((1+t)^{-\alpha}\Gamma_{\beta-\frac{5}{2}}(t)+(1+t)^{-\beta+1}\Gamma_{\alpha-\frac{3}{2}}(t)\) \(B_{\frac{3}{2}}(t,|x|)+B_{\frac{3}{2}}(t,|x|-\lambda t)\)1_{\{|x|\leq\lambda t\}}\nnm\\
&\quad+C(1+t)^{-\alpha+2}\ln(1+t)\(1+|x|\)^{-\beta}1_{\{\sqrt{1+t}\leq|x|\leq\lambda t-\sqrt{1+t}\}}\nnm\\
&\quad+C(1+t)^{-\beta+2}\ln(1+t)\(1+\lambda t-|x|\)^{-\alpha}1_{\{\sqrt{1+t}\leq|x|\leq\lambda t-\sqrt{1+t}\}}.\label{rbepw502}
\ema
\end{lem}


\begin{lem}\label{rbepw6}
For given $\alpha,\beta,D>0$, there exists a constant $C, D_1>0$ such that for $D_1\ge \frac{3D}{2}$,
\bma
&\quad L^{\alpha,\beta}(t,x;0,t;\lambda,0,\lambda,D),\,  L^{\beta,\alpha}(t,x;0,t;\lambda,0,\lambda,D)\nnm\\
&\leq C\((1+t)^{-\alpha}\Gamma_{\beta-\frac{5}{2}}(t)+(1+t)^{-\beta}\Gamma_{\alpha-\frac{5}{2}}(t)\)\(e^{-\frac{3|x|^2}{2D_1(1+t)}}+e^{-\frac{3(|x|-\lambda t)^2}{2D_1(1+t)}}\)\nnm\\
&\quad +C\((1+t)^{-\alpha}\Gamma_{\beta-\frac{5}{2}}(t)B_{2}(t,|x|)+(1+t)^{-\beta}\Gamma_{\alpha-\frac{5}{2}}(t)B_{2}(t,|x|-\lambda t)\)1_{\{|x|\leq\lambda t\}}\nnm\\
&\quad+C(1+t)^{-\alpha+2}\(1+|x|\)^{-\beta}1_{\{\sqrt{1+t}\leq|x|\leq\lambda t-\sqrt{1+t}\}}\nnm\\
&\quad+C(1+t)^{-\beta+2}\(1+\lambda t-|x|\)^{-\alpha}1_{\{\sqrt{1+t}\leq|x|\leq\lambda t-\sqrt{1+t}\}},\label{rbepw603}\\
&\quad L^{\alpha,\beta}(t,x;0,t;\lambda,\lambda,\lambda,D)\nnm\\
&\leq C\((1+t)^{-\alpha}\Gamma_{\beta-\frac{5}{2}}(t)+(1+t)^{-\beta}\Gamma_{\alpha-\frac{5}{2}}(t)\)\(e^{-\frac{3|x|^2}{2D_1(1+t)}}+e^{-\frac{3(|x|-\lambda t)^2}{2D_1(1+t)}}\)\nnm\\
&\quad+C\((1+t)^{-\alpha+\frac{1}{2}}\Gamma_{\beta-2}(t)+(1+t)^{-\beta}\Gamma_{\alpha-\frac{5}{2}}(t)\) B_{2}(t,|x|-\lambda t)1_{\{|x|\leq\lambda t\}}\nnm\\
&\quad + C (1+t)^{-\alpha+1}(1+\lambda t-|x|)^{-\beta+\frac52}1_{\{\sqrt{1+t}\leq|x|\leq\lambda t-\sqrt{1+t}\}} \nnm\\
&\quad + C(1+t)^{-\beta+1}(1+\lambda t-|x|)^{-\alpha+\frac52}1_{\{\sqrt{1+t}\leq|x|\leq\lambda t-\sqrt{1+t}\}} .\label{rbepw602}
\ema
\end{lem}


\begin{lem}\label{rbepw7}
For given $\lambda,\alpha,D>0$, there exist two constants $C, D_1>0$ such that for $D_1\ge 48D\max\{1,\lambda\}^2$,
\bma
&\quad M^{\alpha}(t,x;0,t;0,D,D_1)\nnm\\
&\leq C(1+t)^{-\alpha}\(e^{-\frac{3|x|^2}{2D_1(1+t)}}+e^{-\frac{3(|x|-\lambda t)^2}{2D_1(1+t)}}\)+Ce^{-\frac{3(|x|+t)}{2D_1}},\label{rbepw701}\\
&\quad N^{\alpha}(t,x;0,t;\lambda,0,D)\nnm\\
&\leq C(1+t)^{-\alpha}e^{-\frac{3(|x|-\lambda t)^2}{2D_1(1+t)}}+C(1+t)^{-\alpha}B_{3}(t,|x|)1_{\{|x|\leq\lambda t\}}+Ce^{-\frac{3(|x|+t)}{2D_1}},\label{rbepw801}\\
&\quad M^{\alpha}(t,x;0,t;\lambda,D,D_1),\ N^{\alpha}(t,x;0,t;\lambda,\lambda,D)\nnm\\
&\leq C(1+t)^{-\alpha}\(e^{-\frac{3(|x|-\lambda t)^2}{2D_1(1+t)}}+e^{-\frac{3|x|^2}{2D_1(1+t)}}\)+Ce^{-\frac{3(|x|+t)}{2D_1}}\nnm\\
 &\quad+C(1+t)^{-\alpha}B_{2}(t,|x|-\lambda t)1_{\{|x|\leq\lambda t\}}.\label{rbepw802}
\ema
\end{lem}

\begin{lem}\label{rbepw10}
For any given $\beta\geq0$, there exists a constant $C>0$ such that
\bq\label{plw1-1}
\|S^tg_0(x)\|_{L^{\infty}_{v,\beta}}\leq Ce^{-\frac{2\nu_0t}{3}}\max_{y\in\R^3}e^{-\frac{\nu_0|x-y|}{3}}\|g_0(y)\|_{L^{\infty}_{v,\beta}},
\eq
where $S^t$ and $\nu_0$ are defined by \eqref{Thm1-4} and \eqref{nu0} respectively. In particular, if the function $g_0(x,v)$ satisfies
$$
\|g_0(x)\|_{L^{\infty}_{v,\beta}}\leq Ce^{-\frac{|x|}{D_1}},
$$
then 
\bq\label{plw1-2}
\|W_{k}(t)\ast g_0(x)\|_{L^{\infty}_{v,\beta}}\leq C(1+t)^ke^{-\frac{2\nu_0t}{3}}e^{-\frac{|x|}{D_1}},
\eq
where $D_1\ge \frac3{\nu_0}$, and $W_k,\ k\geq0$ is defined by \eqref{Thm1-3}.
\end{lem}

\begin{proof}
First, we prove \eqref{plw1-1} as follows. Recall that $S^t$ can be represented as
$$
S^tg_0(x,v)=e^{-\nu(v)t}g_0(x-\tilde{v}t,v).
$$
Let $y=x-\tilde{v}t$. Then, we have
$$
\nu_0t \geq\frac{|v|\nu_0t}{\sqrt{1+|v|^2}}=|\tilde{v}|\nu_0t=\nu_0|x-y|.
$$
Thus,
\bmas
| S^tg_0(x,v)|&\le e^{-\frac{2\nu(v)t}3}e^{-\frac{\nu_0|x-y|}3} |g_0(y,v)| \\
&\le e^{-\frac{2\nu_0t}3}(1+v_0)^{-\beta}\max_{y\in \R^3}e^{-\frac{\nu_0|x-y|}3}\|g_0(y)\|_{L^\infty_{v,\beta}},
\emas
which proves \eqref{plw1-1}.

From \eqref{Thm1-3}, we have
$$
W_k(t)=\sum^{3k}_{i=0}J_{i}(t,x),
$$
where
\bq\label{plw1-3}
\left\{\bln
&J_0(t,x)=S^tg_0(x,v)=e^{-\nu(v)t}g_0(x-\tilde{v}t,v),\\
&J_k(t,x)=\int^t_0S^{t-s}KJ_{k-1}ds,\quad k\geq1.
\eln\right.
\eq
By \eqref{plw1-3}, we have
\bq
\|J_0(t,x)\|_{L^{\infty}_{v,\beta}}=\|S^tg_0(t,x)\|_{L^{\infty}_{v,\beta}}\leq Ce^{-\frac{2\nu_0t}{3}}e^{-\frac{|x|}{D_1}}.\label{plw1-4}
\eq

Thus, it follows from Lemma \ref{rbegf8j1}, \eqref{plw1-1} and \eqref{plw1-4} that
\bmas
\|J_1(t,x)\|_{L^{\infty}_{v,\beta}}
&\leq C\int_{0}^{t}e^{-\frac{2\nu_0(t-s)}{3}}\max_{y\in\R^3}e^{-\frac{\nu_0|x-y|}{3}}\|KJ_0(s,y)\|_{L^{\infty}_{v,\beta}}ds\\
&\leq Ce^{-\frac{|x|}{D_1}}\int_0^te^{-\frac{2\nu_0(t-s)}{3}}e^{-\frac{2\nu_0s}{3}} ds\\
&\leq Cte^{-\frac{2\nu_0t}{3}}e^{-\frac{|x|}{D_1}}.
\emas
By a similar argument as to that above, we obtain
$$
\|J_k(t,x)\|_{L^{\infty}_{v,\beta}}\leq Ct^ke^{-\frac{2\nu_0t}{3}}e^{-\frac{|x|}{D_1}},\quad \forall k\ge 1,
$$
which proves \eqref{plw1-2}. The proof of the Lemma is completed.
\end{proof}


\begin{lem}\label{rbepw11}
For any given $\alpha>0,\gamma_1\ge \frac12$, $\gamma_2,\beta\ge 0$ and $\lambda>0$, if the function $F(t,x,v)$ satisfies
$$
\|F(t,x)\|_{L^{\infty}_{v,\beta}}\leq C\Phi(t,x;D_1,\alpha,\gamma_1,\gamma_2),
$$
then we have for any $\eta >1$ and $D_1\ge \frac{6\sqrt{\eta}\max\{1,\lambda\}^2}{\nu_0(\sqrt{\eta}-1)}$ that
\bma
&\quad\bigg\|\int^t_0S^{t-s}F(s,x)ds\bigg\|_{L^{\infty}_{v,\beta}}\leq C\Phi(t,x;\eta D_1,\alpha,\gamma_1,\gamma_2),\label{plw2-1}\\
&\quad\bigg\|\int^t_0W_{k}(t-s)\ast F(s,x)ds\bigg\|_{L^{\infty}_{v,\beta}}
\leq C\Phi(t,x;\eta D_1,\alpha,\gamma_1,\gamma_2),\label{plw2-2}
\ema
where 
\bmas
\Phi(t,x;D_1,\alpha,\gamma_1,\gamma_2)&=(1+t)^{-\alpha}\(e^{-\frac{|x|^2}{D_1(1+t)}}+B_{\gamma_1}(t,|x|)1_{\{|x|\leq\lambda t\}}\)+e^{-\frac{|x|+t}{D_1}}\\
&\quad +(1+t)^{-\alpha-\frac12}\(e^{-\frac{(|x|-\lambda t)^2}{D_1(1+t)}}+B_{\gamma_2}(t,|x|-\lambda t)1_{\{|x|\leq\lambda t\}}\).
\emas
\end{lem}
\begin{proof}
By Lemma \ref{rbepw10}, we have
$$
 \bigg\|\int^t_0S^{t-s}F(s,x)ds\bigg\|_{L^{\infty}_{v,\beta}}
\leq C\int^t_0e^{-\frac{2\nu_0(t-s)}{3}} \max_{y\in\mathbb{R}^3}e^{-\frac{\nu_0|x-y|}{3}}\Phi (s,y;D_1,\alpha,\gamma_1,\gamma_2)ds.
$$

Let
\bma
I_1&=\int^t_0e^{-\frac{2\nu_0(t-s)}{3}}\max_{y\in\mathbb{R}^3}e^{-\frac{\nu_0|x-y|}{3}} \((1+s)^{-\alpha}e^{-\frac{|y|^2}{D_1(1+s)}}+e^{-\frac{|y|+s}{D_1}}\)ds \nnm\\
&\quad+\int^t_0e^{-\frac{2\nu_0(t-s)}{3}}\max_{y\in\mathbb{R}^3}e^{-\frac{\nu_0|x-y|}{3}}(1+s)^{-\alpha}B_{\gamma_1}(s,|y|)1_{\{|y|\leq\lambda s\}}ds \nnm\\
&=:I_1^1+I_1^2. \label{scon1}
\ema
For $I_1^1$, we split $y$ into $|y|\leq\frac{|x|}{\epsilon}$ and $|y|\geq\frac{|x|}{\epsilon}$ with $ \epsilon>1$. Then we have
$$
e^{-\frac{\nu_0|x-y|}{3}}e^{-\frac{|y|^2}{D_1(1+s)}}\le e^{-(1-\frac{1}{\epsilon})\frac{\nu_0|x|}{3}}+e^{-\frac{|x|^2}{\epsilon^2D_1(1+t)}}.
$$
This implies that
\bma
I_1^1&\leq \( e^{-\frac{|x|^2 }{\epsilon^2D_1(1+t)}}+e^{-(1-\frac{1}{\epsilon})\frac{\nu_0|x|}{3}}\)\int^t_0e^{-\frac{2\nu_0(t-s)}{3}}(1+s)^{-\alpha} ds\nnm\\
&\quad + e^{-\frac{|x|}{ D_1}} \int^t_0e^{-\frac{2\nu_0(t-s)}{3}} e^{-\frac{s}{D_1}} ds\nnm\\
&\leq C(1+t)^{-\alpha}e^{-\frac{ |x|^2 }{\epsilon^2D_1(1+t)}}+Ce^{-\frac{|x|+t}{D_1}},\label{rbepw1102j1}
\ema
where we have used \eqref{rbepw204j1}, and $D_1\ge \frac{6\eps}{\nu_0(\eps-1)}\max\{1,\lambda\}$.

For $I_1^2$, we split $y$ into $|y|\leq\frac{|x|}{2}$ and $|y|\geq\frac{|x|}{2}$ when $|x|\leq\lambda t$. Thus,
\bma
I_1^2
&\leq \(e^{-\frac{\nu_0|x|}{6}}+B_{\gamma_1}(t,|x|)\)\int^t_0e^{-\frac{2\nu_0(t-s)}{3}}(1+s)^{-\alpha}ds\nnm\\
&\leq C(1+t)^{-\alpha}B_{\gamma_1}(t,|x|).\label{scon1-1}
\ema
For $|x|>\lambda t$ and $|y|\leq\lambda s$, we have $|x-y|\geq|x|-\lambda t $. We split $y$ into $|y|\leq\frac{\lambda s}{2}$ and $|y|\geq\frac{\lambda s}{2}$. If $|y|\leq\frac{\lambda s}{2}$, then $|x-y|\geq|x|-\frac{\lambda s}{2}\geq \frac{\lambda s}{2}$.
Thus,
\bma
I_1^2
&\leq Ce^{-\frac{\nu_0||x|-\lambda t|}{6}}\int^t_0e^{-\frac{2\nu_0(t-s)}{3}}(1+s)^{-\alpha}e^{-\frac{\nu_0\lambda s}{12}}ds\nnm\\
&\quad+Ce^{-\frac{\nu_0||x|-\lambda t|}{3}}\int^t_0e^{-\frac{2\nu_0(t-s)}{3}}\(1+\frac{\lambda^2s^2}{4(1+s)}\)^{-\gamma_1}(1+s)^{-\alpha}ds\nnm\\
&\leq  C(1+t)^{-\alpha-\gamma_1}e^{-\frac{\nu_0||x|-\lambda t|}{6}}.\label{scon1-2}
\ema
By combining \eqref{scon1}--\eqref{scon1-2}, we can obtain
\be
I_1\leq (1+t)^{-\alpha}\(e^{-\frac{ |x|^2 }{\epsilon^2D_1(1+t)}}+(1+t)^{-\gamma_1}e^{-\frac{(|x|-\lambda t)^2}{D_1(1+t)}}+B_{\gamma_1}(t,|x|)1_{\{|x|\leq \lambda t\}}\)+Ce^{-\frac{|x|+t}{D_1}},\label{scon1-3}
\ee
where we have used \eqref{rbepw204}, and $D_1\ge \frac{12}{\nu_0}\max\{1,\lambda\}$.

Let
\bmas
I_2&=\int^t_0e^{-\frac{2\nu_0(t-s)}{3}}(1+s)^{-\alpha-\frac{1}{2}}\max_{y\in\mathbb{R}^3}e^{-\frac{\nu_0|x-y|}{3}} e^{-\frac{(|y|-\lambda s)^2}{D_1(1+s)}} ds \\
&\quad+\int^t_0e^{-\frac{2\nu_0(t-s)}{3}}(1+s)^{-\alpha-\frac{1}{2}}\max_{y\in\mathbb{R}^3}e^{-\frac{\nu_0|x-y|}{3}} B_{\gamma_2}(s,|y|-\lambda s)1_{\{|y|\leq\lambda s\}} ds\\
&=:I_2^1+I_2^2.
\emas

For $|x|\leq\lambda t$, since $ e^{-\frac{\nu_0(t-s)}{3}}e^{-\frac{\nu_0|x-y|}{3}} \leq e^{-\frac{\nu_0||x-y|-\lambda(t-s)|}{3E}} $ for $E=\max\{1,\lambda\}$, we can obtain by a similar argument as proving \eqref{rbepw7021} that
\bma
I_2&\le \(\int_0^{\frac{\lambda t-|x|}{4\lambda}}+\int^{t-\frac{\lambda t-|x|}{4\lambda}}_{\frac{\lambda t-|x|}{4\lambda}}+\int^t_{t-\frac{\lambda t-|x|}{4\lambda}}\)e^{-\frac{\nu_0(t-s)}{3}}(1+s)^{-\alpha-\frac{1}{2}}\nnm\\
&\qquad \times\max_{y\in\mathbb{R}^3}e^{-\frac{\nu_0||x-y|-\lambda(t-s)|}{3E}}B_{\gamma_2}(s,|y|-\lambda s) ds\nnm\\
&\leq C(1+t)^{-\alpha-\frac{1}{2}} B_{\gamma_2}(t,|x|-\lambda t) .
\ema

For $|x|\geq\lambda t$, we split $y$ into $||y|-\lambda s|\leq \frac{(|x|-\lambda t)}{\epsilon}$ and $||y|-\lambda s|\geq \frac{(|x|-\lambda t)}{\epsilon}$. If $\lambda s-|y|\leq \frac{(|x|-\lambda t)}{\epsilon}$, then we have $||x-y|-\lambda(t-s)|\geq |x|-\lambda t-||y|-\lambda s|\geq(1-\frac{1}{\epsilon})(|x|-\lambda t)$.  Thus,
\bma
I^1_2 &\leq C\(e^{-(1-\frac{1}{\epsilon})\frac{\nu_0||x|-\lambda t|}{3E}}+e^{-\frac{(|x|-\lambda t)^2}{\epsilon^2D_1(1+t)}}\)\int^t_0e^{-\frac{\nu_0(t-s)}{3}}(1+s)^{-\alpha-\frac12} ds\nnm\\
&\leq C(1+t)^{-\alpha-\frac{1}{2}}\(e^{-(1-\frac{1}{\epsilon})\frac{\nu_0||x|-\lambda t|}{3E}}+e^{-\frac{(|x|-\lambda t)^2}{\epsilon^2D_1(1+t)}}\).\label{rbepw-sk2}
\ema
For $|x|\geq\lambda t$ and $|y|\leq\lambda s$, we have $|x-y|\geq|x|-\lambda t$. Thus,
\bma
I^2_2&\leq C e^{- \frac{\nu_0||x|-\lambda t|}{3}} \int^t_0e^{-\frac{\nu_0(t-s)}{3}}(1+s)^{-\alpha-\frac12} ds\nnm\\
&\leq  C(1+t)^{-\alpha-\frac{1}{2}}e^{- \frac{\nu_0||x|-\lambda t|}{3}}.\label{rbepw1101j2}
\ema
By combining \eqref{rbepw1102j1}--\eqref{rbepw1101j2} and \eqref{rbepw204}, it follows that
\be
I_2\leq C (1+t)^{-\alpha-\frac12} \( e^{-\frac{(|x|-\lambda t)^2}{\epsilon^2D_1(1+t)}}+B_{\gamma_2}(t,|x|-\lambda t)1_{\{|x|\leq \lambda t\}}\)+Ce^{-\frac{|x|+t}{D_1}},\label{rbepw1101j3}
\ee
where $\eta=\eps^2$ and $D_1\ge \frac{6E^2\eps}{\nu_0(\eps-1)}$.
Thus, by combining \eqref{scon1-3} and \eqref{rbepw1101j3}, we obtain
 \eqref{plw2-1}. By a similar argument as to that above, we can prove \eqref{plw2-2}.
\end{proof}


\begin{lem}\label{rbepw12}
For any given $\beta\geq0$, there exists a constant $C>0$ such that
\be
 \|\Gamma(f,g)\|_{L^{\infty}_{v,\beta}}\leq C\|f\|_{L^{\infty}_{v,\beta}}\|g\|_{L^{\infty}_{v,\beta}}.\label{rbepw1201}
\ee
\end{lem}

\begin{proof}
We split $\Gamma(f,g)$ into two part:
$$
\Gamma(f,g)=\Gamma_1(f,g)-\Gamma_2(f,g),
$$
where
\bmas
&\Gamma_1(f,g)=\frac{1}{\sqrt{M(v)}}\int_{\R^3}\int_{\S^2}v_M\sqrt{M(v')}f(v')\sqrt{M(u')}g(u')dud\Omega,\\
&\Gamma_2(f,g)=\frac{1}{\sqrt{M(v)}}\int_{\R^3}\int_{\S^2}v_M\sqrt{M(v)}f(v)\sqrt{M(u)}g(u)dud\Omega.
\emas
First, we  estimate $\Gamma_1(f,g)$. Since $u_0+v_0=u'_0+v'_0$, we have
$$
\sqrt{M(v')}\sqrt{M(u')}=\sqrt{M(v)}\sqrt{M(u)}.
$$
Thus,
\bma
(1+v_0)^{\beta}\Gamma_1(f,g)&=\frac{1}{\sqrt{M(v)}}\int_{\R^3}\int_{\S^2}v_M\sqrt{M(v)}(1+v_0)^{\beta}f(v')\sqrt{M(u)}g(u')dud\Omega\nnm\\
&\leq\int_{\R^3}\int_{\S^2}v_M(1+v'_0)^{\beta}f(v')\sqrt{M(u)}(1+u'_0)^{\beta}g(u')dud\Omega\nnm\\
&\leq \|f\|_{L^{\infty}_{v,\beta}}\|g\|_{L^{\infty}_{v,\beta}}\int_{\R^3}\int_{\S^2}v_M\sqrt{M(u)}dud\Omega\nnm\\
&\leq C\|f\|_{L^{\infty}_{v,\beta}}\|g\|_{L^{\infty}_{v,\beta}},\label{rbepw1202}
\ema
where we have used
$$
(1+v'_0)(1+u'_0)\geq1+v'_0+u'_0\geq1+v_0.
$$
Similarly, we can show that  $\Gamma_2(f,g)$ also satisfies \eqref{rbepw1202}. The proof is completed.
\end{proof}


With the help of Theorem \ref{rbeth1} and Lemmas \ref{rbepw2}--\ref{rbepw12}, we can perform the following proof for Theorem \ref{rbeth2}.
\begin{proof}[\textbf{Proof of Theorem 1.2}]
We prove the existence of the solution $f$  by the contraction mapping theorem. Define the mapping $T$ as
\bq\label{rbethpr2-1}
(Tf)(t,x)=G(t)\ast f_0+\int^t_0G(t-s)\ast \Gamma(f,f)ds=:I_1+I_2.
\eq
For $\delta>0$, set the solution space $X$ as
$$
X=\Big\{f\in L^{\infty}_{v,2}(\R^3_v)\,|\,\|f(t)\|_{X}=\sup_{x\in\mathbb{R}^3,0\leq s\leq t}\|f(s,x)\|_{L^{\infty}_{v,2}}\Psi (s,x;D_2)^{-1}\le \delta \Big\},
$$
where
\bmas
\Psi (s,x;D_2)&=(1+s)^{-\frac{3}{2}}\(e^{-\frac{|x|^2}{D_2(1+s)}}+(1+s)^{-\frac{1}{2}}e^{-\frac{(|x|-\mathbf{c} s)^2}{D_2(1+s)}}\)+e^{-\frac{|x|+s}{D_2}}\\
&\quad+(1+s)^{-\frac{3}{2}}\(B_{\frac{3}{2}}(s,|x|)+(1+s)^{-\frac{1}{2}} B_{1}(s,|x|-\mathbf{c} s)\)1_{\{|x|\leq \mathbf{c} s\}}.
\emas

By Lemma \ref{rbepw12}, it holds that for $0\le s\le t$,
\bma
\|\Gamma(f,f)(s,x)\|_{L^{\infty}_{v,2}}&\leq C\|f(t)\|_{X}^2\bigg[(1+s)^{-3}\(e^{-\frac{2|x|^2}{D_2(1+s)}}+(1+s)^{-1}e^{-\frac{2(|x|-\mathbf{c} s)^2}{D_2(1+s)}}\)+e^{-\frac{2(|x|+s)}{D_2}}\nnm\\
&\quad+(1+s)^{-3}\(B_{3}(s,|x|)+(1+s)^{-1}B_{2}(s,|x|-\mathbf{c} s)\)1_{\{|x|\leq\mathbf{c} s\}}\bigg]. \label{rbethpr2-3}
\ema

We estimate $I_j$, $j=1,2$ in \eqref{rbethpr2-1} as follows. By Theorem \ref{rbeth1}, we decompose $I_1$ into
\bmas
I_1&=G_1(t)\ast f_0+G_2(t)\ast f_0+W_2(t)\ast f_0 \\
&=:I_1^1+I_1^2+I_1^3.
\emas
For $I_1^1$, it follows from Theorem \ref{rbeth1} and Lemma \ref{rbepw2} that
\bma
\|I_1^1\|&\leq C\delta_0 \int_{\R^3}\Bigg(\frac{e^{-\frac{|x-y|^2}{D(1+t)}}}{(1+t)^{3/2}}+\frac{e^{-\frac{(|x-y|-\mathbf{c}t)^2}{D(1+t)}}}{(1+t)^{2}}+e^{-\frac{|x-y|+t}{D}} \Bigg)e^{-\frac{|y|}{D_1}}dy \nnm\\
&\quad+ C\delta_0\int_{\{|x-y|\leq\mathbf{c} t\}}(1+t)^{-\frac{3}{2}}B_{\frac{3}{2}}(t,|x-y|)e^{-\frac{|y|}{D_1}}dy\nnm\\
&\leq C\delta_0\Psi\Big(t,x;\frac23D_2\Big).\label{rbethpr2-5}
\ema
For $I_1^2$ and $I_1^3$, we can obtain by Theorem \ref{rbeth1} and Lemma \ref{rbepw10} that
\be
\|I_1^2\|+\|I_1^3\|\leq \|I_1^2\|+\|I_1^3\|_{L^{\infty}_{v,2}}\leq C\delta_0e^{-\frac{|x|+t}{D_2}}.\label{rbethpr2-6}
\ee
Thus, it follows from \eqref{rbethpr2-5}--\eqref{rbethpr2-6} that
\be
\|I_1(t,x)\|\leq C\delta_0\Psi\Big(t,x;\frac23D_2\Big).\label{rbepwnl1}
\ee

By Theorem \ref{rbeth1} and noting that $P_0\Gamma(f,f)=0$, we decompose $I_2$ into
\bma
I_2&=\int^t_0G_1(t-s)P_1\ast \Gamma(f,f)ds+\int^t_0G_2(t-s)\ast\Gamma(f,f)ds\nnm\\
&\quad+\int^t_0W_2(t-s)\ast\Gamma(f,f)ds\nnm\\
&=:I_2^1+I_2^2+I_2^3.\nnm
\ema
For $I_2^1$, we obtain by Theorem \ref{rbeth1}, Lemma \ref{rbepw3}--\ref{rbepw7} and \eqref{rbethpr2-3} that
\bma
\|I_2^1\|&\leq C\|f(t)\|_{X}^2\int^t_0\int_{\R^3}\Bigg(\frac{e^{-\frac{|x-y|^2}{D(1+t-s)}}}{(1+t-s)^{2}} +\frac{e^{-\frac{(|x-y|-\mathbf{c}(t-s))^2}{D(1+t-s)}}}{(1+t-s)^{5/2}}+e^{-\frac{|x-y|+t-s}{D}}\Bigg)\Psi (s,y;D_2)^2dyds\nnm\\
&\quad+C\|f(t)\|_{X}^2\int^t_0\int_{\{|x-y|\leq\mathbf{c}(t-s)\}}\frac{B_{\frac{3}{2}}(t-s,|x-y|)}{(1+t-s)^{2}}\Psi (s,y;D_2)^2dyds\nnm\\
&\leq C\|f(t)\|_{X}^2\Psi \Big(t,x;\frac23D_2\Big).\label{rbethg-1}
\ema

For $I_2^2$ and $I_2^3$, we have by \eqref{rbethpr2-3}, Lemma \ref{rbepw7} and Lemma \ref{rbepw11} that
\be
\|I_2^2\|+\|I_2^3\| \leq\|I_2^2\|+\|I_2^3\|_{L^{\infty}_{v,2}}\leq C\|f(t)\|_{X}^2\Psi \Big(t,x;\frac23D_2\Big).\label{rbethg-2}
\ee
Thus, it follows from \eqref{rbethg-1}--\eqref{rbethg-2} that
\be
 \|I_2(t,x)\|\leq C\|f(t)\|_{X}^2\Psi \Big(t,x;\frac23D_2\Big).\label{rbepwnl2}
\ee

By taking summation of \eqref{rbepwnl1} and \eqref{rbepwnl2}, we obtain
\bq\label{nldf1}
\|(Tf)(t,x)\|\leq C(\delta_0+\|f(t)\|_{X}^2)\Psi \Big(t,x;\frac23D_2\Big).
\eq

Note that $I_1,I_2$ satisfy
\bmas
&\partial_tI_1+\tilde{v}\cdot\nabla_x I_1+\nu(v)I_1=KI_1, \quad I_1(0)=f_0,\\
&\partial_tI_2+\tilde{v}\cdot\nabla_x I_2+\nu(v)I_2=KI_2+\Gamma(f,f),\quad I_2(0)=0.
\emas
Thus, we can represent $I_1,I_2$ as
\bma\label{rbenlv-2}
I_1 &=S^tf_0+\int^t_0S^{t-s} KI_1 ds,\\
I_2 &= \int^t_0S^{t-s} (KI_2+\Gamma(f,f)) ds.
\ema
By Lemma \ref{rbepw10}, it holds that
\bq\label{rbenlv-3}
\|S^t f_0(x)\|_{L^{\infty}_{v,2}}\leq C\delta_0e^{-\frac{\nu_0t}{3}}e^{-\frac{|x|}{D_1}}.
\eq
By Lemma \ref{rbegf8j1}, \eqref{rbepwnl1} and \eqref{rbepwnl2}, we obtain
$$
\left\{\bln
&\|KI_1\|_{L^{\infty}_{v,0}} \leq C\|I_1\|\leq C \delta_0 \Psi \Big(t,x;\frac23D_2\Big),  \\
&\|KI_2\|_{L^{\infty}_{v,0}} \leq C\|I_2\|\leq C \|f(t)\|_{X}^2 \Psi \Big(t,x;\frac23D_2\Big),
\eln\right.
$$
which, together with Lemma \ref{rbepw11} and \eqref{rbenlv-3}, implies that for any $\eta >1$,
\bmas
\|I_1\|_{L^{\infty}_{v,0}}& \leq C \delta_0 \Psi \Big(t,x;\frac{2}{3 }\eta D_2\Big),\\
\|I_2\|_{L^{\infty}_{v,0}}& \leq C \|f(t)\|_{X}^2 \Psi \Big(t,x;\frac{2}{3 }\eta D_2\Big).
\emas

By induction, Lemma \ref{rbepw11}  and
$$
\|KI_i\|_{L^{\infty}_{v,k}}\leq\|I_i\|_{L^{\infty}_{v,k-1}},\quad k\geq1, \,\ i=1,2,
$$
we have
\bma\label{rbenlv-4-1}
\|I_1\|_{L^{\infty}_{v,2}}&= \|G(t)\ast f_0\|_{L^{\infty}_{v,2}}\leq C\delta_0 \Psi(t,x;D_2),\\
\|I_2\|_{L^{\infty}_{v,2}}&=\bigg\|\int^t_0G(t-s)\ast \Gamma(f,f)ds\bigg\|_{L^{\infty}_{v,2}}\leq C \|f(t)\|_{X}^2 \Psi(t,x;D_2), \label{gamma}
\ema
namely
\be\label{rbenlv-6}
\|(Tf)(t,x)\|_{L^{\infty}_{v,2}}\leq C(\delta_0+\|f(t)\|_{X}^2)\Psi (t,x;D_2).
\ee
Thus, for $f\in X$,
$$
\|(Tf)(t)\|_{X}\leq C\delta_0+C\delta^2,
$$
which implies that $T$ is the mapping of $X\rightarrow X$ so long as $\delta_0\leq\frac{\delta}{2C}$ and $0<\delta<\frac{1}{2C}$ is small enough.

 Noting that for any $f_1,f_2\in X$, we have
$$
\Gamma(f_2,f_2)-\Gamma(f_1,f_1)=\Gamma(f_2-f_1,f_2)+\Gamma(f_1,f_2-f_1).
$$
Thus, by \eqref{gamma}  we obtain
\bmas
 \|T(f_2)-T(f_1)\|_{L^{\infty}_{v,2}}&= \bigg\|\int^t_0G(t-s)\ast (\Gamma(f_2,f_2)-\Gamma(f_1,f_1))ds\bigg\|_{L^{\infty}_{v,2}}\nnm\\
&\leq C\Psi(t,x;D_2)(\|f_1\|_X+\|f_2\|_X)\|f_2-f_1\|_X\nnm\\
&\leq C\delta\Psi(t,x;D_2)\|f_2-f_1\|_X.
\emas
Thus, $T$ is a contraction mapping of $X\rightarrow X$. By the contraction mapping theorem, there exists a unique $f\in X$ such that $Tf=f$, i.e., $f$ is a unique global solution to the relativistic Boltzmann equation \eqref{rbe3}--\eqref{rbe4}.
\end{proof}


\section{Appendix}\setcounter{equation}{0}
\label{sect5}

\begin{lem}[\cite{LI-6}]\label{rbepw-Ap-1}
For $s\in[0,t]$, $A^2\geq1+t$ and $\alpha\geq0$,
\bma
e^{-\frac{A^2}{D(1+s)}}&\leq C_{\alpha}\(\frac{1+s}{1+t}\)^{\alpha}e^{-\frac{A^2}{D(1+t)}},\\
\(1+\frac{A^2}{1+s}\)^{-\alpha}&\leq2^{\alpha}\(\frac{1+t}{1+s}\)^{-\alpha}\(1+\frac{A^2}{1+t}\)^{-\alpha}.
\ema
\end{lem}

\begin{lem}[\cite{Wang-1}]\label{rbepw-Ap-2}
For any $a=a(t,s,x)\geq0$, $b=b(t,s,x)\geq0$,
\be
\int_{\R^n}\(1+\frac{(|y|-a)^2}{1+b}\)^{-n}dy\leq C\((1+b)^{\frac{n}{2}}+(1+b)^{\frac{1}{2}}|a|^{n-1}\). \label{con1}
\ee
\end{lem}

Then, we provide the proofs of Lemmas \ref{rbepw3}--\ref{rbepw7}.



\begin{proof}[\textbf{The proof of Lemma \ref{rbepw3}}]
First, we prove \eqref{rbepw301}. It's easy to verify that
\be \(1+\frac1{\eps}\)|x-y|^2+(1+\eps)|y|^2\ge |x|^2, \quad \forall \eps>0. \label{tri-1}\ee
By taking $\eps=1/4$ in \eqref{tri-1}, we obtain
$$\frac{15}{2}|x-y|^2+\frac{15}{8}|y|^2\ge \frac{3}{2}|x|^2 .$$
Thus, it follows that for $D_1\ge 8D$,
\bma
&\quad I^{\alpha,\beta}(t,x;0,t;0,0,D,D_1)\nnm\\
&\le e^{-\frac{3|x|^2}{2D_1(1+t)}}\int_0^t\intr (1+t-s)^{-\alpha}e^{-\frac{|x-y|^2}{2D_1(1+t-s)}}(1+s)^{-\beta}e^{-\frac{|y|^2}{8D_1(1+s)}}dyds\nnm\\
&\leq Ce^{-\frac{3|x|^2}{2D_1(1+t)}}\((1+t)^{-\alpha}\int_{0}^{\frac{t}2}(1+s)^{-\beta+\frac{3}{2}}ds+(1+t)^{-\beta} \int_{\frac{t}2}^t(1+t-s)^{-\alpha+\frac{3}{2}}ds\)\nnm\\
&\leq C\((1+t)^{-\alpha}\Gamma_{\beta-\frac{3}{2}}(t)+(1+t)^{-\beta}\Gamma_{\alpha-\frac{3}{2}}(t)\)e^{-\frac{3|x|^2}{2D_1(1+t)}},\label{rbepw3012}
\ema
which proves \eqref{rbepw301}.

Next, we prove \eqref{rbepw302}. For $|x|\leq\sqrt{1+t}$ and $||x|-\lambda t|\leq\sqrt{1+t}$, we have
\bma
&\quad I^{\alpha,\beta}(t,x;0,t;0,\lambda,D,D_1)\nnm\\
&\leq C(1+t)^{\alpha}\int_0^{\frac{t}{2}}(1+s)^{-\beta+\frac{5}{2}}ds+C(1+t)^{-\beta}\int_0^{\frac{t}{2}}(1+s)^{-\alpha+\frac{3}{2}}ds\nnm\\
&\leq C\((1+t)^{-\alpha}\Gamma_{\beta-\frac{5}{2}}(t)+(1+t)^{-\beta}\Gamma_{\alpha-\frac{5}{2}}(t)\)\nnm\\
&\leq C\((1+t)^{-\alpha}\Gamma_{\beta-\frac{5}{2}}(t)+(1+t)^{-\beta}\Gamma_{\alpha-\frac{5}{2}}(t)\)
\left\{\bal e^{-\frac{3|x|^2}{2D_1(1+t)}}, & |x|\leq\sqrt{1+t},\\
e^{-\frac{3(|x|-\lambda t)^2}{2D_1(1+t)}}, & ||x|-\lambda t|\leq\sqrt{1+t},
\ea\right.\label{rbepw3020}
\ema
where we have used Lemma \ref{rbepw-Ap-2}.

Note that
\be \(1+\frac1{\eps}\)|x-y|^2+(1+\eps)(|y|-\lambda s)^2\ge (|x|-\lambda s)^2, \quad \forall \eps>0. \label{tri-2}\ee
For $\sqrt{1+t}\leq|x|\leq\lambda t-\sqrt{1+t}$, we have
\bma
&\quad I^{\alpha,\beta}(t,x;0,t;0,\lambda,D,D_1)\nnm\\
&=\(\int_0^{\frac{|x|}{16\lambda}}+\int^{\frac{15t}{16}+\frac{|x|}{16\lambda}}_{\frac{|x|}{16\lambda}}+\int_{\frac{15t}{16}+\frac{|x|}{16\lambda}}^{t}\) \int_{\R^3}\frac{e^{-\frac{|x-y|^2}{D(1+t-s)}}}{ (1+t-s)^{\alpha}} \frac{e^{-\frac{2(|y|-\lambda s)^2}{D_1(1+s)}}}{(1+s)^{\beta}}dyds\nnm\\
&=:I_1+I_2+I_3.\label{rbepw3021}
\ema

For $I_1$ and $I_3$, by taking $\eps=1/8$ in \eqref{tri-2} we obtain
$$\frac{384}{25}|x-y|^2+\frac{48}{25}(|y|-\lambda s)^2\ge \frac{128}{75}(|x|-\lambda s)^2
\ge \left\{\bal \frac32|x|^2, & s\le \frac{|x|}{16\lambda},\\
\frac32(|x|-\lambda t)^2, & s\ge \frac{15t}{16}+\frac{|x|}{16\lambda}.
\ea\right.$$
Thus, it follows that for $D_1\ge 16D$,
\bma
I_1&\le e^{-\frac{3|x|^2}{2D_1(1+t)}}\int_0^{\frac{|x|}{16\lambda}}\intr\frac{e^{-\frac{16|x-y|^2}{25D_1(1+t-s)}}}{ (1+t-s)^{\alpha}} \frac{e^{-\frac{2(|y|-\lambda s)^2}{25D_1(1+s)}}}{(1+s)^{\beta}}dyds\nnm\\
&\le C(1+t)^{-\alpha}e^{-\frac{3|x|^2}{2D_1(1+t)}}\int_0^{\frac{|x|}{16\lambda}} (1+s)^{-\beta+\frac{5}{2}}ds\nnm\\
&\leq C(1+t)^{-\alpha}\Gamma_{\beta-\frac{5}{2}}(t)e^{-\frac{3|x|^2}{2D_1(1+t)}},\label{rbepw3022}
\\
I_3&\leq  e^{-\frac{3(|x|-\lambda t)^2}{2D_1(1+t)}}\int_{\frac{15t}{16}+\frac{|x|}{16\lambda}}^{t}\intr\frac{e^{-\frac{16|x-y|^2}{25D_1(1+t-s)}}}{ (1+t-s)^{\alpha}} \frac{e^{-\frac{2(|y|-\lambda s)^2}{25D_1(1+s)}}}{(1+s)^{\beta}}dyds\nnm\\
&\leq C(1+t)^{-\beta}e^{-\frac{3(|x|-\lambda t)^2}{2D_1(1+t)}}\int_{\frac{15t}{16}+\frac{|x|}{16\lambda}}^{t}(1+t-s)^{-\alpha+\frac{3}{2}} ds\nnm\\
&\leq C(1+t)^{-\beta}\Gamma_{\alpha-\frac{5}{2}}(t)e^{-\frac{3(|x|-\lambda t)^2}{2D_1(1+t)}},\label{rbepw3023}
\ema
where we have used Lemma \ref{rbepw-Ap-2}.

For $\frac{|x|}{16\lambda}\leq s\leq\frac{15t}{16}+\frac{|x|}{16\lambda}$, it holds that
$$
(1+s)^{-\beta}\leq \(1+\frac{|x|}{16\lambda}\)^{-\beta},\quad (1+t-s)^{-\alpha}\leq\(1+\frac{\lambda t-|x|}{16\lambda}\)^{-\alpha}.
$$
Thus,
\bma
I_2&=\(\int^{\frac{t}{2}}_{\frac{|x|}{16\lambda}}+\int^{\frac{15t}{16}+\frac{|x|}{16\lambda}}_{\frac{t}{2}}\)\int_{\R^3} \frac{e^{-\frac{|x-y|^2}{D(1+t-s)}}}{ (1+t-s)^{\alpha}} \frac{e^{-\frac{2(|y|-\lambda s)^2}{D_1(1+s)}}}{(1+s)^{\beta}}dyds\nnm\\
&\leq C(1+t)^{-\alpha}\(1+|x|\)^{-\beta}\int^{t}_{0}\int_{\R^3}e^{-\frac{|x-y|^2}{D(1+t)}}e^{-\frac{2(|y|-\lambda s)^2}{D_1(1+t)}}dyds\nnm\\
&\quad+C(1+t)^{-\beta}\(1+\lambda t-|x|\)^{-\alpha}\int^{t}_{0}\int_{\R^3}e^{-\frac{|x-y|^2}{D(1+t)}}e^{-\frac{2(|y|-\lambda s)^2}{D_1(1+t)}}dyds\nnm\\
&\leq C(1+t)^{-\alpha+2}\(1+|x|\)^{-\beta}+C(1+t)^{-\beta+2}\(1+\lambda t-|x|\)^{-\alpha},\label{rbepw3024}
\ema
where we have used
\bmas
&\quad\int_0^t\int_{\R^3}e^{-\frac{|x-y|^2}{D(1+t-s)}}e^{-\frac{2(|y|-\lambda s)^2}{D_1(1+t)}}dyds \nnm\\
&=\int_{\R^3}e^{-\frac{|x-y|^2}{D(1+t)}}\int_0^te^{-\frac{2(|y|-\lambda s)^2}{D_1(1+t)}}dsdy\nnm\\
&\leq C(1+t)^{\frac{1}{2}}\int_{\R^3}e^{-\frac{|x-y|^2}{D_1(1+t)}}dy\leq C(1+t)^2.
\emas

For $|x|\geq\lambda t$, by taking $\eps=1/4$ in \eqref{tri-2} we obtain
$$\frac{15}{2}|x-y|^2+\frac{15}{8}(|y|-\lambda s)^2\ge \frac{3}{2}(|x|-\lambda t)^2 .$$
Thus, it follows that for $D_1\ge 8D$,
\bma
&\quad I^{\alpha,\beta}(t,x;0,t;0,\lambda,D,D_1)\nnm\\
&\leq Ce^{-\frac{3(|x|-\lambda t)^2}{2D_1(1+t)}}\int_0^t\intr (1+t-s)^{-\alpha}e^{-\frac{|x-y|^2}{2D_1(1+t-s)}}(1+s)^{-\beta}e^{-\frac{(|y|-\lambda s)^2}{8D_1(1+s)}}dyds\nnm\\
&\leq C\((1+t)^{-\alpha}\Gamma_{\beta-\frac{5}{2}}(t)+(1+t)^{-\beta}\Gamma_{\alpha-\frac{5}{2}}(t)\)e^{-\frac{3(|x|-\lambda t)^2}{2D_1(1+t)}}.\label{rbepw3025}
\ema
 By combining \eqref{rbepw3021}--\eqref{rbepw3025}, we can obtain \eqref{rbepw302}.

By changing variable $z=x-y$ and $\tau=t-s$, we can obtain
\bma
&\quad I^{\alpha,\beta}(t,x;0,t;\lambda,0,D,D_1)\nnm\\
&=\int_0^t\int_{\R^3}(1+\tau)^{-\alpha}e^{-\frac{(|z|-\lambda\tau)^2}{D(1+\tau)}}(1+t-\tau)^{-\beta}e^{-\frac{2|x-z|^2}{D_1(1+t-\tau)}}dzd\tau\nnm\\
&=I^{\beta,\alpha}(t,x;0,t;0,\lambda,D_1/2,2D).
\ema
By the similar arguments as \eqref{rbepw3021}--\eqref{rbepw3025}, we can obtain \eqref{rbepw302}.

Finally, we prove \eqref{rbepw303}. For $|x|\leq\sqrt{1+t}$ and $||x|-\lambda t|\leq\sqrt{1+t}$, we have
\bma
&\quad I^{\alpha,\beta}(t,x;0,t;\lambda,\lambda,D,D_1)\nnm\\
&\leq C(1+t)^{\alpha}\int_0^{\frac{t}{2}}(1+s)^{-\beta+\frac{5}{2}}ds+C(1+t)^{-\beta}\int_0^{\frac{t}{2}}(1+s)^{-\alpha+\frac{5}{2}}ds\nnm\\
&\leq C\((1+t)^{-\alpha}\Gamma_{\beta-\frac{5}{2}}(t)+(1+t)^{-\beta}\Gamma_{\alpha-\frac{5}{2}}(t)\).\label{rbepw3030}
\ema

Note that
\bma &\quad \(1+\frac1{\eps}\)(|x-y|-\lambda (t-s))^2+(1+\eps)(|y|-\lambda s)^2\nnm\\
&\ge (\lambda t-|x-y|+|y|-2\lambda s)^2, \quad \forall \eps>0. \label{tri-3}
\ema
For $\sqrt{1+t}\leq|x|\leq\lambda t-\sqrt{1+t}$, we have
\bma
&\quad I^{\alpha,\beta}(t,x;0,t;\lambda,\lambda,D,D_1)\nnm\\
&=\(\int^{\frac{\lambda t-|x|}{32\lambda}}_0+\int_{\frac{\lambda t-|x|}{32\lambda}}^{t-\frac{\lambda t-|x|}{32\lambda}}+\int^{t}_{t-\frac{\lambda t-|x|}{32\lambda}}\) \int_{\R^3}\frac{e^{-\frac{(|x-y|-\lambda(t-s))^2}{D(1+t-s)}}}{(1+t-s)^{\alpha}}\frac{e^{-\frac{2(|y|-\lambda s)^2}{D_1(1+s)}}}{(1+s)^{\beta}}dyds\nnm\\
&=:I_4+I_5+I_6.\label{rbepw3031}
\ema
For $I_4$ and $I_6$, due to the fact that
\bmas
\lambda t-|x-y|+|y|-2\lambda s&\ge \lambda t-|x|-2\lambda s,\\
-\lambda t+|x-y|-|y|+2\lambda s&=\lambda t+|x-y|-|y|-2\lambda (t-s)\ge \lambda t-|x|-2\lambda(t-s),
\emas
we can obtain by taking $\eps=1/8$ in \eqref{tri-3} that
\bmas
 &\quad \frac{384}{25}(|x-y|-\lambda (t-s))^2+\frac{48}{25}(|y|-\lambda s)^2\nnm\\
&\ge
\left\{\bal \frac{128}{75}(\lambda t-|x|-2\lambda s)^2\ge \frac32(\lambda t-|x|)^2, & s\le \frac{\lambda t-|x|}{32\lambda},\\
 \frac{128}{75}(\lambda t-|x|-2\lambda(t-s))^2\ge \frac32(\lambda t-|x|)^2, & s\ge t-\frac{\lambda t-|x|}{32\lambda}.
 \ea\right.
\emas
Thus, it follows that for $D_1\ge 16D$,
\bma
I_4&\le e^{-\frac{3(|x|-\lambda t)^2}{2D_1(1+t)}}\int^{\frac{\lambda t-|x|}{32\lambda}}_0\intr  \frac{e^{-\frac{16(|x-y|-\lambda(t-s))^2}{25D_1(1+t-s)}}}{(1+t-s)^{\alpha}}\frac{e^{-\frac{2(|y|-\lambda s)^2}{25D_1(1+s)}}}{(1+s)^{\beta}}dyds\nnm\\
&\leq C(1+t)^{-\alpha}e^{-\frac{3(|x|-\lambda t)^2}{2D_1(1+t)}}\int^{\frac{\lambda t-|x|}{32\lambda}}_0 (1+s)^{-\beta+\frac{5}{2}}ds\nnm\\\
&\leq C(1+t)^{-\alpha}\Gamma_{\beta-\frac{5}{2}}(t)e^{-\frac{3(|x|-\lambda t)^2}{2D_1(1+t)}},\label{rbepw3032}
\\
I_6&\le e^{-\frac{3(|x|-\lambda t)^2}{2D_1(1+t)}}\int^{t}_{t-\frac{\lambda t-|x|}{32\lambda}}\intr \frac{e^{-\frac{16(|x-y|-\lambda(t-s))^2}{25D_1(1+t-s)}}}{(1+t-s)^{\alpha}}\frac{e^{-\frac{2(|y|-\lambda s)^2}{25D_1(1+s)}}}{(1+s)^{\beta}}dyds\nnm\\
&\leq C(1+t)^{-\beta}e^{-\frac{3(|x|-\lambda t)^2}{2D_1(1+t)}}\int^{t}_{t-\frac{\lambda t-|x|}{32\lambda}}(1+t-s)^{-\alpha+\frac{5}{2}} ds\nnm\\
&\leq C(1+t)^{-\beta}\Gamma_{\alpha-\frac{5}{2}}(t)e^{-\frac{3(|x|-\lambda t)^2}{2D_1(1+t)}}.\label{rbepw3033}
\ema

For $I_5$, it follows that
\bma
I_5&=\(\int_{\frac{\lambda t-|x|}{32\lambda}}^{\frac{t}{2}}+\int_{\frac{t}{2}}^{t-\frac{\lambda t-|x|}{32\lambda}}\)\int_{\R^3}\frac{e^{-\frac{(|x-y|-\lambda(t-s))^2}{D(1+t-s)}}}{(1+t-s)^{\alpha}}\frac{e^{-\frac{2(|y|-\lambda s)^2}{D_1(1+s)}}}{(1+s)^{\beta}}dyds\nnm\\
&\leq C(1+t)^{-\alpha} \int_{\frac{\lambda t-|x|}{32\lambda}}^{\frac{t}{2}}\int_{\R^3} (1+s)^{-\beta}e^{-\frac{2(|y|-\lambda s)^2}{D_1(1+s)}}dyds\nnm\\
&\quad+C(1+t)^{-\beta} \int_{\frac{t}{2}}^{t-\frac{\lambda t-|x|}{32\lambda}}\int_{\R^3}(1+t-s)^{-\alpha}e^{-\frac{(|x-y|-\lambda(t-s))^2}{D(1+t-s)}} dyds\nnm\\
&\leq C(1+t)^{-\alpha} \int_{\frac{\lambda t-|x|}{32\lambda}}^{\frac{t}{2}} (1+s)^{-\beta+\frac52} ds
 +C(1+t)^{-\beta} \int_{\frac{t}{2}}^{t-\frac{\lambda t-|x|}{32\lambda}} (1+t-s)^{-\alpha+\frac52} ds\nnm\\
&\leq C (1+t)^{-\alpha+1}(1+\lambda t-|x|)^{-\beta+\frac52}+C(1+t)^{-\beta+1}(1+\lambda t-|x|)^{-\alpha+\frac52}  .\label{rbepw3034}
\ema

For $|x|\geq\lambda t$, since
$$\frac{15}{2}(|x-y|-\lambda(t-s))^2+\frac{15}{8}(|y|-\lambda s)^2\ge \frac{3}{2}(|x|-\lambda t)^2 ,$$
it follows that for $D_1\ge 8D$,
\bma
&\quad I^{\alpha,\beta}(t,x;0,t;\lambda,\lambda,D,D_1)\nnm\\
&\leq  e^{-\frac{3(|x|-\lambda t)^2}{2D_1(1+t)}}\int^t_0\int_{\R^3}\frac{e^{-\frac{(|x-y|-\lambda(t-s))^2}{8D_1(1+t-s)}}}{(1+t-s)^{\alpha}}\frac{e^{-\frac{(|y|-\lambda s)^2}{8D_1(1+s)}}}{(1+s)^{\beta}}dyds\nnm\\
&\leq C\((1+t)^{-\alpha}\Gamma_{\beta-\frac{5}{2}}(t)+(1+t)^{-\beta}\Gamma_{\alpha-\frac{5}{2}}(t)\)e^{-\frac{3(|x|-\lambda t)^2}{2D_1(1+t)}}.\label{rbepw3036}
\ema
By combining \eqref{rbepw3031}--\eqref{rbepw3036}, we can obtain \eqref{rbepw303}.
\end{proof}


\begin{proof}[\textbf{The proof of Lemma \ref{rbepw4}}]
 Firstly, we prove \eqref{rbepw401}. For $|x|\leq\sqrt{1+t}$, we have
\bma
&\quad J^{\alpha,\beta}(t,x;0,t;\lambda,0,D_1)\nnm\\
&\leq C(1+t)^{-\alpha}\int_0^{\frac{t}{2}}(1+s)^{-\beta+\frac{3}{2}}ds+C(1+t)^{-\beta}\ln(1+t)\int_{\frac{t}{2}}^t(1+t-s)^{-\alpha+\frac{3}{2}}ds\nnm\\
&\leq C\((1+t)^{-\alpha}\Gamma_{\beta-\frac{3}{2}}(t) +(1+t)^{-\beta}\ln(1+t)\Gamma_{\alpha-\frac{3}{2}}(t)\)B_{\frac{3}{2}}(t,|x|).\label{rbepw4011}
\ema
For $\sqrt{1+t}\le |x|\leq\lambda t$, we split $y$ into $|y|\leq\frac{|x|}{2}$ and $|y|\geq\frac{|x|}{2}$. Thus, 
\bma
&\quad J^{\alpha,\beta}(t,x;0,t;\lambda,0,D_1)\nnm\\
&=\int_0^t\(\int_{\{|y|\leq\frac{|x|}{2}\}}+\int_{\{|y|\geq\frac{|x|}{2}\}}\) \frac{B_{\frac{3}{2}}(t-s,|x-y|)}{(1+t-s)^{\alpha}}\frac{e^{-\frac{2|y|^2}{D_1(1+s)}}}{(1+s)^{\beta}} 1_{\{|x-y|\le \lambda(t-s)\}}dyds\nnm\\
&\leq C(1+t)^{-\frac32}B_{\frac{3}{2}}(t,|x|)\int_0^t(1+t-s)^{-\alpha+\frac32}(1+s)^{-\beta+\frac32} ds\nnm\\
&\quad+C(1+t)^{-\frac32}\ln(1+t)e^{-\frac{|x|^2}{4D_1(1+t)}}\int_0^t(1+t-s)^{-\alpha+\frac32}(1+s)^{-\beta+\frac32} ds\nnm\\
&\leq C\((1+t)^{-\alpha}\Gamma_{\beta-\frac{3}{2}}(t)+(1+t)^{-\beta}\ln(1+t)\Gamma_{\alpha-\frac{3}{2}}(t)\)B_{\frac{3}{2}}(t,|x|).
\ema

For $|x|>\lambda t$ and $|x-y|\leq\lambda(t-s)$, we have $|y|\geq|x|-|x-y|\geq|x|-\lambda t$. Thus,
\bma
&\quad J^{\alpha,\beta}(t,x;0,t;\lambda,0,D_1)\nnm\\
&\leq e^{-\frac{3(|x|-\lambda t)^2}{2D_1(1+t)}}\int_0^{t}\int_{\{|x-y|\leq\lambda(t-s)\}}\frac{B_{\frac{3}{2}}(t-s,|x-y|)}{(1+t-s)^{\alpha}}\frac{e^{-\frac{|y|^2}{2D_1(1+s)}}}{(1+s)^{\beta}}dyds\nnm\\
&\leq Ce^{-\frac{3(|x|-\lambda t)^2}{2D_1(1+t)}}\((1+t)^{-\alpha}\int_0^{\frac{t}{2}}(1+s)^{-\beta+\frac{3}{2}}ds+(1+t)^{-\beta}\ln(1+t)\int^t_{\frac{t}{2}}(1+t-s)^{-\alpha+\frac32}ds\)\nnm\\
&\leq C\((1+t)^{-\alpha}\Gamma_{\beta-\frac{3}{2}}(t)+(1+t)^{-\beta}\ln(1+t)\Gamma_{\alpha-\frac32}(t)\)e^{-\frac{3(|x|-\lambda t)^2}{2D_1(1+t)}}.\label{rbepw4012}
\ema
By combining \eqref{rbepw4011}--\eqref{rbepw4012}, we can obtain \eqref{rbepw401}.

By changing variable $z=x-y$ and $\tau=t-s$, we have
\bma
&\quad L^{\alpha,\beta}(t,x;0,t;\lambda,0,0,D)\nnm\\
&=\int_0^t\int_{\{|y|\leq\lambda s\}}(1+t-s)^{-\alpha}e^{-\frac{|x-y|^2}{D(1+t-s)}}(1+s)^{-\beta}B_3(s,|y|)dyds\nnm\\
&=\int_0^t\int_{\{|x-z|\leq\lambda (t-\tau)\}}(1+\tau)^{-\alpha}e^{-\frac{|z|^2}{D(1+\tau)}}(1+t-\tau)^{-\beta}B_3(t-\tau,|x-z|)dzd\tau.
\ema
By the similar arguments as \eqref{rbepw4011}--\eqref{rbepw4012}, we can obtain \eqref{rbepw401j1} for $D_1\ge 2D$.

Finally, we prove \eqref{rbepw402}. For $|x|\leq\sqrt{1+t}$ and $||x|-\lambda t|\leq\sqrt{1+t}$, we have
\bma
&\quad J^{\alpha,\beta}(t,x;0,t;\lambda,\lambda,D_1)\nnm\\
&\leq C\int_0^{\frac{t}{2}}(1+t-s)^{-\alpha}(1+s)^{-\beta+\frac{5}{2}}ds+C\int^t_{\frac{t}{2}}(1+t-s)^{-\alpha+\frac{5}{2}}(1+s)^{-\beta}ds\nnm\\
&\leq C\((1+t)^{-\alpha}\Gamma_{\beta-\frac{5}{2}}(t)+(1+t)^{-\beta}\Gamma_{\alpha-\frac{5}{2}}(t)\).
 \label{rbepw4021}
\ema

For $\sqrt{1+t}\leq|x|\leq\lambda t-\sqrt{1+t}$, we have
\bma
&\quad J^{\alpha,\beta}(t,x;0,t;\lambda,D_1)\nnm\\
&=\(\int_0^{\frac{|x|}{2\lambda}}+\int^{\frac{t}{2}+\frac{|x|}{2\lambda}}_{\frac{|x|}{2\lambda}}+\int_{\frac{t}{2}+\frac{|x|}{2\lambda}}^t\)\int_{\{|x-y|\leq\lambda(t-s)\}} \frac{B_{\frac{3}{2}}(t-s,|x-y|)}{(1+t-s)^{\alpha}}
\frac{e^{-\frac{2(|y|-\lambda s)^2}{D_1(1+s)}}}{(1+s)^{\beta}}dyds\nnm\\
&=:I_1+I_2+I_3.\label{rbepw4022}
\ema
For $I_1$, we split $y$ into $|x-y|\leq\frac{|x|}{4}$ and $|x-y|\geq\frac{|x|}{4}$. If $|x-y|\leq\frac{|x|}{4}$ and $s\leq\frac{|x|}{2\lambda}$, then we have $|y|-\lambda s\geq|x|-|x-y|-\lambda s\geq\frac{|x|}{4}$. Thus,
\bma
I_1&\leq\int_0^{\frac{|x|}{2\lambda}}\(\int_{\{|x-y|\leq\frac{|x|}{4}\}}+\int_{\{|x-y|\geq\frac{|x|}{4}\}}\) \frac{B_{\frac{3}{2}}(t-s,|x-y|)}{(1+t-s)^{\alpha}}
\frac{e^{-\frac{2(|y|-\lambda s)^2}{D_1(1+s)}}}{(1+s)^{\beta}}dyds\nnm\\
&\leq C(1+t)^{-\frac{5}{2}}e^{-\frac{|x|^2}{8D_1(1+t)}}\int_0^{\frac{|x|}{2\lambda}}(1+t-s)^{-\alpha+\frac{3}{2}}\ln(1+t-s)(1+s)^{-\beta+\frac{5}{2}}ds\nnm\\
&\quad+C(1+t)^{-\frac{3}{2}}B_{\frac{3}{2}}(t,|x|)\int_0^{\frac{|x|}{2\lambda}}(1+t-s)^{-\alpha+\frac{3}{2}}(1+s)^{-\beta+\frac{5}{2}}ds\nnm\\
&\leq C(1+t)^{-\alpha}\Gamma_{\beta-\frac{5}{2}}(t)B_{\frac{3}{2}}(t,|x|),\label{rbepw4023}
\ema
where we have used Lemmas \ref{rbepw-Ap-1}--\ref{rbepw-Ap-2}.

For $I_3$, we split $y$ into $|x-y|\leq\frac{\lambda t-|x|}{4}$ and $|x-y|\geq\frac{\lambda t-|x|}{4}$. If $|x-y|\leq\frac{\lambda t-|x|}{4}$ and $s\geq\frac{t}{2}+\frac{|x|}{4\lambda}$, then we have $\lambda s-|y|\geq\lambda s-|x|-|x-y|\geq\frac{\lambda t-|x|}{4}$. Thus,
\bma
I_3&\leq\int_{\frac{t}{2}+\frac{|x|}{2\lambda}}^t\(\int_{\{|x-y|\leq\frac{\lambda t-|x|}{4}\}}+\int_{\{|x-y|\geq\frac{\lambda t-|x|}{4}\}}\) \frac{B_{\frac{3}{2}}(t-s,|x-y|)}{(1+t-s)^{\alpha}}
\frac{e^{-\frac{2(|y|-\lambda s)^2}{D_1(1+s)}}}{(1+s)^{\beta}}dyds\nnm\\
&\leq C(1+t)^{-\frac{5}{2}}e^{-\frac{(|x|-\lambda t)^2}{8D_1(1+t)}}\int_{\frac{t}{2}+\frac{|x|}{2\lambda}}^t(1+t-s)^{-\alpha+\frac{3}{2}}\ln(1+t-s)(1+s)^{-\beta+\frac{5}{2}}ds\nnm\\
&\quad+C(1+t)^{-\frac{3}{2}}B_{\frac{3}{2}}(t,|x|-\lambda t)\int_{\frac{t}{2}+\frac{|x|}{2\lambda}}^t(1+t-s)^{-\alpha+\frac{3}{2}}(1+s)^{-\beta+\frac{5}{2}}ds\nnm\\
&\leq C\((1+t)^{-\beta}\Gamma_{\alpha-\frac{5}{2}}(t)+(1+t)^{-\beta+1}\Gamma_{\alpha-\frac{3}{2}}(t)\)B_{\frac{3}{2}}(t,|x|-\lambda t),\label{rbepw4024}
\ema
where we have used Lemmas \ref{rbepw-Ap-1}--\ref{rbepw-Ap-2}.

For $I_2$, it follows that
\bma
I_2&=\(\int^{\frac{t}{2}}_{\frac{|x|}{2\lambda}}+\int^{\frac{t}{2}+\frac{|x|}{2\lambda}}_{\frac{t}{2}}\)\int_{\{|x-y|\leq\lambda(t-s)\}} \frac{B_{\frac{3}{2}}(t-s,|x-y|)}{(1+t-s)^{\alpha}}
\frac{e^{-\frac{2(|y|-\lambda s)^2}{D_1(1+s)}}}{(1+s)^{\beta}}dyds\nnm\\
&\leq C(1+t)^{-\alpha}\(1+|x|\)^{-\beta}\int^{\frac{t}{2}}_{\frac{|x|}{2\lambda}}\int_{\{|x-y|\leq\lambda t\}} B_{\frac{3}{2}}(t,|x-y|)e^{-\frac{2(|y|-\lambda s)^2}{D_1(1+t)}}dyds\nnm\\
&\quad+C(1+t)^{-\beta}\(1+\lambda t-|x|\)^{-\alpha}\int_{\frac{t}{2}}^{\frac{t}{2}+\frac{|x|}{2\lambda}}\int_{\{|x-y|\leq\lambda t\}} B_{\frac{3}{2}}(t,|x-y|)e^{-\frac{2(|y|-\lambda s)^2}{D_1(1+t)}}dyds\nnm\\
&\leq C(1+t)^{-\alpha+2}\ln(1+t)\(1+|x|\)^{-\beta}+C(1+t)^{-\beta+2}\ln(1+t)\(1+\lambda t-|x|\)^{-\alpha},\label{rbepw4025}
\ema
where we have used
\bmas
&\quad\int_0^t\int_{\{|x-y|\leq\lambda t\}}B_{\frac{3}{2}}(t,|x-y|)e^{-\frac{2(|y|-\lambda s)^2}{D_1(1+t)}}dyds\\
&\leq\int_{\{|x-y|\leq\lambda t\}}\(1+\frac{|x-y|^2}{1+t}\)^{-\frac{3}{2}}\int_{\frac{|x-y|}{\lambda}}^te^{-\frac{2(|y|-\lambda s)^2}{D_1(1+t)}}dsdy\\
&\leq C(1+t)^2\ln(1+t).
\emas

For $|x|\geq\lambda t+\sqrt{1+t}$ and $|x-y|\leq\lambda(t-s)$, we have $|y|-\lambda s\geq|x|-\lambda t-|x-y|+\lambda t-\lambda s\geq|x|-\lambda t$. Thus,
\bma
&\quad J^{\alpha,\beta}(t,x;0,t;\lambda,\lambda,D_1)\nnm\\
&\leq C(1+t)^{-\frac{5}{2}}e^{-\frac{2(|x|-\lambda t)^2}{D_1(1+t)}}\int_0^t(1+t-s)^{-\alpha+\frac{3}{2}}\ln(1+t-s)(1+s)^{-\beta+\frac{5}{2}}ds\nnm\\
&\leq C\((1+t)^{-\alpha}\Gamma_{\beta-\frac{5}{2}}(t)+(1+t)^{-\beta}\Gamma_{\alpha-\frac{5}{2}}(t)\)e^{-\frac{2(|x|-\lambda t)^2}{D_1(1+t)}},\label{rbepw4026}
\ema
where we have used Lemma \ref{rbepw-Ap-1}. By combining \eqref{rbepw4021}--\eqref{rbepw4026}, we obtain \eqref{rbepw402}.
\end{proof}


\begin{proof}[\textbf{The proof of Lemma \ref{rbepw5}}]
We first prove \eqref{rbepw501}. For $|x|\leq \sqrt{1+t}$, we have
\bma
&\quad K^{\alpha,\beta}(t,x;0,t;\lambda,0)\nnm\\
&\leq C\int_0^{\frac{t}{2}}(1+t-s)^{-\alpha}(1+s)^{-\beta+\frac{3}{2}}ds+C\ln(1+t)\int_{\frac{t}{2}}^t(1+t-s)^{-\alpha+\frac{3}{2}}(1+s)^{-\beta}ds\nnm\\
&\leq C\((1+t)^{-\alpha}\Gamma_{\beta-\frac{3}{2}}(t)+(1+t)^{-\beta}\Gamma_{\alpha-\frac{3}{2}}(t)\)\ln(1+t)B_{\frac{3}{2}}(t,|x|).\label{rbepw5011}
\ema

For $\sqrt{1+t}\leq|x|\leq\lambda t$, we split $y$ into $|y|\leq\frac{|x|}{2}$ and $|y|\geq\frac{|x|}{2}$. If $|y|\leq\frac{|x|}{2}$, then we have $|x-y|\geq||x|-|y||\geq\frac{|x|}{2}$. Thus,
\bma
&\quad K^{\alpha,\beta}(t,x;0,t;\lambda,0)\nnm\\
&\leq C(1+t)^{-\frac{3}{2}}B_{\frac{3}{2}}(t,|x|)\int_0^t(1+t-s)^{-\alpha+\frac{3}{2}}(1+s)^{-\beta+\frac{3}{2}}ds\nnm\\
&\quad+C(1+t)^{-\frac{3}{2}}\ln(1+t)B_{3}(t,|x|)\int_0^t(1+t-s)^{-\alpha+\frac{3}{2}}(1+s)^{-\beta+\frac{3}{2}}ds\nnm\\
&\leq C\((1+t)^{-\alpha}\Gamma_{\beta-\frac{3}{2}}(t)+(1+t)^{-\beta}\Gamma_{\alpha-\frac{3}{2}}(t)\)\ln(1+t)B_{\frac{3}{2}}(t,|x|).\label{rbepw5011-1}
\ema

For $|x|>\lambda t$, we have $\{|x-y|\leq\lambda(t-s)\}\cap\{|y|\leq\lambda s\}=\phi$. Thus,
$
 K^{\alpha,\beta}(t,x;0,t;\lambda,0)=0
$ for $|x|>\lambda t$.
This together with \eqref{rbepw5011} and \eqref{rbepw5011-1} implies \eqref{rbepw501}.

Now, we prove \eqref{rbepw502}. For $|x|\leq\sqrt{1+t}$ and $\lambda t-\sqrt{1+t}\leq|x|\leq\lambda t$,
\bma
&\quad K^{\alpha,\beta}(t,x;0,t;\lambda,\lambda)\nnm\\
&\leq C\int_0^{\frac{t}{2}}(1+t-s)^{-\alpha}(1+s)^{-\beta+\frac{5}{2}}ds+C\int^t_{\frac{t}{2}}(1+t-s)^{-\alpha+\frac{3}{2}}\ln(1+t-s)(1+s)^{-\beta}ds\nnm\\
&\leq C\((1+t)^{-\alpha}\Gamma_{\beta-\frac{5}{2}}(t)+(1+t)^{-\beta}\Gamma_{\alpha-\frac{5}{2}}(t)\)
\left\{\bal B_{\frac{3}{2}}(t,|x|), & |x|\leq\sqrt{1+t},\\
B_{\frac{3}{2}}(t,|x|-\lambda t), & \lambda t-\sqrt{1+t}\leq|x|\leq\lambda t.
\ea\right.\label{rbepw5021}
\ema

For $\sqrt{1+t}\leq|x|\leq\lambda t-\sqrt{1+t}$, we have
\bma
&\quad K^{\alpha,\beta}(t,x;0,t;\lambda,\lambda)\nnm\\
&\le \(\int_0^{\frac{|x|}{2\lambda}}+\int_{\frac{|x|}{2\lambda}}^{\frac{t}{2}+\frac{|x|}{2\lambda}}+\int^t_{\frac{t}{2}+\frac{|x|}{2\lambda}}\)
\int_{\{|x-y|\leq\lambda(t-s)\}}\frac{B_{\frac{3}{2}}(t-s,|x-y|)}{(1+t-s)^{\alpha}} \frac{B_2(s,|y|-\lambda s)}{(1+s)^{\beta}} dyds\nnm\\
&=:I_1+I_2+I_3.\label{rbepw5022}
\ema

For $I_1$, we split $y$ into $|x-y|\leq\frac{|x|}{4}$ and $|x-y|\geq\frac{|x|}{4}$. If $|x-y|\leq\frac{|x|}{4}$ and $s\leq\frac{|x|}{2\lambda}$, then we have $|y|-\lambda s\geq|x|-|x-y|-\lambda s\geq\frac{|x|}{4}$. Thus,
\bma
I_1
&\leq C(1+t)^{-2}B_{2}(t,|x|)\int_0^{\frac{|x|}{2\lambda}}(1+t-s)^{-\alpha+\frac{3}{2}}\ln(1+t-s)(1+s)^{-\beta+2}ds\nnm\\
&\quad+C(1+t)^{-\frac{3}{2}}B_{\frac{3}{2}}(t,|x|)\int_0^{\frac{|x|}{2\lambda}}(1+t-s)^{-\alpha+\frac{3}{2}}(1+s)^{-\beta+\frac{5}{2}}ds\nnm\\
&\leq C(1+t)^{-\alpha}\Gamma_{\beta-\frac{5}{2}}(t)B_{\frac{3}{2}}(t,|x|).\label{rbepw5023}
\ema

For $I_3$, we split $y$ into $|y|\leq\frac{\lambda t-|x|}{4}$ and $|y|\geq\frac{\lambda t-|x|}{4}$. If $|x-y|\leq\frac{\lambda t-|x|}{4}$ and $s\geq\frac{t}{2}+\frac{|x|}{2\lambda}$, then we have $\lambda s-|y|\geq\lambda s-|x|-|x-y|\geq\frac{\lambda t-|x|}{4}$. Thus,
\bma
I_3
&\leq C(1+t)^{-2}B_{2}(t,|x|-\lambda t)\int^t_{\frac{t}{2}+\frac{|x|}{2\lambda}}(1+t-s)^{-\alpha+\frac{3}{2}}\ln(1+t-s)(1+s)^{-\beta+2}ds\nnm\\
&\quad+C(1+t)^{-\frac{3}{2}}B_{\frac{3}{2}}(t,|x|-\lambda t)\int^t_{\frac{t}{2}+\frac{|x|}{2\lambda}}(1+t-s)^{-\alpha+\frac{3}{2}}(1+s)^{-\beta+\frac{5}{2}}ds\nnm\\
&\leq C(1+t)^{-\beta+1}\Gamma_{\alpha-\frac{3}{2}}(t)B_{\frac{3}{2}}(t,|x|-\lambda t).\label{rbepw5024}
\ema

For $I_2$, it holds that
\bma
I_2&=\(\int_{\frac{|x|}{2\lambda}}^{\frac{t}{2}}+\int_{\frac{t}{2}}^{\frac{t}{2}+\frac{|x|}{2\lambda}}\)\int_{\{|x-y|\leq\lambda(t-s)\}}
\frac{B_{\frac{3}{2}}(t-s,|x-y|)}{(1+t-s)^{\alpha}} \frac{B_2(s,|y|-\lambda s)}{(1+s)^{\beta}}dyds\nnm\\
&\leq C(1+t)^{-\alpha}\(1+|x|\)^{-\beta}\int_0^t\int_{\{|x-y|\leq\lambda t\}}B_{\frac{3}{2}}(t,|x-y|)B_2(t,|y|-\lambda s)dyds\nnm\\
&\quad+C(1+t)^{-\beta}\(1+\lambda t-|x|\)^{-\alpha}\int_0^t\int_{\{|x-y|\leq\lambda t\}}B_{\frac{3}{2}}(t,|x-y|)B_2(t,|y|-\lambda s)dyds\nnm\\
&\leq C(1+t)^{-\alpha+2}\ln(1+t)\(1+|x|\)^{-\beta}+C(1+t)^{-\beta+2}\ln(1+t)\(1+\lambda t-|x|\)^{-\alpha},\label{rbepw5025}
\ema
where we have used
\bmas
&\quad\int_0^t\int_{\{|x-y|\leq\lambda t\}}B_{\frac{3}{2}}(t,|x-y|)B_2(t,|y|-\lambda s)dyds\\
&=\int_{\{|x-y|\leq\lambda t\}}\(1+\frac{|x-y|^2}{1+t}\)^{-\frac{3}{2}}\int_{\frac{|x-y|}{\lambda}}^t\(1+\frac{(|y|-\lambda s)^2}{1+t}\)^{-2}dsdy\\
&\leq C(1+t)^2\ln(1+t).
\emas

For $|x|>\lambda t$,  we have $\{|x-y|\leq\lambda(t-s)\}\cap\{|y|\leq\lambda s\}=\phi$. Thus,
$
K^{\alpha,\beta}(t,x;0,t;\lambda,\lambda)=0
$ for $|x|>\lambda t$.
This together with \eqref{rbepw5021}--\eqref{rbepw5025} implies  \eqref{rbepw502}.
\end{proof}


\begin{proof}[\textbf{The proof of Lemma \ref{rbepw6}}]
First, we prove \eqref{rbepw603}. For $|x|\leq\sqrt{1+t}$ and $||x|-\lambda t|\leq\sqrt{1+t}$, we have
\bma
&\quad L^{\alpha,\beta}(t,x;0,t;\lambda,0,\lambda,D)\nnm\\
&\le \int_{0}^t\intr (1+t-s)^{-\alpha}e^{-\frac{|x-y|^2}{D(1+t-s)}}(1+s)^{-\beta}B_{2}(s,|y|-\lambda s)dyds\nnm\\
&\leq C\((1+t)^{-\alpha}\Gamma_{\beta-\frac{5}{2}}(t)+(1+t)^{-\beta}\Gamma_{\alpha-\frac{5}{2}}(t)\).\label{rbepw6030}
\ema

For $\sqrt{1+t}\leq|x|\leq\lambda t-\sqrt{1+t}$, we have
\bma
&\quad L^{\alpha,\beta}(t,x;0,t;0,\lambda,D)\nnm\\
&=\(\int_{0}^{\frac{|x|}{2\lambda}}+\int^{\frac{t}{2}+\frac{|x|}{2\lambda}}_{\frac{|x|}{2\lambda}}+\int_{\frac{t}{2}+\frac{|x|}{2\lambda}}^t\)\int_{\{|y|\leq\lambda s\}}\frac{e^{-\frac{|x-y|^2}{D(1+t-s)}}}{(1+t-s)^{\alpha}} \frac{B_{2}(s,|y|-\lambda s)}{(1+s)^{\beta}}dyds\nnm\\
&=:I_1+I_2+I_3.\label{rbepw6031}
\ema

For $I_1$, we split $y$ into $|x-y|\leq\frac{|x|}{4}$ and $|x-y|\geq\frac{|x|}{4}$. If $|x-y|\leq\frac{|x|}{4}$ and $s\leq\frac{|x|}{2\lambda}$, then we have $|y|-\lambda  s\geq|x|-|x-y|-\lambda s\geq\frac{|x|}{4}$. Thus,
\bma
I_1
&\leq C(1+t)^{-2}B_{2}(t,|x|)\int_{0}^{\frac{|x|}{2\lambda}}(1+t-s)^{-\alpha+\frac{3}{2}}(1+s)^{-\beta+2}ds\nnm\\
&\quad+C(1+t)^{-\frac{5}{2}}e^{-\frac{|x|^2}{16D(1+t)}}\int_{0}^{\frac{|x|}{2\lambda}}(1+t-s)^{-\alpha+\frac{5}{2}}(1+s)^{-\beta+\frac{5}{2}}ds\nnm\\
&\leq C(1+t)^{-\alpha}\Gamma_{\beta-\frac{5}{2}}(t)B_{2}(t,|x|).\label{rbepw6033}
\ema

For $I_3$, we split $y$ into $|x-y|\leq\frac{\lambda t-|x|}{4}$ and $|x-y|\geq\frac{\lambda t-|x|}{4}$. If $|x-y|\leq\frac{\lambda t-|x|}{4}$ and $s\geq\frac{t}{2}+\frac{|x|}{2\lambda}$, then we have $\lambda s-|y|\geq\lambda s-|x|-|x-y|\geq\frac{\lambda t-|x|}{4}$. Thus,
\bma
I_3
&\leq C(1+t)^{-2}B_{2}(t,|x|-\lambda t)\int_{\frac{t}{2}+\frac{|x|}{2\lambda}}^t(1+t-s)^{-\alpha+\frac{3}{2}}(1+s)^{-\beta+2}ds\nnm\\
&\quad+C(1+t)^{-\frac{5}{2}}e^{-\frac{(|x|-\lambda t)^2}{16D(1+t)}}\int_{\frac{t}{2}+\frac{|x|}{2\lambda}}^t(1+t-s)^{-\alpha+\frac{5}{2}}(1+s)^{-\beta+\frac{5}{2}}ds\nnm\\
&\leq C(1+t)^{-\beta}\Gamma_{\alpha-\frac{5}{2}}(t)B_{2}(t,|x|-\lambda t).\label{rbepw6035}
\ema

For $I_2$, it holds that
\bma
I_5
&\leq C(1+t)^{-\alpha}\(1+|x|\)^{-\beta}\int^{t}_{0}\int_{\R^3}e^{-\frac{|x-y|^2}{D(1+t)}}\(1+\frac{(|y|-\lambda s)^2}{1+t}\)^{-2}dyds\nnm\\
&\quad+C(1+t)^{-\beta}\(1+\lambda t-|x|\)^{-\alpha}\int^{t}_{0}\int_{\R^3}e^{-\frac{|x-y|^2}{D(1+t)}}\(1+\frac{(|y|-\lambda s)^2}{1+t}\)^{-2}dyds\nnm\\
&\leq C(1+t)^{-\alpha+2}\(1+|x|\)^{-\beta}+C(1+t)^{-\beta+2}\(1+\lambda t-|x|\)^{-\alpha}.\label{rbepw6037}
\ema

For $|x|\geq\lambda t+\sqrt{1+t}$ and $|y|\leq\lambda s$, we have $|x-y|\geq|x|-|y| \geq|x|-\lambda t$. Thus,
\bma
&\quad L^{\alpha,\beta}(t,x;0,t;\lambda,0,\lambda,D)\nnm\\
&\leq C(1+t)^{-\frac{5}{2}}e^{-\frac{(|x|-\lambda t)^2}{D(1+t)}}\int_0^t(1+t-s)^{-\alpha+\frac{5}{2}}(1+s)^{-\beta+\frac{5}{2}}ds\nnm\\
&\leq C\((1+t)^{-\alpha}\Gamma_{\beta-\frac{5}{2}}(t)+(1+t)^{-\beta}\Gamma_{\alpha-\frac{5}{2}}(t)\)e^{-\frac{3(|x|-\lambda t)^2}{2D_1(1+t)}},\label{rbepw6039}
\ema
where   $D_1\ge \frac{3D}{2}$. By combining \eqref{rbepw6030}--\eqref{rbepw6039}, we can obtain \eqref{rbepw603}.

By changing variable $z=x-y$ and $\tau=t-s$, we have
\bma
&\quad L^{\alpha,\beta}(t,x;0,t;\lambda,\lambda,0,D)\nnm\\
&=\int_0^t\int_{\{|y|\leq\lambda s\}}(1+t-s)^{-\alpha}e^{-\frac{(|x-y|-\lambda(t-s))^2}{D(1+t-s)}}(1+s)^{-\beta}B_3(s,|y|)dyds\nnm\\
&=\int_0^t\int_{\{|x-z|\leq\lambda (t-\tau)\}}(1+\tau)^{-\alpha}e^{-\frac{(|z|-\lambda\tau)^2}{D(1+\tau)}}(1+t-\tau)^{-\beta}B_3(t-\tau,|x-z|)dzd\tau.
\ema
By the similar arguments as \eqref{rbepw6030}--\eqref{rbepw6039}, we can obtain \eqref{rbepw603}.

Finally, we prove \eqref{rbepw602}. For $|x|\leq\sqrt{1+t}$ and $||x|-\lambda t|\leq\sqrt{1+t}$, it holds that
\bma
&\quad L^{\alpha,\beta}(t,x;0,t;\lambda,\lambda,\lambda,D)\nnm\\
&\leq\int_0^t(1+t-s)^{-\alpha}(1+s)^{-\beta}\int_{\R^3}e^{-\frac{(|x-y|-\lambda(t-s))^2}{D(1+t-s)}}B_{2}(s,|y|-\lambda s)dyds\nnm\\
&\leq C\((1+t)^{-\alpha}\Gamma_{\beta-\frac{5}{2}}(t)+(1+t)^{-\beta}\Gamma_{\alpha-\frac{5}{2}}(t)\). \label{rbepw6021}
\ema

For $\sqrt{1+t}\leq|x|\leq\lambda t-\sqrt{1+t}$, we have
\bma
&\quad L^{\alpha,\beta}(t,x;0,t;\lambda,\lambda,D)\nnm\\
&=\(\int_0^{\frac{\lambda t-|x|}{4\lambda}}+\int^{t-\frac{\lambda t-|x|}{4\lambda}}_{\frac{\lambda t-|x|}{4\lambda}}+\int_{t-\frac{\lambda t-|x|}{4\lambda}}^t\)\int_{\{|y|\leq\lambda s\}}\frac{e^{-\frac{(|x-y|-\lambda(t-s))^2}{D(1+t-s)}}}{(1+t-s)^{\alpha}}\frac{B_{2}(s,|y|-\lambda s)}{(1+s)^{\beta}}dyds\nnm\\
&=:I_4+I_5+I_6.   \label{rbepw6022}
\ema

For $I_4$, we split $y$ into $|y|\leq\frac{\lambda t-|x|}{2}$ and $|y|\geq\frac{\lambda t-|x|}{2}$. If $|y|\geq\frac{\lambda t-|x|}{2}$ and $s\leq\frac{\lambda t-|x|}{4\lambda}$, then it follows that $|y|-\lambda s\geq\frac{\lambda t-|x|}{4}$. If $|y|\leq\frac{\lambda t-|x|}{2}$ and $s\leq\frac{\lambda t-|x|}{4\lambda}$, then it holds that $\lambda(t-s)-|x-y|\geq\lambda t-\lambda s-|x|-|y|\geq\frac{\lambda t-|x|}{4}$. Thus,
\bma
I_4
&\leq C(1+t)^{-\frac{5}{2}}e^{-\frac{(|x|-\lambda t)^2}{8D(1+t)}}\int_0^{\frac{\lambda t-|x|}{4\lambda}}(1+t-s)^{-\alpha+\frac{5}{2}}(1+s)^{-\beta+\frac{5}{2}}ds\nnm\\
&\quad+C(1+t)^{-2}B_{2}(t,|x|-\lambda t)\int_0^{\frac{\lambda t-|x|}{4\lambda}}(1+t-s)^{-\alpha+\frac{5}{2}}(1+s)^{-\beta+2}ds\nnm\\
&\leq C(1+t)^{-\alpha+\frac{1}{2}}\Gamma_{\beta-2}(t)B_{2}(t,|x|-\lambda t).\label{rbepw6023}
\ema

For $I_6$, we split $y$ into $|x-y|\leq\frac{\lambda t-|x|}{2}$ and $|x-y|\geq\frac{\lambda t-|x|}{2}$. If $|x-y|\leq\frac{\lambda t-|x|}{2}$ and $t-s\leq\frac{\lambda t-|x|}{4\lambda}$, then we have $\lambda s-|y|\geq\lambda t-|x|-|x-y|-\lambda(t-s)\geq\frac{\lambda t-|x|}{4}$. If $|x-y|\geq\frac{\lambda t-|x|}{2}$ and $t-s\leq\frac{\lambda t-|x|}{4\lambda}$, then it holds that $|x-y|-\lambda(t-s)\geq\frac{\lambda t-|x|}{4}$. Thus,
\bma
I_6
&\leq C(1+t)^{-2}B_{2}(t,|x|-\lambda t)\int_{t-\frac{\lambda t-|x|}{4\lambda}}^t(1+t-s)^{-\alpha+\frac{5}{2}}(1+s)^{-\beta+2}ds\nnm\\
&\quad+C(1+t)^{-\frac{5}{2}}e^{-\frac{(|x|-\lambda t)^2}{16D(1+t)}}\int_{t-\frac{\lambda t-|x|}{4\lambda}}^t(1+t-s)^{-\alpha+\frac{5}{2}}(1+s)^{-\beta+\frac{5}{2}}ds\nnm\\
&\leq C(1+t)^{-\beta}\Gamma_{\alpha-\frac{5}{2}}B_{2}(t,|x|-\lambda t).\label{rbepw6024}
\ema

For $I_5$, it holds that
\bma
I_5&\leq\(\int^{\frac{t}{2}}_{\frac{\lambda t-|x|}{4\lambda}}+\int_{\frac{t}{2}}^{t-\frac{\lambda t-|x|}{4\lambda}}\)\int_{\R^3}\frac{e^{-\frac{(|x-y|-\lambda(t-s))^2}{D(1+t-s)}}}{(1+t-s)^{\alpha}}\frac{B_{2}(s,|y|-\lambda s)}{(1+s)^{\beta}}dyds\nnm\\
&\leq C(1+t)^{-\alpha}\int^{\frac{t}{2}}_{\frac{\lambda t-|x|}{4\lambda}}(1+s)^{-\beta}\int_{\R^3}B_{2}(s,|y|-\lambda s)dyds\nnm\\
&\quad+C(1+t)^{-\beta}\int_{\frac{t}{2}}^{t-\frac{\lambda t-|x|}{4\lambda}}(1+t-s)^{-\alpha}\int_{\R^3}e^{-\frac{(|x-y|-\lambda(t-s))^2}{D(1+t-s)}}dyds\nnm\\
&\leq C(1+t)^{-\alpha+1}\(1+\lambda t-|x|\)^{-\beta+\frac{5}{2}}+C(1+t)^{-\beta+1}\(1+\lambda t-|x|\)^{-\alpha+\frac{5}{2}},\label{rbepw6025}
\ema
where we have used Lemma \ref{rbepw-Ap-2}.

For $|x|\geq\lambda t+\sqrt{1+t}$ and $|y|\leq\lambda s$, we have $|x-y|-\lambda(t-s)\geq|x|-\lambda t-(|y|-\lambda s)\geq|x|-\lambda t$. Thus,
\bma
&\quad L^{\alpha,\beta}(t,x;0,t;\lambda,\lambda,\lambda,D)\nnm\\
&\leq C(1+t)^{-\frac{5}{2}}e^{-\frac{(|x|-\lambda t)^2}{D(1+t)}}\int_0^t(1+t-s)^{-\alpha+\frac{5}{2}}(1+s)^{-\beta+\frac{5}{2}}ds\nnm\\
&\leq C\((1+t)^{-\alpha}\Gamma_{\beta-\frac{5}{2}}(t)+(1+t)^{-\beta}\Gamma_{\alpha-\frac{5}{2}}(t)\)e^{-\frac{3(|x|-\lambda t)^2}{2D_1(1+t)}},\label{rbepw6027}
\ema
where  $D_1\ge \frac{3D}{2}$. By combining \eqref{rbepw6021}--\eqref{rbepw6027}, we obtain \eqref{rbepw602}.
\end{proof}


\begin{proof}[\textbf{The proof of Lemma \ref{rbepw7}}]
First, we prove \eqref{rbepw701}.  We split $y$ into $|y|\leq\frac{7|x|}{8}$ and $|y|\geq\frac{7|x|}{8}$. Thus,
\bma
&\quad M^{\alpha}(t,x;0,t;0,D,D_1)\nnm\\
&=\int_0^t\(\int_{\{|y|\leq\frac{7|x|}{8}\}}+\int_{\{|y|\geq\frac{7|x|}{8}\}}\)e^{-\frac{|x-y|+t-s}{D}}(1+s)^{-\alpha}e^{-\frac{2|y|^2}{D_1(1+s)}}dyds\nnm\\
&\leq Ce^{-\frac{|x|}{16D}}\int_0^te^{-\frac{t-s}{D}}(1+s)^{-\alpha}\int_{\R^3}e^{-\frac{|x-y|}{2D}}dyds\nnm\\
&\quad+Ce^{-\frac{3|x|^2}{2D_1(1+t)}}\int_0^te^{-\frac{t-s}{D}}(1+s)^{-\alpha}\int_{\R^3}e^{-\frac{|x-y|}{D}}dyds\nnm\\
&\leq C(1+t)^{-\alpha}e^{-\frac{3|x|^2}{2D_1(1+t)}}+C(1+t)^{-\alpha}e^{-\frac{|x|}{16D}},\label{rbepw7012}
\ema
where we have used \eqref{rbepw204j1}, and $D_1\ge 48D\max\{1,\lambda\}$.  This proves \eqref{rbepw701}.

Then, we prove \eqref{rbepw801}. For $|x|\leq\lambda t$, we split $y$ into $|y|\leq\frac{|x|}{2}$ and $|y|\geq\frac{|x|}{2}$. Thus,
\bma
N^{\alpha}(t,x;0,t;\lambda,0,D)
&\leq C\(e^{-\frac{|x|}{4D}}+B_3(t,|x|)\)\int_0^te^{-\frac{t-s}{D}}(1+s)^{-\alpha}\int_{\R^3}e^{-\frac{|x-y|}{2D}}dy ds\nnm\\
&\leq C(1+t)^{-\alpha}B_{3}(t,|x|)+Ce^{-\frac{3(|x|+t)}{2D_1}},\label{rbepw8014-1}
\ema
where we have used \eqref{rbepw204j1}, and $D_1\ge 12D\max\{1,\lambda\}$.

For $|x|\geq\lambda t$ and $|y|\leq\lambda s$, we have $|x-y|\geq|x|-\lambda t$. Thus,
\bma
N^{\alpha}(t,x;0,t;\lambda,0,D)&\leq Ce^{-\frac{||x|-\lambda t|}{2D}}\int_0^te^{-\frac{t-s}{D}}(1+s)^{-\alpha}\int_{\R^3}e^{-\frac{|x-y|}{2D}}dy ds\nnm\\
&\leq C(1+t)^{-\alpha}e^{-\frac{3(|x|-\lambda t)^2}{2D_1(1+t)}}+Ce^{-\frac{3(|x|+t)}{2D_1}},\label{rbepw8014}
\ema
where we have used \eqref{rbepw204}, and $D_1\ge 6D\max\{1,\lambda\}$. By \eqref{rbepw8014} and \eqref{rbepw8014-1}, we obtain \eqref{rbepw801}.

Next, we prove \eqref{rbepw802}. For $|x|\leq\sqrt{1+t}$ and $||x|-\lambda t|\leq\sqrt{1+t}$, we have
\be
 M^{\alpha}(t,x;0,t;\lambda,D,D_1),\,N^{\alpha}(t,x;0,t;\lambda,\lambda,D)
\leq C(1+t)^{-\alpha}.\label{rbepw802-0}
\ee

Let
\be
Q^{\alpha} =\int_0^{t}\int_{\R^3}e^{-\frac{|x-y|+(t-s)}{D}}(1+s)^{-\alpha}B_{2}(s,|y|-\lambda s)dyds.\label{rbepw-mlm-0}
\ee
It holds that
\be
 M^{\alpha}(t,x;0,t;\lambda,D,D_1),\,N^{\alpha}(t,x;0,t;\lambda,\lambda,D)\leq Q^{\alpha} .\label{rbepw-mlm-1}
\ee

For $\sqrt{1+t}\leq|x|\leq\lambda t-\sqrt{1+t}$, since $ e^{-\frac{|x-y|+t-s}{2D}} \leq e^{-\frac{ ||x-y|-\lambda(t-s)|}{DE}} $ for $E=2\max\{1,\lambda\}$, it follows that
\bma
Q^{\alpha}&\leq\(\int_0^{\frac{\lambda t-|x|}{4\lambda}}+\int^{t-\frac{\lambda t-|x|}{4\lambda}}_{\frac{\lambda t-|x|}{4\lambda}}+\int_{t-\frac{\lambda t-|x|}{4\lambda}}^{t}\)\int_{\R^3}e^{-\frac{|x-y|+t-s}{2D}}\nnm\\
&\qquad\qquad\times(1+s)^{-\alpha}e^{-\frac{||x-y|-\lambda(t-s)|}{DE}}B_{2}(s,|y|-\lambda s)dyds\nnm\\
&=:I_1+I_2+I_3.\label{rbepw7021}
\ema

For $I_1$, we split $y$ into $|y|\leq\frac{\lambda t-|x|}{2}$ and $|y|\geq\frac{\lambda t-|x|}{2}$. If $|y|\geq\frac{\lambda t-|x|}{2}$ and $s\leq\frac{\lambda t-|x|}{4\lambda}$, then we have $|y|-\lambda s\geq\frac{\lambda t-|x|}{4}$. If $|y|\leq\frac{\lambda t-|x|}{2}$ and $s\leq\frac{\lambda t-|x|}{4\lambda}$, then we have $\lambda(t-s)-|x-y|\geq\lambda t-\lambda s-|x|-|y|\geq\frac{\lambda t-|x|}{4}$. Thus,
\bma
I_1&\leq C\(e^{-\frac{|\lambda t-|x||}{4DE}}+B_{2}(t,|x|-\lambda t)\)\int_0^{\frac{\lambda t-|x|}{4\lambda}}e^{-\frac{ t-s}{2D}}(1+s)^{-\alpha} ds\nnm\\
&\leq C(1+t)^{-\alpha}B_{2}(t,|x|-\lambda t).
\ema

For $I_3$, we split $y$ into $|x-y|\leq\frac{\lambda t-|x|}{2}$ and $|x-y|\geq\frac{\lambda t-|x|}{2}$. Since $|x-y|\leq\frac{\lambda t-|x|}{2}$ and $t-s\leq\frac{\lambda t-|x|}{4\lambda}$, then we have $\lambda s-|y|\geq\lambda t-|x|-|x-y|-\lambda(t-s)\geq\frac{\lambda t-|x|}{4}$. Since $|x-y|\geq\frac{\lambda t-|x|}{2}$ and $t-s\leq\frac{\lambda t-|x|}{4\lambda}$, then we have $|x-y|-\lambda(t-s)\geq\frac{\lambda t-|x|}{4}$. Thus,
\bma
I_3&\leq C\(B_{2}(t,|x|-\lambda t)+e^{-\frac{|\lambda t-|x||}{4DE}}\)\int_{t-\frac{\lambda t-|x|}{4\lambda}}^{t}e^{-\frac{ t-s}{2D}}(1+s)^{-\alpha} ds\nnm\\
&\leq C(1+t)^{-\alpha}B_{2}(t,|x|-\lambda t).
\ema

For $I_2$, it holds that
\bma
I_2&\leq \(\int^{\frac{t}{2}}_{\frac{\lambda t-|x|}{4\lambda}}+\int^{t-\frac{\lambda t-|x|}{4\lambda}}_{\frac{t}{2}}\)e^{-\frac{ t-s}{2D}}(1+s)^{-\alpha} ds\nnm\\
&\leq C\max\{1,(1+t)^{4-\alpha}\}te^{-\frac{t}{4D}}\(1+\frac{\lambda t-|x|}{4\lambda}\)^{-4}+C(1+t)^{-\alpha+1}e^{-\frac{(\lambda t-|x|)}{8\lambda D}}\nnm\\
&\leq C(1+t)^{-\alpha}B_{2}(t,|x|-\lambda t).
\ema

For $|x|\geq\lambda t$ and $|y|\leq\lambda s$, we have $|x-y|\geq|x|-\lambda t$. Thus, it holds that
\bma
 N^{\alpha}(t,x;0,t;\lambda,\lambda,D)
&\leq Ce^{-\frac{||x|-\lambda t|}{2D}}\int_0^te^{-\frac{t-s}{D}}(1+s)^{-\alpha}\int_{\R^3}e^{-\frac{|x-y|}{2D}}dyds\nnm\\
&\leq C(1+t)^{-\alpha}e^{-\frac{3(|x|-\lambda t)^2}{2D_1(1+t)}}+Ce^{-\frac{3(|x|+t)}{2D_1}},\label{rbepw8016}
\ema
where we have used \eqref{rbepw204} and $D_1\ge 6D\max\{1,\lambda\}$.

For $|x|\geq\lambda t$, we split $y$ into $||y|-\lambda s|\leq\frac{7(|x|-\lambda t)}{8}$ and $||y|-\lambda s|\geq\frac{7(|x|-\lambda t)}{8}$. Thus,
\bma
&\quad M^{\alpha}(t,x;0,t;\lambda,D,D_1)\nnm\\
&\leq C\(e^{-\frac{|x|-\lambda t}{8DE}}+e^{-\frac{3(|x|-\lambda t)^2}{2D_1(1+t)}}\)\int_0^te^{-\frac{t-s}{2D}}(1+s)^{-\alpha}\int_{\R^3}e^{-\frac{|x-y|}{2D}}dyds\nnm\\
&\leq C(1+t)^{-\alpha}e^{-\frac{3(|x|-\lambda t)^2}{2D_1(1+t)}}+Ce^{-\frac{3(|x|+t)}{2D_1}},\label{rbepw7023}
\ema
where we have used \eqref{rbepw204}, and $D_1\ge 24DE\max\{1,\lambda\}$. By combining \eqref{rbepw-mlm-0}--\eqref{rbepw7023}, we can obtain \eqref{rbepw802}.
\end{proof}

\medskip
\noindent {\bf Acknowledgements:} The research of this work was supported  by  the National Natural Science Foundation of China  grants (No.  12171104),  Guangxi Natural Science Foundation (No. 2019JJG110010).


\end{document}